\documentclass[a4paper,12pt]{article}
\usepackage{graphicx} 

\usepackage{eufrak}
\usepackage{amssymb}
\usepackage{geometry}
\usepackage{amsmath}
\usepackage{amsthm}
\usepackage{graphicx}
\usepackage{framed}
\usepackage{xcolor}
\usepackage{hyperref}
\usepackage{mathtools}
\usepackage{enumitem}
\usepackage{indentfirst}
\usepackage{tocloft}
\usepackage{soul}
\definecolor{darkblue}{rgb}{0,0,0.3}
\definecolor{urlblue}{rgb}{0,0,0.7}
\hypersetup{
	colorlinks=true,
	citecolor=blue,
	linkcolor=darkblue,
	urlcolor=urlblue,
}

\newcommand{\TODO}[1]{\message{TODO Warning: {#1}}}
\newcommand{\updatetag}[2]{}

\newtheorem{theorem}{Theorem}[section]
\newtheorem{lemma}[theorem]{Lemma}
\newtheorem{cor}[theorem]{Corollary}
\newtheorem{prop}[theorem]{Proposition}

\theoremstyle{definition}
\newtheorem{rmk}[theorem]{Remark}

\newtheorem{condition}[theorem]{Condition}

\newtheorem{setup}[theorem]{General Setup}
\numberwithin{equation}{section}

\newcommand{\RR}{\mathbb{R}}
\DeclareMathOperator{\tr}{tr}
\DeclareMathOperator{\Ric}{Ric}
\DeclareMathOperator{\diam}{diam}
\DeclareMathOperator{\Area}{Area}

\DeclareMathOperator{\Vol}{Vol}
\DeclareMathOperator{\GH}{GH}

\DeclareMathOperator{\IF}{IF}

\renewcommand{\bar}{\overline}
\renewcommand{\tilde}{\widetilde}
\renewcommand{\epsilon}{\varepsilon}
\renewcommand{\leq}{\leqslant}
\renewcommand{\geq}{\geqslant}

\newcommand{\D}{\nabla}
\newcommand{\p}{\partial}
\renewcommand{\th}{\theta}
\newcommand{\metric}[2]{\langle{#1},{#2}\rangle}
\renewcommand{\L}{\mathcal{L}}

\renewcommand{\H}{\mathcal{H}}
\newcommand{\V}{\mathcal{V}}
\newcommand{\x}{\boldsymbol{x}}

\DeclareMathOperator{\Ch}{Ch}
\DeclareMathOperator{\length}{length}
\DeclareMathOperator{\pt}{pt}

\title{Codimension 2 drawstrings with scalar curvature lower bounds}
\author{Demetre Kazaras and Kai Xu}
\date{}

\begin{document}
\maketitle
\begin{abstract}
    We produce new examples of Riemannian manifolds with scalar curvature lower bounds and collapsing behavior along codimension 2 submanifolds. Applications of this construction are given, primarily on questions concerning the stability of scalar curvature rigidity phenomena, such as Llarull's Theorem and the Positive Mass Theorem.
\end{abstract}

\tableofcontents

\section{Introduction}\label{sec:intro}

\subsection{Background and main theorem}

The Geroch conjecture states that the only Riemannian metrics $g$ on the $3$-torus with nonnegative scalar curvature $R_g\geq0$ are flat. This fact was first established by Schoen-Yau and Gromov-Lawson \cite{SY, GL}, and represents a cornerstone in the study of scalar curvature. The present work is partially motivated by the associated {\emph{stability problem}}: if a suitably normalized $(T^3,g)$ satisfies $R_g\geq -\epsilon$, in what sense is $g$ close to a flat metric? One may probe the subtlety of this question using the Gromov-Lawson and Schoen-Yau tunnel construction \cite{GL,SY}. Indeed, there exist Riemannian $3$-tori of almost non-negative scalar curvature with regions of cylindrical and spherical geometry (respectively termed {\emph{gravity wells}} and {\emph{bags of gold}}), showing that Gromov-Hausdorff closeness is not generally expected. It was later conjectured by Sormani that a uniform ``minA'' lower bound (see \eqref{eq-intro:minA} below) would prohibit small Gromov-Lawson tunnels and that one could hope for closeness using the Sormani-Wenger Intrinsic Flat metric \cite{SorWen}. This is an entry point to a wide array of stability problems related to scalar curvature lower bounds, e.g. the Positive Mass Theorem and Llarull's theorem -- we refer the reader to the survey article of Sormani \cite{SormaniConj} for a more detailed account of these considerations.


In the authors' previous work \cite{KXdrawstring}, a new class of examples called \textit{drawstrings} were discovered, demonstrating complexity in the Geroch conjecture stability problem beyond those observed from the tunnel construction. These examples were partially inspired by the influential ideas of Lee-Naber-Neumayer \cite{LNN}. By \textit{creating a drawstring} (with parameter $\epsilon$) along an embedded submanifold $\Sigma$ in a Riemannian manifold, we mean to modify the metric in an $\epsilon$-neighborhood of $\Sigma$, such that the scalar curvature is decreased no more than $\epsilon$, and the diameter of $\Sigma$ becomes at most $\epsilon$. In \cite{KXdrawstring} it is shown that one can create a drawstring around a closed geodesic $\Sigma$ in a flat 3-torus, with any $\epsilon>0$. When $\epsilon\to0$, the resulting spaces converge to a metric quotient obtained by identifying $\Sigma$ to a point. This sequence of metrics represents the first counterexample to the minA-Intrinsic Flat stability conjecture, and a modification of this construction was used to study the volume entropy of hyperbolic 3-manifolds \cite{KSXentropy}.

The main theorem of this paper, stated as follows, produces drawstrings along any oriented codimension-2 submanifold. All the manifolds and metrics in this paper are assumed to be smooth. Below and throughout, $N_g(A,r)$ denotes the $r$-distance neighborhood of a set $A$ in a Riemannian manifold $(M,g)$.

\begin{theorem}\label{thm-intro:main}
    Let $n\geq3$ and $\Sigma^{n-2}$ be a closed oriented embedded codimension-2 submanifold of an oriented Riemannian $n$-manifold $(M^n,g)$. Given any $\epsilon>0$ and function $v_0\in C^\infty(\Sigma)$ with $v_0\leq0$, there is a radius $r_1<\epsilon$ and another metric $g'$ on $M$ satisfying the following:
    \begin{enumerate}[label=(\roman*), topsep=2pt, itemsep=-0.1ex]
        \item $g'=g$ outside $N_g(\Sigma,r_1)$,
        \item $g'|_\Sigma=e^{2v_0}g|_\Sigma$,
        \item $R_{g'}\geq R_g-\epsilon$ everywhere,
        \item for all $x\in\p N_g(\Sigma,r_1)$ we have $d_{g'}(x,\Sigma)\leq 3r_1$.
    \end{enumerate}
\end{theorem}


\begin{rmk}
\phantom{.}
    \begin{enumerate}[label=(\arabic*), nosep]
        \item A detailed version of Theorem \ref{thm-intro:main} is given in Section \ref{sec:main} (in particular, Theorems \ref{thm-main:scal_estimate} and \ref{thm-main:h_and_u}) with more refined statements about the metric $g'$.
        \item In Theorem \ref{thm-intro:main} (ii), note that we can prescribe the conformal factor $v_0$ on $\Sigma$. Choosing a sequence of constant functions $v_0\to-\infty$, the resulting induced metrics $g'|_\Sigma$ degenerate and the construction recovers the original notion of drawstrings mentioned above.
    \end{enumerate}
\end{rmk}

To place Theorem \ref{thm-intro:main} in context, we recall the work of Lee-Naber-Neumayer \cite[Section 9]{LNN} where drawstring-like examples were constructed to collapse closed geodesics in flat $n$-tori, $n\geq4$. The construction of \cite{LNN} made crucial use of the codimension at least $3$ nature of curves in dimensions $n\geq 4$, where the essential geometric observation boils down to the fact that 3-dimensional acute cones have scalar curvature of the order $+O(1/r^2)$ as one apporaches the cone tip $r=0$. In the present work, the normal bundle of $\Sigma$ has dimension 2 -- the situation becomes more subtle since 2-dimensional cones are flat. The idea behind Theorem \ref{thm-intro:main} is that by smoothing an acute $2$-dimensional cone, one can create large positive curvature at the tip region. The main difficulty is to confirm the existence of a cone smoothing so that the resulting curvature is large enough to compensate the negative curvature introduced by imposing the requisite metric degeneration $g'|_{\Sigma}=e^{2v_0}g|_\Sigma$.

\vspace{3pt}

This paper is organized as follows. The remainder of this section gives an outline of the proof of Theorem \ref{thm-intro:main} (Subsection \ref{subsec:sketch}), followed by an overview of our applications (Subsection \ref{subsec:app}).
Section \ref{sec:main} contains the precise description of the objects and properties summerized in Theorem \ref{thm-intro:main}, the proof of which is carried out in Sections \ref{sec:scal} and \ref{sec:cutoff}. Finally, Section \ref{sec:app} is devouted to the applications of the main result.

\subsection{The structure of drawstring metrics}\label{subsec:sketch}

The construction in Theorem \ref{thm-intro:main} is rather involved and so we begin by considering the geometrically simple setting of \cite{KXdrawstring}, where $g'$ takes the form of a warped product on $T^2\times S^1$.
The metric $g'$ is flat outside a small neighborhood of a vertical fiber, and inside that small neighborhood, the metric takes the following form in a cylindrical coordinate chart:
\begin{equation}\label{eq-intro:warped_prod}
    g'=e^{-2u(r)}\big(dr^2+f(r)^2d\th^2\big)+e^{2u(r)}dt^2.
\end{equation}
The function $u$ is designed to contract the central fiber and thus we suppose $u(0)\ll-1$. Meanwhile, $f$ is chosen to provide sufficiently large positive curvature so that the geometry of $dr^2+f^2d\th^2$ resembles a smoothened cone tip. The following prototype functions turn out to serve these purposes:
\begin{equation}\label{eq-intro:KX23_main_func}
    f(r)=r\big(1-\frac{c_1}{\log(1/r)}\big),\qquad u(r)=-c_2\log\log\frac1r.
\end{equation}
Indeed, one has $\lim_{r\to0}u(r)=-\infty$, and with some calculation,
\begin{equation}\label{eq-intro:scal_1}
    R_g=\frac2{r^2(\log 1/r)^{2+2c_2}}\Big[\frac{c_1(c_1+2)}{\log(1/r)-c_1}+c_1-c_2^2\Big]>0\quad\text{if $r\ll1$ and $c_2^2\leq c_1$}.
\end{equation}
Eventually, the examples in \cite{KXdrawstring} are built from \eqref{eq-intro:KX23_main_func} by smoothing the (mild) singularity at $r=0$ and gluing to a flat metric along an outer cylinder of constant $r$-value. 
We note that negative scalar curvature is only introduced in the outer gluing region.

\begin{rmk}
    It is natural to ask whether the functions in \eqref{eq-intro:KX23_main_func} appear in any natural geometric context. In Section \ref{sec:confinv}, we consider a model situation and provide a heuristic derivation of \eqref{eq-intro:KX23_main_func} using notion we call {\emph{conformal inversion}}.
\end{rmk}

To prove Theorem \ref{thm-intro:main}, we generalize \eqref{eq-intro:warped_prod} in the following way, which is formally presented in Section \ref{sec:main}. In a small neighborhood of $\Sigma$, we introduce unit radial and angular 1-forms $dr$, $\omega_\th$,
and the horizontal distribution $\H=\operatorname{span}(dr,\omega_\th)^\perp$. The original metric can be written as
\begin{equation}\label{eq-intro:old_metric}
    g=dr^2+\omega_\th^2+g_{\H}.
\end{equation}
Then we design the drawstring metric to take the form
\begin{equation}
    g'=e^{-2(n-2)u}\big(dr^2+h(r)^2\omega_\th^2\big)+e^{2u}g_{\H}.
\end{equation}
Here $h$ and $u$ are carefully chosen functions defined near $\Sigma$.\footnote{The function $h$ here is related to the previous function $f$ in \eqref{eq-intro:KX23_main_func} by $h=f/r$, since $d\th$ has length $r+o(r)$ while $\omega_\th$ has unit length.} 

Theorem \ref{thm-main:scal_estimate} is our main technical statement, which estimates the scalar curvature of $g'$. Upon choosing $h$ and $u$ in a manner inspired by \eqref{eq-intro:KX23_main_func}, we show that $R_{g'}$ is bounded from below by an expression similar to \eqref{eq-intro:scal_1}, plus a few error terms due to the curvature of $M$ and $\Sigma$, the non-flatness of the normal connection, and the non-constancy of $v_0$. Eventually, we show that these error terms have lower orders compared to the positive main term. The computation of $R_{g'}$ is carried out in Section \ref{sec:scal}. The distribution $\H$ is non-integrable in dimensions $n\geq4$ and so we apply the method of moving frames in our computation. The use of this approach is inspired by \cite{Zhou_2023}. We refer the reader to Subsection \ref{subsec:setup} for more details of the strategy. Finally, in Section \ref{sec:cutoff}, we construct the suitable functions $h$ and $u$ as modifications of the prototype functions in \eqref{eq-intro:KX23_main_func}.

\subsection{Applications}\label{subsec:app}

We are ready to discuss the applications of Theorem \ref{thm-intro:main}. As our primary motivation comes from scalar curvature stability problems, many of the following results construct sequences of manifolds with controlled scalar curvature and new pathological limits. On the other hand, we discuss positive stability results in the literature and demonstrate aspects of their sharpness with our examples. The results stated below, namely Theorems \ref{thm-intro:collapse_submfd}, \ref{thm-intro:partial_collapse}, \ref{thm-intro:leetopping}, \ref{thm-intro:AFexample}, \ref{thm-intro:larrull} below, are all proved in Section \ref{sec:app}. The contents of Sections \ref{sec:scal} and \ref{sec:cutoff} are not needed to access these applications.

\subsubsection{Collapsing under scalar curvature lower bounds.}

Fix an oriented manifold $(M^n,g)$ and $\Sigma^{n-2}\subset M^n$ as in Theorem \ref{thm-intro:main}. One may produce sequences of metric $g_i$ on $M$ by making increasingly extreme choices for the parameters in the drawstring construction. The Gromov-Hausdorff and Intrinsic-Flat limits exhibit interesting examples of limit spaces in the boundary of the space of metrics with scalar curvature lower bounds. 

To give a precise description of the examples, we require a few notions. Fix a compact subset $K\subset M^n$ and a constant $c>0$. Define the {\emph{$c$-partially pulled metric}} by
\begin{equation}\label{eq-intro:def_dc}
    \begin{aligned}
        d_c(x,y) &= \inf\Big\{d_M(x,p_1)+e^{-c}d_K(p_1,q_1)+d_M(q_1,p_2)+e^{-c}d_K(p_2,q_2)+\cdots \\
        &\hspace{60pt}+e^{-c}d_K(p_N,q_N)+d_M(q_N,y): p_i,q_i\in K, 1\leq i\leq N, N\in\mathbb{N}\Big\},
    \end{aligned}
\end{equation}
where $d_M$ is the $g$-Riemannian distance on $M$ and $d_K$ is the length distance on $K$ induced by $d_M$. If $K$ is an embedded submanifold, then $d_K$ is simply the Riemannian distance on $K$. We call $(M^n,d_c)$ the {\emph{space resulting from $c$-partially pulling $\Sigma$ in $(M^n,g)$}}. If $c=\infty$, then $d_\infty$ descends to the quotient $M^n/\{K\sim\pt\}$ and the resulting metric space is called a {\emph{pulled string space}} (see Appendix \ref{sec:scrunch} for more discussion). Pulled string limit spaces were first observed as limits of manifolds with controlled scalar curvature by Basilio-Dodziuk-Sormani \cite{BDS} and Basilio-Sormani \cite{BS}. Their methods are based on tunnel constructions and produce sequences with unbounded topology, whereas the following sequence of metrics is supported on a common manifold. 

\begin{theorem}[collapsing a submanifold]\label{thm-intro:collapse_submfd}
    Let $(M^n,g)$ be a smooth oriented manifold, and $K\subset M^n$ be a compact connected submanifold (possibly with boundary). There exists a sequence of smooth metrics $g_i$ on $M^n$, such that $R_{g_i}\geq R_g-1/i$ and $(M^n,g_i)\to (M^n,g)/\{K\sim\text{pt}\}$ in the Gromov-Hausdorff sense as $i\to\infty$. If additionally $n=3$, then the convergence is also in the Intrinsic-Flat sense.
\end{theorem}

Drawstrings also produce examples of {\emph{partial collapsing}}. Below, we refer to the {\emph{$\mathrm{minA}$ invariant}}, defined by
\begin{equation}\label{eq-intro:minA}
    \mathrm{minA}(M^n,g)=\inf\Big\{|\mathcal{S}|: \mathcal{S}\subset M^n\text{ is an embedded minimal hypersurface}\Big\}.
\end{equation}

\begin{theorem}[partial collapsing]\label{thm-intro:partial_collapse}
    Let $(M^n,g)$ be a closed oriented manifold. Fix an oriented closed codimension 2 submanifold $\Sigma^{n-2}\subset M$ and a constant $c\in(0,+\infty]$. Then there is a sequence $(M,g_i)$ and a number $A_0>0$ such that $R_{g_i}\geq R_g-1/i$, $minA(M,g_i)\geq A_0$, and $(M,g_i)$ converges to the space resulting from $c$-partially pulling $\Sigma$ in $(M,g)$ in the Gromov-Hausdorff sense as $i\to0$. If additionally $n=3$ and $c=+\infty$, then the convergence is also in the Intrinsic-Flat sense.
\end{theorem}

We note that the proof of Theorem \ref{thm-intro:partial_collapse} shows that the distance functions $d_i$ converge to the partially pulled metric $d_c$ in the $C^0$-topology. The authors suspect that the sequence in Theorem \ref{thm-intro:partial_collapse} also converges to the $c$-partially pulled string space in the Intrinsic-Flat sense for finite $c$. However, to obtain the Intrisic-Flat convergence, we make essential use of the Basilio-Sormani Scrunching Theorem \cite{BS}, which requires $c=+\infty$.

\subsubsection{Pointwise limit of distance functions.}

What are the topologies on the space of Riemannian manifolds so that scalar curvature lower bounds are preserved under limits $(M,g_i)\to(M,g_\infty)$? One might na{\"i}vely guess that $C^2$ convergence is necessary for such preservation. However, a remarkable theorem of Gromov \cite{Billiards} (significantly generalized by Bamler \cite{Bamler}) states that if non-negative scalar curvature metrics $g_i$ converge to a $C^2$ metric $g_\infty$ in the $C^0$-topology, then the limit necessarily has non-negative scalar curvature. On the other hand, Lee-Topping \cite{Lee-Topping} produced examples on $S^n$ ($n\geq4$) with $R>0$ and converging to an arbitrary metric in the round conformal class. This shows that $C^0$-convergence of metric tensors cannot be weakened to pointwise convergence of distance functions, see also \cite[Section 6.8]{Gromov4lectures}. Using Theorem \ref{thm-intro:main}, we extend the Lee-Topping phenomenon to dimension 3, with arbitrary background metric and any scalar curvature lower bound.

\begin{theorem}\label{thm-intro:leetopping}
    Let $(M^3,g)$ be a closed manifold with $R_g>\lambda$, and $f<0$ be a smooth function. Then there is a sequence of smooth metrics $g_i$ such that $R_{g_i}>\lambda$ and $d_{g_i}\to d_{e^{2f}g}$ uniformly. Moreover, there is a constant $C$ independent of $i$ such that $C^{-1}g\leq g_i\leq Cg$.
\end{theorem}

In particular, if a conformal manifold $(M^3,[g'])$ is Yamabe positive, then there exists metrics $g_i$ with positive scalar curvature $R>0$ so that $d_{g_i}\to d_{g'}$ uniformly. Since $g'$ itself can be chosen to have negative scalar curvature on large portions of $M^3$, we see that positive scalar curvature is generally not preserved under this convergence.

\subsubsection{Stability of the Positive Mass Theorem.}

The Positive Mass Theorem of Schoen-Yau \cite{SYPMT} and Witten \cite{Witten} states that flat Euclidean space is the only complete asymptotically flat $3$-manifold with non-negative scalar curvature and non-positive ADM mass. The reader is refered to Lee's book \cite{geometric} for detailed statements and proofs on this matter. We now turn to the associated stability question, where one considers asymptotically flat $3$-manifolds of non-negative scalar curvature and small ADM mass, seeking a sense in which they are close to flat $\mathbb{R}^3$. Questions of this type were first raised by Huisken-Ilmanen \cite{HI} and Lee-Sormani \cite{LeeSormanispherical}. The next statement produces small-mass manifolds with drawstring regions. Below, the {\emph{Schwarzschild metric of mass $m$}} refers to the constant time-slice of the corresponding Schwarzschild spacetime.

\begin{theorem}\label{thm-intro:AFexample}
    Let $\sigma=\{z=0, x^2+y^2=1\}\subset \mathbb{R}^3$ denote a unit circle, $B_2\subset \mathbb{R}^3$ the ball of radius 2 centered at 0, and $g_0$ the standard Euclidean metric. Given sufficiently small $\varepsilon>0$, there exists an asymptotically flat metric $g$ on $\mathbb{R}^3$ such that:
    \begin{enumerate}[label={(\roman*)}, topsep=3pt, itemsep=0ex]
        \item $R_g\geq0$,
        \item $(\mathbb{R}^3,g)$ possesses no closed embedded minimal surfaces,
        \item Outside $B_2$, $g$ is isometric to the Schwarzschild metric of mass $\epsilon$,
        \item $||g-g_0||_{C^2(\mathbb{R}^3\setminus N_{g_0}(\sigma,\varepsilon))}<\varepsilon$,
        \item the $g$-distance between any two points $p,q\in N_{g_0}(\sigma,\varepsilon)$ is at most $10\varepsilon$.
    \end{enumerate}
\end{theorem}

For small $\varepsilon$, the loop $\sigma$ in Theorem \ref{thm-intro:AFexample} acts as a shortcut -- points on opposite sides of $\sigma$ are close with respect to $g$, even though they are distance 2 with respect to $g_0$. Regarding the above examples as initial data for Einstein's equations, property (ii) implies that $(\mathbb{R}^3,g)$ contain no apparent horizons. This suggests that $\sigma$ exhibits properties of a worm-hole without necessitating a black-hole region. A rigorous statement of this sort, however, would require procuding a suitable spacetime into which $(\mathbb{R}^3,g)$ naturally embeds.

Taking $\varepsilon\to0$ in the theorem, we obtain a sequence of manifolds which probes the stability of the positive mass theorem.

\begin{cor}\label{cor-intro:pmt}
    There exists a sequence of asymptotically flat metrics $\{g_i\}_{i=1}^\infty$ on $\RR^3$ 
    which have nonnegative scalar curvature, no closed minimal surfaces, with masses $m_{\text{ADM}}(M_i,g_i)\to0$, and such that $\{(\mathbb{R}^3,g_i,\vec{0})\}_{i=1}^\infty$ converge in the pointed Gromov-Hausdorff and Intrinsic Flat senses to a pulled string space.
\end{cor}

\begin{rmk}
    Corollary \ref{cor-intro:pmt} represents a counterexample to conjectures \cite[Conjecture 6.2]{LeeSormanispherical} and \cite[Conjecture 10.1]{SormaniConj} concerning the Intrinsic Flat stability of the positive mass theorem. It is worth pointing out that if one additionally assumes that the sequence $g_i$ satisfies a uniform Ricci curvature lower bound, pulled string limit spaces are not possible, see \cite{KKL} and \cite{DongStability}.
\end{rmk}

Dong and Song \cite{Dong-Song_2023} recently proved that, given an asymptotically flat manifold $(M,g)$ with $R\geq0$ and small mass, one can remove a ``bad set'' $Z\subset M$ with small boundary area, such that the induced length metric on $M\setminus Z$ is Gromov-Hausdorff close to flat $\RR^3$ via a harmonic map $\mathcal{U}:M\setminus Z\to\RR^3$. The examples of Corollary \ref{cor-intro:pmt} show that it is necessary to consider the \textit{induced length metric} on $M\setminus Z$ (rather than the restriction of $d_M$ to $M\setminus Z$) in their statement.

When the asymptotically flat manifold $(M,g)$ contains Gromov-Lawson tunnels or drawstrings, the image of $Z$ in $\mathbb{R}^3$ (under $\mathcal{U}$) resembles a neighborhood of a set of points or loops. The following result shows that the bad set has ``codimension larger than 1'' in the coarse sense stated in item (3) below. It is a natural open question: does the image of $\p Z$ coarsely resembles tubular neighborhoods of codimension 2 or 3 sets? An appropriate answer may help rule out potential geometric phenomena in scalar curvature almost ridity problems. 

\begin{theorem} {\emph{(Corollary to \cite[Theorem 1.3]{Dong-Song_2023})}}\label{thm-intro:DSplane}
    Suppose $\{(M_i,g_i)\}_{i=1}^\infty$ are complete asymptotically flat $3$-manifolds, with $R_{g_i}\geq0$ and ADM masses $m_i\to0$. Then there exist open sets $Z_i\subset M_i$ and maps $\mathcal{U}_i:M_i\setminus Z_i\to\mathbb{R}^3$ so that
    \begin{enumerate}[topsep=3pt, itemsep=-3pt]
        \item the boundary areas satisfy $|\partial Z_i|\to0$,
        \item $\mathcal{U}_i$ is a pointed measured $\epsilon_i$-Gromov-Hausdorff approximation with $\epsilon_i\to0$,
        \item for any 2-plane $P\subset \mathbb{R}^3$ we have $\big|P\cap\mathcal{U}_i( Z_i)\big|\to0$.
    \end{enumerate}
\end{theorem}

\noindent{}Items (1) and (2) of Theorem \ref{thm-intro:DSplane} are the main content of \cite[Theorem 1.3]{Dong-Song_2023}, and item (3) follows from a brief argument presented in Section \ref{sec:AFexamples}.

\subsubsection{Area instability of Llarull's Theorem.} 

Llarull \cite{Llarull} showed that the unit sphere $(S^n,g_0)$ satisfies a certain area-extremality property in the class of all manifolds with scalar curvature at least $n(n-1)$. In particular, a special case of Llarull's Theorem states that any metric $g$ on $S^n$ satisfying the following conditions must be isometric to $g_0$:
\begin{enumerate}[label={(\arabic*)}, nosep]
    \item $R_g\geq n(n-1)$,
    \item $|T|_g\geq|T|_{g_0}$ for all 2-forms $T$, i.e. the identity map $(S^n,g)\to(S^n,g_0)$ is area non-increasing.
\end{enumerate}
The stability problem for Llarull's Theorem was investigated by Allen-Bryden-Kazaras in dimension 3 \cite{AllenBrydenK} and later in all dimensions by Hirsch-Zhang in \cite{HirschZhang}. In these works, it is shown that if a sequence of metrics $g_i$ on $S^n$ satisfies 
\[
    R_{g_i}\geq n(n-1)-1/i,\qquad\qquad |X|_{g_i}\geq|X|_{g_0}\text{ for all vector fields }X,
\]
and an additional ``no-bubble condition," then $(S^n,g_i)$ converges to the unit sphere in the intrinsic flat sense.  The no-bubble condition imposed in \cite{AllenBrydenK} is a uniform (in $i$) lower bound on the Cheeger constants
\[
    \mathrm{Ch}(S^n,g_i)=\inf\Big\{\frac{|\partial E|_{g_i}}{\min\big\{|E|_{g_i},|S^n\setminus E|_{g_i}\big\}}:E\subset S^n\Big\}
\]
and \cite[Theorem 1.6]{HirschZhang} imposed the weaker assumption of a uniform lower bound on the Poincar{\'e} constants of $g_i$. 

It was an open question whether or not similar stability statements held under the weaker assumption $|T|_{g_i}\geq|T|_{g_0}$ for all 2-forms $T$. The following provides a negative answer to this question.

\begin{theorem}\label{thm-intro:larrull}
    There exists Riemannian $3$-spheres $(S^3,g_i)$ and a constant $C>0$ which satisfy the following for all $i=1,2,3\dots$ 
    \begin{enumerate}[label={(\roman*)}, topsep=3pt, itemsep=0ex]
        \item $|T|_{g_i}\geq |T|_{g_0}$ for all 2-forms $T$,
        \item $R_{g_i}\geq6-1/i$,
        \item $\mathrm{Ch}(S^3,g_i)\geq C$,
    \end{enumerate}
    and which converge to a pulled string space in the intrinsic flat and Gromov-Hausdorff senses as $i\to\infty$.
\end{theorem}

In general dimensions Hirsch-Zhang also showed \cite[Theorem 1.3]{HirschZhang} that under conditions (i) and (ii) of Theorem \ref{thm-intro:larrull}, there are sets $Z_i\subset S^n$ of small volume such that $g_i$ $C^0$-converges to the unit sphere metric on the ``good" region $S^n\setminus Z_i$. The examples in Theorem \ref{thm-intro:larrull} demonstrate that the removal of $Z_i$ is necessary.

\vspace{12pt}
\noindent \textbf{Acknowledgements.} We would like to thank Man-Chun Lee for helpful discussions on the topic of this paper, Antoine Song for consultation on Theorem \ref{thm-intro:DSplane}, and Christina Sormani for her interest in Section \ref{sec:confinv}. 


\section{The main construction and statements}\label{sec:main}

We fix the following geometrical setup for this and the following two sections. Let $(M^n,g)$ be an oriented Riemannian manifold, and $\Sigma^{n-2}$ be a closed connected embedded oriented submanifold. We write $\mathcal{N}\Sigma$ for the normal bundle of $\Sigma$ in $M$. Given $r>0$, we let $N(\Sigma,r)=N_g(\Sigma,r)$ be the $r$-neighborhood of $\Sigma$, and set $\Sigma_r:=\p N(\Sigma,r)$. Let $r_I<1/100$ be a sufficiently small radius so that: (i) the normal exponential map
\begin{equation}
    \exp^\perp: \big\{(x,v)\in \mathcal{N}\Sigma: |v|<2r_I\big\}\to N(\Sigma,2r_I)
\end{equation}
is a diffeomorphism onto its image and (ii) for all $r\leq 2r_I$, the surface $\Sigma_r$ has outward mean curvature at least $1/2r$ and area at most $4\pi|\Sigma|$. Denote the radial function $r(x)=d(x,\Sigma)$, which is smooth in $N(\Sigma,r_I)$. In our work below, the modification of $g$ is only performed within the region $N(\Sigma, r_I)$.

There is an induced fiber bundle structure $\pi: N(\Sigma,r_I)\to\Sigma$ given by $\pi(\exp^\perp(x,v))=x$. Denote by $\V$ the vertical distribution and $\H=\V^\perp$ the horizontal distribution, which we note is not generally integrable when $n\geq4$. Both distributions are smooth, and $\H$ coincides with $\Sigma$'s tangent bundle on $\Sigma$. On $N(\Sigma,r_I)\setminus\Sigma$, we define an oriented orthonormal frame $\{e_r,e_\th\}$ for $\V$ where $e_r:=\D r$ (and so $e_\th$ represents the angular direction). Let $\omega_\th=\metric{-}{e_\th}$ be the 1-form dual to $e_\th$. In this notation, the original metric on $N(\Sigma,r_I)$ takes the form
\begin{equation}\label{eq-main:original_g}
    g=dr^2+\omega_\th^2+g_{\H}.
\end{equation}
where $g_{\H}$ the restriction of $g$ on $\H$. It's worth reminding ourselves that the fiber bundle $\pi: N(\Sigma,r_I)\to\Sigma$ does not generally respect the Riemannian structure i.e. $\pi^*g_{\Sigma}\ne g_\H$.

Finally, given two functions $h$ and $u$ on $N(\Sigma,r_I)$, we consider the metric
\begin{equation}\label{eq-main:new_g}
    g'=e^{-2pu}dr^2+e^{-2pu}h^2\omega_\th^2+e^{2u}g_{\H},
\end{equation}
where $p=n-2$.

For the sake of clarity and to avoid needless repetition, we summarize the above setting and constructions:

\begin{setup}\label{setup-main:general}
    When assuming this setup, we will use the notations $M$, $g$, $\Sigma$, $\mathcal{N}\Sigma$, $N(\Sigma,r)$, $\Sigma_r,$ $r_I$, $r(x)$, $\pi$, $\V$, $\H$, $e_r$, $e_\th$, $\omega_\th$, $g'$ for the objects described above.
\end{setup}

To simplify the computation of $R_{g'}$, it is sufficient to consider functions $h$ and $u$ appearing in \eqref{eq-main:new_g} which satisfy the certain conditions which we will often assume.

\begin{condition}\label{cond-main:h_and_u}
    Functions $h$ and $u$ on $N(\Sigma,r_I)$ satisfy this condition if:

    \begin{enumerate}[label={(\roman*)}, nosep]
        \item $h=h(r)$ depends only on $r$, 
        \item $u$ splits as a product in the sense that there exists $v\in C^\infty(\Sigma)$ and $w=w(r)$ with $w(0)=1$, such that $u(x)=v(\pi(x))\cdot w(r(x))$,
        \item these functions satisfy
        \begin{equation}\label{eq-main:further_assump}
            \frac12\leq h\leq 2,\qquad
            u\leq 0,\qquad |w|\leq1,\qquad re^{-2nu}\leq1.
        \end{equation}
    \end{enumerate}
\end{condition}

The following is the main estimate of the scalar curvature of $g'$:

\begin{theorem}\label{thm-main:scal_estimate}
    Assume the General Setup \ref{setup-main:general}. There are constants $r_0\in(0,r_I)$ and $C_1,\dots, C_5>0$, depending only on $g$ and $\Sigma$, such that the following holds. If $h, u$ satisfy Condition \ref{cond-main:h_and_u}, then in $N_g(\Sigma,r_0)$ we have
    \begin{equation}\label{eq-main:scal_final}
        \begin{aligned}
            R_{g'} &\geq e^{2pu}R_g+2e^{2pu}\Big[-\frac{h_{rr}}h-\frac{2h_r}{rh}-\alpha_nu_r^2\Big] \\
            &\qquad -C_1e^{2pu}|h_r|
            -C_2e^{2pu}|u_r|
            -C_3\frac{e^{2pu}}{r}\big|1-h\big|
            -C_4e^{-2nu}\big|1-e^{(p+1)u}\big| \\
            &\qquad -C_5e^{-2u}|w|\big(1+||v||^2_{C^2(g_\Sigma)}\big),
        \end{aligned}
    \end{equation}
    where $\alpha_n:=\frac12(n-1)(n-2)$ and $v,w$ are the functions appearing in Condition \ref{cond-main:h_and_u}.
\end{theorem}

It is now time to present the main statement on the creation of drawstrings, summerized previously in Theorem \ref{thm-intro:main}.

\begin{theorem}\label{thm-main:h_and_u}
    Assume the General Setup \ref{setup-main:general}. Given $v_0\in C^\infty(\Sigma)$ with $v_0\leq0$, and any $\epsilon>0$, there exists a radius $r_1<\min\{\epsilon,r_I\}$ and functions $h, u$ satisfying the following:
    \begin{enumerate}[label=\emph{(\Roman*)}, topsep=3pt, itemsep=0ex]
        \item Condition \ref{cond-main:h_and_u} is satisfied with $v=v_0$, and in particular we have $g'|_\Sigma=e^{2v_0}g|_\Sigma$.
        \item The metric $g'$ defined by \eqref{eq-main:new_g} is smooth across $\Sigma$. Moreover, we have $h(r)= u(r)= 0$ for $r\geq r_1$, so $g'=g$ outside $N(\Sigma,r_1)$.
        \item The scalar curvature lower bound $R_{g'}\geq R_g-\epsilon$.
        \item For all $x\in\p N(\Sigma,r_1)$ we have $d_{g'}(x,\pi(x))\leq 3r_1$.
    \end{enumerate}
    Moreover, $g'$ has the following additional properties:
    \begin{enumerate}[label=\emph{(\Roman*)}, topsep=3pt, itemsep=0ex]
        \setcounter{enumi}{4}
        \item the $g'$-mean curvature of $\Sigma_r$ is at least $\frac{1}{4r\log(1/r)}$ for all $r\leq r_1$,
        \item $\Vol_{g'}(N_g(\Sigma,r_1))\leq 12\pi\Area_g(\Sigma)r_1^2$,
        \item the functions $u,h$ satisfy $-\max_\Sigma|v_0|\leq u\leq0$ and $1-r_1\leq h\leq1$.
    \end{enumerate}
\end{theorem}


\section{Computation of scalar curvature}\label{sec:scal}

\subsection{Strategy and setup to prove Theorem \ref{thm-main:scal_estimate}}\label{subsec:setup}

Throughout this section we assume the General Setup \ref{setup-main:general} and that $h,u$ are given functions defined on $N(\Sigma,r_I)$. To establish the scalar curvature estimate of Theorem \ref{thm-main:scal_estimate}, we work locally, on a fixed coordinate chart $U\subset \Sigma$ on which $\mathcal{N}\Sigma$ is trivial. We denote by $\x=(x_1,\cdots,x_{n-2})$ the coordinates on $U$, and make implicit the coordinate map for brevity. For a radius $r_0<r_I$, set $\Omega_{r_0}=\pi^{-1}(U)\cap N(\Sigma,r_0)$. The estimate \eqref{eq-main:scal_final} of Theorem \ref{thm-main:scal_estimate} is obtained on $\Omega_{r_0}$ for some small $r_0$, and then the estimate is patched together via a finite covering of coordinate charts. The constants $C_1\,\sim\,C_5$ in Theorem \ref{thm-main:scal_estimate} thus depend on the choice of charts.

To describe the strategy of our scalar curvature computation, recall that $\Sigma_r=\p N(\Sigma,r)$. We use the shorthand $g_r=g|_{\Sigma_r}$, $g'_r=g'|_{\Sigma_r}$, and let $A,H,R_r$ (resp. $A,'H',R'_r$) denote the second fundamental form, mean curvature, and scalar curvature of $\Sigma_r$ with respect to $g$ (resp. $g'$). Thus, we may express $g=dr^2+g_r$ and $g'=e^{-2pu}dr^2+g'_r$ in $N(\Sigma,r_I)$ (cf. the formulas \eqref{eq-main:original_g}, \eqref{eq-main:new_g}). By a standard combination of the traced Gauss equation and the variation of mean curvature, we have
\begin{equation}\label{eq-setup:gauss_eq1}
	R_g=R_r-\Big(H^2+|A|^2_g+2\frac{\p H}{\p r}\Big).
\end{equation}
For the new metric $g'$ we have
\begin{equation}\label{eq-setup:gauss_eq2}
	R_{g'}=R'_r-\Big(H'^2+|A'|_{g'}^2+2e^{pu}\frac{\p H'}{\p r}\Big),
\end{equation}
where the $e^{pu}$ coefficient comes from the scaling in the $dr$ direction. To compare $R_g$ to $R_{g'}$, it suffices to compute $H'$, $|A'|_{g'}^2$, and $R'_r$ in terms of $H$, $|A|_g^2$, and $R_r$. This is accomplished in Theorems \ref{thm-ds:extrinsic_curv} and \ref{thm-ds:scal_gr} below.

The computations are carried out using the moving frame method. Subsection \ref{subsec:exp} contains general facts of Riemannian geometry near codimension-2 submanifolds. There we set up a particular orthonormal frame $\{e_a\}_{1\leq a\leq n-2}$ for $\H$ near $U\subset\Sigma$ and obtain asymptotic expansions for the pairwise Lie brackets of $\{e_a,e_r,e_\th\}$. Then in Subsections \ref{subsec:fund} and \ref{subsec:scal}, we specialize to the setting of Theorem \ref{thm-main:scal_estimate} and compute the second fundamental form and scalar curvature of $\Sigma_r$ respectively.

\vspace{3pt}

\noindent{\bf{Convention:}} For the rest of this section, the notation $O(r^k)$ is used to denote specific expressions that are bounded uniformly by $Cr^k$ throughout $\Omega_{r_0}$ for some constant $C$ depending only on $g,\Sigma,r_0$, and the chart $U$ (instead of being merely asymptotically bounded as $r\to0$ in the traditional sense).

\subsection{Controlling the Lie brackets}\label{subsec:exp}

The aim of this subsection is to prove the following statement establishing a desirable frame for $\mathcal{H}$ with essential estimates used later in the section.

\begin{theorem}\label{thm-exp:local_frame}
    Assume the General Setup \ref{setup-main:general}. There exists a small radius $r_0<r_I$, a local orthonormal frame $\{e_a\}_{a=1}^{n-2}$ for the horizontal distribution $\H$ over $\Omega_{r_0}$, and a constant $C$, such that the following hold in $\Omega_{r_0}$:
    \begin{enumerate}[label=(\roman*), topsep=3pt, itemsep=0ex]
        \item $\metric{[e_a,e_b]}{e_\th}=O(r)$ and $\metric{[e_a,e_b]}{e_c}=O(1)$ and $e_a\metric{[e_b,e_c]}{e_d}=O(1)$,
        \item $\metric{[e_\th,e_a]}{e_b}=O(1)$ and $\metric{[e_\th,e_a]}{e_\th}= O(r)$, moreover, $e_\th\metric{[e_\th,e_a]}{e_b}= O(r^{-1})$ and $e_a\metric{[e_\th,e_b]}{e_\th}= O(r)$,
        \item $[e_r,e_\th]=\big(-r^{-1}+O(r)\big)e_\th$, and $\metric{[e_r,e_a]}{e_\th}= O(r)$, and $\metric{[e_r,e_a]}{e_b}= O(1)$,
        \item Given any smooth function $v$ on $\Sigma$, define a function $u\in C^\infty(\Omega_{r_0})$ by $u(x)=v(\pi(x))$. Then $|{e_a}(u)|\leq C|\D_\Sigma v|$ and $|{e_a}({e_b}(u))|\leq C\big(|\D_\Sigma v|+|\D^2_\Sigma v|\big)$.
    \end{enumerate}
\end{theorem}

\begin{rmk}
    The frame $\{e_a\}_{a=1}^{n-2}$ need not be smooth across $\Sigma$ in the construction and is only guaranteed to have Lipschitz continuity. This does not affect the smoothness of the drawstring metric or the validity of the eventual scalar curvature estimate in Theorem \ref{thm-main:scal_estimate}, since these do not depend on the choice of $e_a$.
\end{rmk}

To facilitate the calculations in Theorem \ref{thm-exp:local_frame}, we work in a polar coordinate system on $\Omega_{r_0}$ in the following way. Choose a smooth oriented orthonormal frame $\{\nu,\nu^\perp\}$ for $\mathcal{N}\Sigma|_U$. For an angle $\th$, define the normal vector 
\begin{equation}\label{eq-exp:nu_th}
    \nu_\th=\nu\cos\th+\nu^\perp\sin\th
\end{equation} 
and define polar coordinates $\Phi: U\times B^{\RR^2}(r_0)\to\Omega_{r_0}$ by
\begin{equation}\label{eq-exp:polarcoords}
    \Phi(\x,r,\th)=\exp_{\x}(r\nu_\th).
\end{equation}
The metric tensor in this coordinate can be written in general as
\begin{equation}\label{eq-exp:FGdef}
    g=dr^2+Ed\th^2+2\sum_{a} F_ad\th dx_a+\sum_{a,b}G_{ab}dx_adx_b,
\end{equation}
where $E$, $F_a$, $G_{ab}$ are smooth functions on $\Omega_{r_0}$.

We first calculate the Taylor expansion of $E, F_a$. Note that these coefficients are only defined for $r>0$. When referring to their value at $r=0$, we always mean the limit as $r\to0$ while fixing the remaining coordinates $\x,\th$, whenever this limit exists.

\begin{lemma}\label{lemma-exp:deriv_pth}
Assume the General Setup \ref{setup-main:general}. Then the vector \eqref{eq-exp:nu_th} and chart \eqref{eq-exp:polarcoords} satisfy

    (i) $\D_r\p_\th|_{r=0}=\nu_{\th+\pi/2}$ for each $\th$,
    
    (ii) $\D_r\D_r\p_\th|_{r=0}=0$.
\end{lemma}
\begin{proof}
    To establish (i), we switch to an associated Cartesian coordinate system $(\x,y,z)$ where $y=r\cos\th$ and $z=r\sin\th$. Let $\Gamma_{ij}^k$ denote a generic Christoffel symbol for $g$ in these coordinates, and note that these are bounded on $\Omega_{r_0}$. Then we have 
    \[\begin{aligned}
        \D_{\p_r}\p_\th &= \frac{y\D_{\p_y}+z\D_{\p_z}}{\sqrt{y^2+z^2}}\big(y\p_z-z\p_y\big) \\
        &= \frac{y\,\p/\p y+z\,\p/\p z}{\sqrt{y^2+z^2}}\big(y\p_z-z\p_y\big)+O\big(|\Gamma_{ij}^k|\cdot|y^2+z^2|^{3/2}\big) \\
	&= \frac{y\p_z-z\p_y}{\sqrt{y^2+z^2}}+O(r^{3/2})
	= \nu_{\th+\pi/2}+o(1).
    \end{aligned}\]
    Item (ii) follows from the Jacobi equation $\D_{\p_r}\D_{\p_r}\p_\th=R(\p_r,\p_\th)\p_r$ and $\p_\th|_{r=0}=0$.
\end{proof}

\begin{lemma}\label{lemma-exp:E}
    Assume the General Setup \ref{setup-main:general} and consider the function $R_1=-\frac13\sec(\nu,\nu^\perp)$ along $\Sigma$. Then vector field $re_\th$ and function $r^{-2}E$ defined in \eqref{eq-exp:FGdef} are smooth in $N(\Sigma,r_I)$. Moreover, we have
    \begin{equation}\label{eq-exp:asymp_E}
        E(\x,r,\th)=r^2+R_1(\x)r^4+O(r^5)\quad\text{in}\ \ \Omega_{r_0}.
    \end{equation}
\end{lemma}
\begin{proof}
    The formula \eqref{eq-exp:asymp_E} follows from the well-known metric expansion in geodesic polar coordinates, which is applied to the 2-dimensional fiber $S_x=\pi^{-1}(x)\cap N(\Sigma,r_I)$. The smoothness of $r^{-2}E$ and $re_\th$ follow from Lemma \ref{lemma-ana:smoothness_jac}, taking $h$ and $K$ in that Lemma to be $\sqrt{E}$ and the Gauss curvature of $S_x$.
\end{proof}

In the next calculation, we refine our understanding of the coefficients $F_a$. First, let $\tau\in\Omega^1(\Sigma)$ denote the normal connection form, defined by $\tau(W)=\metric{\D_W\nu}{\nu^\perp}$ for $W\in T\Sigma$. To denote the components of $\tau$ in the frame $\{x^a\}$, we use $\tau_a=\tau(\p_{x_a})$.

\begin{lemma}\label{lemma-exp:Fa}
    Assume the General Setup \ref{setup-main:general}. Then the components of $F$ satisfy
    \begin{equation}\label{eq-exp:asymp_F}
        F_a(\x,r,\th)=\tau_a(\x)r^2+O(r^3)\quad\text{in}\ \ \Omega_{r_0}
    \end{equation}
    where $\tau(w)=\metric{\D_w\nu}{\nu^\perp}$, $w\in T\Sigma$, denotes the normal connection form of $\Sigma$.
\end{lemma}
\begin{proof}
    Clearly, $F_a$ is smooth and equals 0 when $r=0$. The first order term is
    \begin{align}
        \frac{\p F_a}{\p r}|_{r=0} &= \metric{\D_{\p_r}\p_\th}{\p_{x_a}}|_{r=0}+\metric{\p_\th}{\D_{\p_r}\p_{x_a}}|_{r=0} \\
        &= \metric{\nu_{\th+\pi/2}}{\p_{x_a}}|_\Sigma+\metric{\p_\th}{\D_{\p_{x_a}}\nu_\th}|_{r=0}=0.
    \end{align}
    The second order term is
    \begin{align}
        \frac{\p^2F_a}{\p r^2}|_{r=0} &= \metric{\D_{\p_r}\D_{\p_r}\p_\th}{\p_{x_a}}|_{r=0}+2\metric{\D_{\p_r}\p_\th}{\D_{\p_r}\p_{x_a}}|_{r=0}+\metric{\p_\th}{\D_{\p_r}\D_{\p_r}\p_{x_a}}|_{r=0} \\
        &= 0+2\metric{\nu_{\th+\pi/2}}{\D_{\p x_a}\nu_\th}+0\qquad
        \text{(by Lemma \ref{lemma-exp:deriv_pth})}\\
        &= 2\tau_a,
    \end{align}
    where the last line follows from the fact $\metric{\D_w\nu_\th}{\nu_{\th+\pi/2}}=\tau(w)$ for all angle $\th$.
\end{proof}

Finally, the following analytic fact is useful.

\begin{lemma}\label{lemma-exp:deriv_asymp}
    If $f\in C^\infty(\bar{\Omega_{r_I}})$ with $f=O(r^k)$ for some integer $k\geq0$, then it satisfies
    \begin{equation}\label{eq-exp:deriv_asymp}
        \p_r f=O\big(r^{\max\{k-1, 0\}}\big),\quad
        \p_{x_a}f= O(r^k),\quad
        \p_\th f= O\big(r^{\max\{k,1\}}\big),
    \end{equation}
    Moreover, for all $\alpha_1,\dots,\alpha_{n-2},\beta\geq0$ we have
    \begin{equation}
        \partial^{\alpha_1}_{x^1}\dots \partial^{\alpha_{n-2}}_{x^{n-2}}\partial^\beta_{\theta}f= O(r^k).
    \end{equation}
\end{lemma}
\begin{proof}
    Switch to Cartesian coordinates $(\x,y=r\cos\th,z=r\sin\th)$. Now we Taylor expand with precise remainder to write
    \[f(\x,y,z)=\sum_{i=0}^k f_i(\x,y,z)y^iz^{k-i}\]
    for some \textit{smooth} functions $f_i$. The bound $\p_{x_a}f=O(r^k)$ immediately follows. Since $\p_r=r^{-1}(y\p_y+z\p_z)$ and $\p_\th=y\p_z-z\p_y$, when $k\geq 1$ it is not hard to see that $\p_r f=O(r^{k-1})$ and $\p_\th f= O(r^k)$. For the case $k=0$, we instead use
    \[f(\x,y,z)=f(\x,0,0)+\sum_{i=0}^1f_i(\x,y,z)y^iz^{1-i},\]
    and note that the zeroth order term vanishes when taking $r$- and $\th$-derivatives. The final statement follows by iteratively applying \eqref{eq-exp:deriv_asymp} to the partial derivatives of $f$.
\end{proof}

With that preparation complete, we may now begin:

\begin{proof}[Proof of Theorem \ref{thm-exp:local_frame}]
    First let $r_0$ be sufficiently small so that the quantity $E$ defined in \eqref{eq-exp:FGdef} satisfies
    \begin{equation}\label{eq-exp:nondeg}
        E\geq\frac12r^2\ \ \text{in}\ \ \Omega_{r_0},
    \end{equation}
    which is possible due to Lemma \ref{lemma-exp:E}. Note that $\frac{r^2}E$ is smooth in $\Omega_{r_0}$.
    
    Working in the polar coordinates defined by \eqref{eq-exp:polarcoords}, consider the vector fields
    \begin{equation}
        w_a:=\p_{x^a}-\frac{\metric{\p_{x^a}}{\p_\th}}{|\p_\th|^2}\p_\th=\p_{x^a}-\frac{F_a}{E}\p_\th,\quad\quad i=1,\dots,n-2,
    \end{equation}
    which are orthogonal to both $\p_r$ and $\p_\th$ by construction. Choose $\{e_a\}$ as the Gram-Schmidt orthogonalization of $\{w_a\}$. In particular, $\{e_a\}$ is an orthonomal frame for $\H$. Let $\gamma_a^b$ denote the matrix of base change, namely $e_a=\gamma_{a}^bw_b$. By Lemma \ref{lemma-exp:E} and \ref{lemma-exp:Fa} we have $|w_a-\p_{x^a}|\leq Cr$, and hence we may choose $r_0$ perhaps smaller to ensure that $w_a$ is well-defined and that $\gamma_a^b$ is bounded in $\Omega_{r_0}$.

    Let us investigate the regularity of $\frac{F_a}E$ and $\gamma_a^b$ which is needed later. For zeroth order bounds, Lemmas \ref{lemma-exp:E} and \ref{lemma-exp:Fa} imply $\big|\frac{F_a}E\big|\leq C$. To bound its derivatives, first consider
    \[
    \p_\th\big(\frac{F_a}E\big)=\frac{\p_\th F_a}{E}-F_a\frac{\p_\th E}{E^2}.
    \]
    Applying Lemma \ref{lemma-exp:deriv_asymp} to the functions $E-r^2-R_1(\x)r^4$ and $F_a-\tau_a(\x)r^2$ (which are smooth of order $r^5$ and $r^3$, respectively), we have $\p_\th E=O(r^5)$ and $\p_\th F_a=O(r^3)$. Combining those observations yeilds $\p_\th\big(\frac{F_a}E\big)=O(r)$. Next, note that $r^2\frac{F_a}E$ is a smooth function, and therefore $r^2\p_\th\big(\frac{F_a}E\big)=\p_\th\big(r^2\frac{F_a}E\big)$ is also smooth and has the order $r^3$. By Lemma \ref{lemma-exp:deriv_asymp} we may then conclude that $\p_\th^2\big(\frac{F_a}E\big)=O(r)$ and $\p_{x^b}\p_\th\big(\frac{F_a}E\big)=O(r)$. Another conclusion from the smoothness of $r^2\frac{F_a}E$ is that $\p_{x^b}\big(\frac{F_a}E\big)=O(1)$ and $\p_{x^b}\p_{x^c}\big(\frac{F_a}E\big)=O(1)$. Finally, we use the asymptotics in Lemmas \ref{lemma-exp:E} and \ref{lemma-exp:Fa} to compute
    \[\begin{aligned}
        \p_r\big(\frac{F_a}E\big)
        &= \frac{2\tau_ar+O(r^2)}{E}-\big(\tau_ar^2+O(r^3)\big)\frac{2r+O(r^3)}{E^2}\\
        {}&= \frac{2\tau_a r\big(r^2+O(r^4)\big)-2\tau_ar^3+O(r^4)}{E^2}
        = \frac{O(r^4)}{E^2}=O(1).
    \end{aligned}\]
    To summarize our work so far, there exists a $C$ so that the following hold throughout $\Omega_{r_0}$:
    \begin{equation}\label{eq-exp:Fa/E_deriv}
        \begin{aligned}
            & \big|\frac{F_a}E\big|\leq C,\quad
            \big|\p_\th\big(\frac{F_a}E\big)\big|\leq Cr,\quad
            \big|\p_\th^2\big(\frac{F_a}E\big)\big|\leq Cr,\quad
            \big|\p_r\big(\frac{F_a}E\big)\big|\leq C, \\
            &\qquad
            \big|\p_{x^b}\big(\frac{F_a}E\big)\big|\leq C,\quad
            \big|\p_{x^b}\p_{x^c}\big(\frac{F_a}E\big)\big|\leq C,\quad
            \big|\p_{x^b}\p_{\th}\big(\frac{F_a}E\big)\big|\leq Cr.
        \end{aligned}
    \end{equation}
    
    We would like to emphasize that, in the above calculations our bounds on $\p_\th(F_a/E)$ and $\p_r(F_a/E)$ depend on the particular form of $F_a$ and $E$, namely that $R_1,\tau_a$ do not depend on $\th$. Directly applying Lemma \ref{lemma-exp:deriv_asymp} with the rough asymptotics $E=O(r^2)$, $F_a=O(r^2)$ would have yielded a bound that is one degree worse than \eqref{eq-exp:Fa/E_deriv}.
    
    From the Gram-Schmidt orthogonalization process, we note that there is sufficiently small $r_0$ and multivariable functions $P_a^b(x,q_1,\cdots,q_{n-2})\in C^\infty\big(\Omega_{r_0}\times(-C,C)^{n-2}\big)$, $a,b=1,\dots,n-2$, such that $\gamma_a^b=P_a^b(x,\frac{F_1}E,\cdots,\frac{F_{n-2}}E)$. By the chain rule and Lemma \ref{lemma-exp:deriv_asymp} we have
    \begin{equation}\label{eq-exp:gamma_deriv}
        \begin{aligned}
            & \big|\gamma_a^b\big|\leq C,\quad
            \big|\p_\th\gamma_a^b\big|\leq Cr,\quad
            \big|\p_\th^2\gamma_a^b\big|\leq Cr,\quad
            \big|\p_r\gamma_a^b\big|\leq C,\quad
            \big|\p_{x^c}\gamma_a^b\big|\leq C, \\
            &\qquad\qquad\qquad\qquad
            \big|\p_{x^c}\p_{x^d}\gamma_a^b\big|\leq C,\quad
            \big|\p_{x^c}\p_\th\gamma_a^b\big|\leq Cr.
        \end{aligned}
    \end{equation}
    We can start verifying the items in the theorem statement.

    \vspace{6pt}
    \noindent\textbf{Item (i).} We compute
    \begin{align}\label{eq-exp:item1prf1}
        [e_a,e_b] &= [\gamma_a^cw_c,\gamma_b^dw_d]
        = \gamma_a^c\gamma_b^d[w_c,w_d]+\gamma_a^c(w_c\gamma_b^d)w_d-\gamma_b^d(w_d\gamma_a^c)w_c.
    \end{align}
    The first term is
    \[\begin{aligned}
        [w_c,w_d] &= \big[\p_{x^c}-\frac{F_c}E\p_\th,\p_{x^d}-\frac{F_d}{E}\p_\th\big] \\
        &= -\p_{x^c}\big(\frac{F_d}E\big)\p_\th+\p_{x^d}\big(\frac{F_c}E\big)\p_\th+\frac{F_c}E\p_\th\big(\frac{F_d}E\big)\p_\th-\frac{F_d}E\p_\th\big(\frac{F_c}E\big)\p_\th \\
        &= O(1)\p_\th
        = O(r)e_\th\qquad\text{(by \eqref{eq-exp:Fa/E_deriv})}.
    \end{aligned}\]
    To estimate the last two terms of \eqref{eq-exp:item1prf1}, we use \eqref{eq-exp:gamma_deriv} to find
    \[\begin{aligned}
        \gamma_a^c(w_c\gamma_b^d)w_d &= \gamma_a^c\big(\p_{x^c}\gamma_b^d-\frac{F_c}E\p_\th\gamma_b^d\big)w_d \\
        &= O(1)\cdot\big(O(1)+O(r)\big)\cdot\big(\p_{x^d}+O(1)\p_\th\big),
    \end{aligned}\]
    and the other term can be bounded similarly. Hence $[e_a,e_b]=\sum_{c=1}^{n-2}O(1)\p_{x^c}+O(r)e_\th$.
    
    To bound $e_a\metric{[e_b,e_c]}{e_d}$, we first note that
    \begin{equation}\label{eq-exp:eaebec}
        \metric{[e_a,e_b]}{e_c}=\underbrace{\gamma_a^p(w_p\gamma_b^q)(\gamma^{-1})_q^l\delta_{cl}}_{=:\mathcal{A}}-\underbrace{\gamma_b^q(w_q\gamma_a^p)(\gamma^{-1})_p^l\delta_{cl}}_{=:\mathcal{B}},
    \end{equation}
    where $\gamma^{-1}$ is the inverse matrix of $\gamma$. It suffices to bound $e_d\mathcal{A}$ and $e_d\mathcal{B}$. For the first term, we have
    \[\begin{aligned}
        e_d\mathcal{A} &= \gamma_d^m\Big(\p_{x^m}-\frac{F_m}E\p_\th\Big)\Big(\gamma_a^p\big(\p_{x^p}\gamma_b^q-\frac{F_p}E\p_\th\gamma_b^q)(\gamma^{-1})_q^l\delta_{cl}\Big).
    \end{aligned}\]
    By expanding this expression and applying \eqref{eq-exp:Fa/E_deriv} and \eqref{eq-exp:gamma_deriv}, we find that all the terms are bounded by constants. The quantity $e_d\mathcal{B}$ is estimated in a nearly identical manner.

    \vspace{6pt}
    \noindent\textbf{Item (ii).} First note that 
    \[\begin{aligned}
        [\p_\th,w_a] &= \big[\p_\th,\p_{x^a}-\frac{F_a}{E}\p_\th\big]
        = -\p_\th\big(\frac{F_a}{E}\big)\p_\th.
    \end{aligned}\]
    Hence
    \begin{equation}\label{eq-exp:eth_ea}
        \begin{aligned}
        [e_\th,e_a] &= \big[\frac{\p_\th}{\sqrt E},\gamma_{a}^bw_b\big] \\
        &= \frac1{\sqrt E}(\p_\th \gamma_{a}^b)w_b
            -\frac1{\sqrt E}\gamma_{a}^b\p_\th\big(\frac{F_b}{E}\big)\p_\th
            -\gamma_{a}^b\big(w_b(\frac1{\sqrt E})\big)\p_\th.
    \end{aligned}
    \end{equation}
    Inserting the definition of $w_b$ and using $w_b\perp e_\th$, $\metric{\p_\th}{e_\th}=\sqrt E$, we have
    \begin{equation}\label{eq-exp:thath}
        \metric{[e_\th,e_a]}{e_\th} = 
        \gamma_a^b\p_\th\big(\frac{F_b}E\big)+\frac1{2E}\gamma_a^b\big(\p_{x^b}E-\frac{F_b}E\p_\th E\big).
    \end{equation}
    Here the first term is bounded by $Cr$ by \eqref{eq-exp:Fa/E_deriv}. For the second term, we observe that $\p_{x^b}E=\p_{x^b}(E-r^2)=O(r^4)$ by \eqref{eq-exp:asymp_E} and Lemma \ref{lemma-exp:deriv_asymp}. Similarly, $\p_\th E=O(r^4)$. Combining $E\geq\frac12r^2$ from \eqref{eq-exp:nondeg} and the fact $\frac{F_a}{E}=O(1)$, it follows that $\metric{[e_\th,e_a]}{e_\th}=O(r)$. 
    
    Next we compute the derivative of $\metric{[e_\th,e_a]}{e_\th}$ from \eqref{eq-exp:thath}:
    \begin{equation}\label{eq-exp:ecetheaeth}
        \begin{aligned}
            e_c\metric{[e_\th,e_a]}{e_\th} &= e_c\big(\gamma_a^b\p_\th(\frac{F_b}E)\big)
            -\frac1{2E^2}(e_cE)\big(\p_{x^b}E-\frac{F_b}E\p_\th E\big) \\
            &\qquad\qquad +\frac1{2E}\Big[e_c\p_{x^b}E-e_c\big(\frac{F_b}E\big)\p_\th E-\frac{F_b}Ee_c\p_\th E\Big].
        \end{aligned}
    \end{equation}
    Using $e_c=\gamma_c^d\big(\p_{x^d}+O(1)\p_\th\big)$ and the inequalities \eqref{eq-exp:Fa/E_deriv} and \eqref{eq-exp:gamma_deriv}, the first term of \eqref{eq-exp:ecetheaeth} is bounded by $Cr$. To estimate the remaining two terms, we apply Lemma \ref{lemma-exp:deriv_asymp} to find
    \begin{equation}\label{eq-exp:bigEineqs}
        \p_\th E,\ \ \p_{x^a}E,\ \ \p_\th^2E,\ \ \p_{x^a}\p_{x^b}E,\ \ \p_{x^a}\p_\th E\ \  \leq \ Cr^4.
    \end{equation}
    Inserting these estimates into the remaining terms of \eqref{eq-exp:ecetheaeth}, we conclude $e_c\metric{[e_\th,e_a]}{e_\th}=O(r)$.
    
    Moving on, by \eqref{eq-exp:eth_ea} again we have
    \[
    \metric{[e_\th,e_a]}{e_b}=\frac1{\sqrt E}(\p_\th\gamma_a^c)\metric{w_c}{e_b}=\frac1{\sqrt E}(\p_\th\gamma_a^c)\cdot(\gamma^{-1})_c^d\delta_{bd},
    \]
    where $\gamma^{-1}$ is the inverse matrix of $\gamma$ and $\delta$ denotes the Kronecker delta. Inequalities \eqref{eq-exp:nondeg} and \eqref{eq-exp:gamma_deriv} then imply that $\metric{[e_\th,e_a]}{e_b}=O(1)$. Finally, we compute
    \[\begin{aligned}
        e_\th\metric{[e_\th,e_a]}{e_b}
        &= \frac1{\sqrt E}\p_\th\Big[\frac1{\sqrt E}(\p_\th\gamma_a^c)(\gamma^{-1})_c^d\delta_{bd}\Big] \\
        &= -\frac1{2E^2}(\p_\th E)(\p_\th\gamma_a^c)(\gamma^{-1})_c^d\delta_{bd}
        +\frac1E(\p_\th^2\gamma_a^c)(\gamma^{-1})_c^d\delta_{bd} \\
        &\qquad +\frac1E(\p_\th\gamma_a^c)(\gamma^{-1})_c^d(\p_\th\gamma_d^f)(\gamma^{-1})_f^g\delta_{bg} \\
        &= O(r^{-4})\cdot O(r^4)\cdot O(r)\cdot O(1)
        +O(r^{-2})\cdot O(r)\cdot O(1) \\
        &\qquad +O(r^{-2})\cdot O(r)\cdot O(1)\cdot O(r)\cdot O(1) \\
        &= O(r^{-1})
    \end{aligned}\]
    where we have made use of \eqref{eq-exp:gamma_deriv} and \eqref{eq-exp:bigEineqs}.

    \vspace{6pt}
    
    \noindent\textbf{Item (iii).} We directly compute
    \[\begin{aligned}
        [e_r,e_\th] &= [\p_r,\frac{\p_\th}{\sqrt E}]
        = -\frac1{2E^{3/2}}\p_r E\cdot\p_\th
        = \big(\!-r^{-1}+O(r)\big)e_\th
    \end{aligned}\]
    in which the last equality follows from Lemma \ref{lemma-exp:E}. Next,
    \[\begin{aligned}
        [e_r,e_a] &= [\p_r,\gamma_a^b(\p_{x^b}-\frac{F_b}{E}\p_\th)]
        = (\p_r\gamma_a^b)\big(\p_{x^b}-\frac{F_b}{E}\p_\th\big)-\gamma_a^b\big(\p_r\frac{F_b}E\big)\p_\th \\
        &= O(1)(\p_{x^b}+O(r)e_\th)+O(1)\cdot O(1)\cdot O(r)e_\th,
    \end{aligned}\]
    where we used \eqref{eq-exp:Fa/E_deriv} and \eqref{eq-exp:gamma_deriv}.

    \vspace{6pt}
    
    \noindent\textbf{Item (iv).} Writen in our coordinate system, the assumption on $u$ becomes $u(\x,r,\th)=v(\x)$. We have
    \[
    e_au=\gamma_a^b\big(\p_{x^b}u-\frac{F_a}E\p_\th u\big)=\gamma_a^b\p_{x^b}u=\gamma_a^b\p_{x^b}v=O(|\D_\Sigma v|)
    \]
    where we used $\p_\th u=0$. For the Hessian of $u$, we compute
    \[\begin{aligned}
        e_ae_bu &= e_a(\gamma_b^c\p_{x^c}u)
        = \gamma_a^d\p_{x^d}\big(\gamma_b^c\p_{x^c}u\big)-\gamma_a^d\frac{F_d}E(\p_\th\gamma_b^c)\p_{x^c}u \\
        &= \gamma_a^d(\p_{x^d}\gamma_b^c)\p_{x^c}u+\gamma_a^d\gamma_b^c(\p_{x^d}\p_{x^c}u)-\gamma_a^d\frac{F_d}E(\p_\th\gamma_b^c)\p_{x^c}u \\
        &= O(1)\cdot O(1)\cdot O(|\D_\Sigma v|)+O(1)\cdot O(|\D_\Sigma v|+|\D^2_\Sigma v|)-O(1)\cdot O(r)\cdot O\big(|\D_\Sigma v|\big),
    \end{aligned}\]
    where we have made use of \eqref{eq-exp:gamma_deriv} in the last line. Item (iv) follows.
\end{proof}

\subsection{Second fundamental form of \texorpdfstring{$\Sigma_r$}{Σr}}\label{subsec:fund}

Having set up the desirable frame and estimates in Theorem \ref{thm-exp:local_frame}, we are ready to compute the extrinsic curvature of the distance neighborhoods $\Sigma_r$.

\begin{theorem}\label{thm-ds:extrinsic_curv}
    Assume the General Setup \ref{setup-main:general} and suppose $r_0<r_I$ is the radius given by Theorem \ref{thm-exp:local_frame}. For any $u,h\in C^\infty(\Omega_{r_0})$, the mean curvature and second fundamental form of $\Sigma_r$ with respect to $g'$ satisfy
    \begin{equation}\label{eq-ds:formula_H}
	H'=e^{pu}\Big(H+\frac{h_r}h\Big)
    \end{equation}
    and
    \begin{equation}\label{eq-ds:formula_A2}
	\begin{aligned}
            |A'|^2_{g'} &= e^{2pu}\Big[|A|^2+\frac{h_r^2}{h^2}-2pu_r\frac{h_r}h+(n-1)(n-2)u_r^2-2pHu_r+2H\frac{h_r}h\Big] \\
		&\qquad +I_1+I_2+I_3,
	\end{aligned}
    \end{equation}
    where $I_1,I_2,I_3$ are certain expressions which satisfy
    \begin{equation}
	|I_1|\leq Ce^{2pu}|u_r|,\quad
	|I_2|\leq Ce^{2pu}\frac{|h_r|}h,\quad
	|I_3|\leq Cr^2e^{2pu}\big|e^{-2pu-2u}h^2-1\big|
    \end{equation}
    in $\Omega_{r_0}$, for some constant $C$ independent of $h, u$.
\end{theorem}
\begin{proof}
    We compute using the local framing $\{e_a\}_{a=1}^{n-2}$ of $\mathcal{H}$ from Theorem \ref{thm-exp:local_frame} and its corresponding coframing $\{\omega_a\}_{a=1}^{n-2}$. (Note that we choose to keep the index $\omega_a$ lower for ease of notation and refrain from using Einstein summation notation in the relevant expressions.) Denote by $g_r=\omega_\th^2+\sum_a\omega_a^2$ and $g'_r=e^{-2pu}h^2\omega_\th^2+e^{2u}\sum_a\omega_a^2$ the induced metrics on $\Sigma_r$. Recall the facts
    \[A=\frac12\L_{\p r}g_r,\quad
        A'=\frac12e^{pu}\L_{\p_r}g'_r.\]
    Therefore, by the chain rule,
    \begin{equation}\label{eq-ds:new_2nd_fund}
	\begin{aligned}
            e^{-pu}A' &= -pe^{-2pu}u_rh^2\omega_\th^2+e^{-2pu}hh_r\omega_\th^2+e^{-2pu}h^2\Big(\frac12\L_{\p_r}\omega_\th^2\Big) \\
            &\qquad +e^{2u}u_r\sum_a\omega_a^2+e^{2u}\sum_a\Big(\frac12\L_{\p_r}\omega_a^2\Big).
	\end{aligned}
    \end{equation}
    Notice that, for all $a,b,c$ with $c\ne a,\,c\ne b$, we have by the Leibniz rule
    \begin{equation}\label{eq-ds:vanishing}
        \L_{\p_r}\omega_a^2(e_\th,e_\th)
        =\L_{\p_r}\omega_\th^2(e_a,e_b)
        =\L_{\p_r}\omega_c^2(e_a,e_b)
        =\L_{\p_r}\omega_c^2(e_\th,e_a)=0.
    \end{equation}
    Taking trace of \eqref{eq-ds:new_2nd_fund} and using \eqref{eq-ds:vanishing}, we obtain
    \[\begin{aligned}
        e^{-pu}H' &= e^{-pu}A'(e^{pu}h^{-1}e_\th,e^{pu}h^{-1}e_\th)+\sum_a e^{-pu}A'(e^{-u}e_a,e^{-u}e_a) \\
        &= -pu_r+\frac{h_r}h+\frac12\L_{\p_r}\omega_\th^2(e_\th,e_\th)+(n-2)u_r+\sum_a \frac12\L_{\p_r}\omega_a^2(e_a,e_a).
    \end{aligned}\]
    Since 
    \begin{equation}\label{eq-ds:fund_form_reduction}
	\begin{aligned}
		\frac12\L_{\p_r}\omega_\th^2(e_\th,e_\th)=A(e_\th,e_\th),\quad
		\frac12\L_{\p_r}\omega_a^2(e_a,e_a)=A(e_a,e_a),
	\end{aligned}
    \end{equation}
    we find
    \[\begin{aligned}
        e^{-pu}H' &= (n-2-p)u_r+\frac{h_r}h+A(e_\th,e_\th)+\sum_a A(e_a,e_a)
	= H+\frac{h_r}h,
    \end{aligned}\]
    establishing \eqref{eq-ds:formula_H}.
    
    Next, we compute $|A'|^2_{g'}$. Note that
    \[\begin{aligned}
        \L_{\p_r}\omega_a^2(e_\th,e_a) &= \big[\omega_a\otimes\L_{\p_r}\omega_a+\L_{\p_r}\omega_a\otimes\omega_a\big](e_\th,e_a)
	= \L_{\p_r}\omega_a(e_\th) \\
	&= \big[(d\iota_{\p_r}+\iota_{\p r}d)\omega_a\big](e_\th)
	= d\omega_a(\p_r,e_\th) \\
	&= -\omega_a([\p_r,e_\th])=0\qquad\text{(by Theorem \ref{thm-exp:local_frame} (iii))}.
    \end{aligned}\]
    Together with \eqref{eq-ds:vanishing} and \eqref{eq-ds:fund_form_reduction}, we calculate from \eqref{eq-ds:new_2nd_fund}
    \begin{align}
        e^{-2pu}|A'|^2_{g'} &= e^{-2pu}A'(e^{pu}h^{-1}e_\th,e^{pu}h^{-1}e_\th)^2
        + e^{-2pu}\sum_{a,b} A'(e^{-u}e_a,e^{-u}e_b)^2 \nonumber\\
        &\qquad +2e^{-2pu}\sum_a A'(e^{pu}h^{-1}e_\th,e^{-u}e_a)^2 \nonumber\\
 	\begin{split}&= \Big[-pu_r+\frac{h_r}h+A(e_\th,e_\th)\Big]^2 \\
	&\qquad +\sum_a\Big[u_r+A(e_a,e_a)\Big]^2
        +\sum_{a\ne b}\sum_c\Big[\frac12\L_{\p_r}\omega_c^2(e_a,e_b)\Big]^2 \label{eq-ds:new_A2}\\
        &\qquad +2\sum_a\Big[\frac12e^{-pu-u}h\L_{\p_r}\omega_\th^2(e_\th,e_a)\Big]^2 
        \end{split}
    \end{align}
    Notice that the original $|A|^2_g$ is recovered from \eqref{eq-ds:new_A2} by making the formal replacements $h\equiv1$, $u\equiv0$, and thus
    \[\begin{aligned}
        |A|^2_g 
        &= A(e_\th,e_\th)^2+\sum_a A(e_a,e_a)^2+\sum_{a\ne b}\sum_c\Big[\frac12\L_{\p r}\omega_c^2(e_a,e_b)\Big]^2 \\
        &\qquad +2\sum_a\Big[\frac12\L_{\p r}\omega_\th^2(e_\th,e_a)\Big]^2.
    \end{aligned}\]
    Expanding the terms in \eqref{eq-ds:new_A2} and extracting the above expression leads to
    \[\begin{aligned}
        e^{-2pu}|A'|^2_{g'} &= |A|_g^2+(p^2+n-2)u_r^2+\frac{h_r^2}{h^2}-2pu_r\frac{h_r}h \\
        &\qquad +2\big(-pu_r+\frac{h_r}h\big)A(e_\th,e_\th)+2u_r\sum_a A(e_a,e_a) \\
        &\qquad +2\Big[e^{-2pu-2u}h^2-1\Big]\cdot\sum_a\Big[\frac12\L_{\p r}\omega_\th^2(e_\th,e_a)\Big]^2.
    \end{aligned}\]
    Using $A(e_\th,e_\th)=H-\sum_a A(e_a,e_a)$ and $p=n-2$, we may convert the expression to
    \begin{align}
    \begin{split}
        e^{-2pu}|A'|^2_{g'} &= |A|_g^2+(n-1)(n-2)u_r^2+\frac{h_r^2}{h^2}-2pu_r\frac{h_r}h \\
        &\qquad -2pHu_r+2H\frac{h_r}h+2\big((p+1)u_r-\frac{h_r}h\big)\sum_a A(e_a,e_a)\\
        &\qquad +\Big[e^{-2pu-2u}h^2-1\Big]\cdot\sum_a\Big[\frac12\L_{\p r}\omega_\th^2(e_\th,e_a)\Big]^2.
            \end{split}
    \end{align}
    We identify the following three terms in the above expression:
    \[\begin{aligned}
        I'_1&=2(p+1)u_r\sum_a A(e_a,e_a), \qquad I'_2=-2\frac{h_r}{h}\sum_a A(e_a,e_a), \\
        I'_3&=\Big[e^{-2pu-2u}h^2-1\Big]\cdot\sum_a\Big[\frac12\L_{\p r}\omega_\th^2(e_\th,e_a)\Big]^2.
    \end{aligned}\]
    To bound these terms, we find
    \[\begin{aligned}
        A(e_a,e_a)=\frac12\L_{\p_r}\omega_a^2(e_a,e_a)=\L_{\p_r}\omega_a(e_a)=-\omega_a([\p_r,e_a]).
    \end{aligned}\]
    Therefore $|A(e_a,e_a)|\leq C$ by Theorem \ref{thm-exp:local_frame} (iii). Finally,
    \[\frac12\L_{\p_r}\omega_\th^2(e_\th,e_a)=\frac12\L_{\p_r}\omega_\th(e_a)=\frac12d\omega_\th(\p_r,e_a)=-\frac12\omega_\th([\p_r,e_a]),\]
    thus $\big|\frac12\L_{\p_r}\omega_\th^2(e_\th,e_a)\big|\leq Cr$ by Theorem \ref{thm-exp:local_frame} (iii). This shows \eqref{eq-ds:formula_A2} with $I_i=e^{2pu}I'_i$.
\end{proof}

For the purposes of our eventual scalar curvature computation, it is useful to organize our estimates in the following fashion.

\begin{cor}
    Assume the General Setup \ref{setup-main:general} and suppose $r_0<r_I$ is the radius given by Theorem \ref{thm-exp:local_frame}. For any $h,u\in C^\infty(\Omega_{r_0})$ we have
    \begin{equation}\label{eq-ds:ext_terms_final}
        \begin{aligned}
            \Big(H'^2+|A'|^2_{g'}+2e^{pu}\frac{\p H'}{\p r}\Big) &= e^{2pu}\Big(H^2+|A|^2_{g}+2\frac{\p H}{\p r}\Big) \\
            &\qquad +2e^{2pu}\Big(\frac{h_{rr}}h+2\frac{h_r}{rh}+\alpha_n u_r^2\Big)+\big(I_4+I_5+I_6\big),
        \end{aligned}
    \end{equation}
    where $\alpha_n=\frac12(n-1)(n-2)$ and $I_4,I_5,I_6$ are specific expressions that satisfy the bounds
    \begin{equation}\label{eq-ds:remainder_bounds}
	|I_4|\leq Ce^{2pu}|u_r|,\quad
	|I_5|\leq Ce^{2pu}\frac{|h_r|}h,\quad
	|I_6|\leq Cr^2e^{2pu}\big|e^{-2pu-2u}h^2-1\big|
    \end{equation}
    in $\Omega_{r_0}$ for some constant $C$.
\end{cor}
\begin{proof}
    Combining the straight-forward computations
    \begin{equation}
        H'^2=e^{2pu}\big(H^2+\frac{h_r^2}{h^2}+2H\frac{h_r}h\big)
    \end{equation}
    and
    \begin{equation}
        e^{pu}\frac{\p H'}{\p r}=e^{2pu}\Big(pu_rH+pu_r\frac{h_r}h+\frac{\p H}{\p r}+\frac{h_{rr}}{h}-\frac{h_r^2}{h^2}\Big)
    \end{equation}
    with \eqref{eq-ds:formula_A2}, we obtain
    \begin{equation}\label{eq-ds:ext_curv}
        \begin{aligned}
            \Big(H'^2+|A'|^2_{g'}+2e^{pu}\frac{\p H'}{\p r}\Big) &= e^{2pu}\Big(H^2+|A|^2_{g}+2\frac{\p H}{\p r}\Big) \\
            &\qquad +2e^{2pu}\Big(\frac{h_{rr}}h+2H\frac{h_r}h+\alpha_nu_r^2\Big)+(I_1+I_2+I_3),
        \end{aligned}
    \end{equation}
    where $I_1,I_2,I_3$ are as in \eqref{eq-ds:formula_A2}. To simplify the $2H\frac{h_r}h$ term, we use \eqref{eq-ds:fund_form_reduction} and Theorem \ref{thm-exp:local_frame}(iii) to find
    \[\begin{aligned}
        H = A(e_\th,e_\th)+\sum_a A(e_a,e_a)
        = -\omega_\th([e_r,e_\th])-\sum_a\omega_a([e_r,e_a])
        = r^{-1}+O(1).
    \end{aligned}\]
    This implies
    \[
    2e^{2pu}\cdot 2H\frac{h_r}h=2e^{2pu}\cdot 2\frac{h_r}{rh}+I_7,
    \]
    where $|I_7|\leq Ce^{2pu}\frac{|h_r|}h$. Then apply this to \eqref{eq-ds:ext_curv} to obtain the desired \eqref{eq-ds:ext_terms_final} (by setting $I_4=I_1$, $I_5=I_2+I_7$, $I_6=I_3$).
\end{proof}

Before moving on, we record an elementary and useful fact about the extrinsic curvature of $\Sigma_r$.

\begin{lemma}\label{lemma-exp:H2-A2}
    Assume the General Setup \ref{setup-main:general} and suppose $r_0<r_I$ is the radius given by Theorem \ref{thm-exp:local_frame}. We have $\big|H^2-|A|^2\big|\leq Cr^{-1}$ for some constant $C$ in $\Omega_{r_0}$.
\end{lemma}
\begin{proof}
    We have calculated above that $A(e_\th,e_\th)=r^{-1}+O(r)$ and $A(e_a,e_a)=O(1)$ and $H=r^{-1}+O(1)$. Moreover, 
    \[
    A(e_a,e_\th)=\frac12\L_{\p_r}\omega_\th^2(e_a,e_\th)+\frac12\L_{\p_r}\omega_a^2(e_a,e_\th)=-\frac12\omega_\th([\p_r,e_a])-\frac12\omega_a([\p_r,e_\th])=O(r)
    \]
    by Theorem \ref{thm-exp:local_frame}(iii). Cancelling the common term $A(e_\th,e_\th)^2$ in both $H^2$ and $|A|^2$, we can bound
    \[\begin{aligned}
        H^2-|A|^2 &= \Big[\sum_a A(e_a,e_a)\Big]^2+2A(e_\th,e_\th)\sum_a A(e_a,e_a)-\sum_a A(e_a,e_\th)^2-\sum_{a,b}A(e_a,e_b)^2 \\
        &= O(r^{-1}). \hfill\qedhere
    \end{aligned}\]
\end{proof}

\subsection{The scalar curvature of \texorpdfstring{$\Sigma_r$}{Σr}}\label{subsec:scal}

The next step is to compute the scalar curvature of the constant $r$ slice $\Sigma_r$. Theorem \ref{thm-exp:local_frame} defined radius $r_0<r_I$ and a $g$-orthonormal frame $\{e_a,e_\th\}$ for $T\Sigma_r$ ($r<r_0$) with dual 1-forms $\{\omega_a,\omega_\th\}$, where $a=1,\dots,n-2$, which we fix throughout this subsection. Suppose we are given two functions $h,u$ on $\Omega_{r_0}$. For simplicity, we introduce the short-hand notation $X=e^{-pu}h$ and $Y=e^u$. The drawstring metric $g'$ restricted to $\Sigma_r$ is
\begin{equation}\label{eq-ds:dsmetricXY}
    g'_r=X^2\omega_\th^2+\sum_a Y^2\omega_a^2,
\end{equation}
the corresponding $g'_r$-orthonormal frame is $e'_\th=X^{-1}e_\th$, $e'_a=Y^{-1}e_a$, and the dual forms are $\omega'_\th=X\omega_\th$, $\omega'_a=Y\omega_a$.

We compute the scalar curvature $R'_r:=R_{\Sigma_r,g'_r}$ by the method of moving frames, which we now briefly recall. Fix a local orthonormal frame $\{e_i\}$ and dual forms $\{\omega_i\}$ on a Riemannian manifold. Denote by $\gamma_{ijk}=\metric{[e_i,e_j]}{e_k}$ the structure constants, and by $\omega_{ij}=\metric{\D_{(\cdot)}e_i}{e_j}$ the connection forms. By Koszul's formula, we have
\begin{equation}\label{eq-ds:koszul}
    \omega_{ij}(e_k)=\frac12(\gamma_{kij}-\gamma_{ijk}+\gamma_{jki}).
\end{equation}
Moreover, $\omega_{ij}$ are the unique 1-forms satisfying $d\omega_i=\sum_j\omega_{ij}\wedge\omega_j$ and $\omega_{ji}=-\omega_{ij}$. The curvature 2-forms are given by $\Omega_{ij}=d\omega_{ij}-\sum_k\omega_{ik}\wedge\omega_{kj}$ and the scalar curvature is expressed as $R=\sum_{i,j}\Omega_{ij}(e_j,e_i)$.

Returning to our specific setting, we adopt the convention that indices $a,b,c,\cdots$ range in $\{1,2,\cdots,n-2\}$, and indices $i,j,k,\cdots$ range in $\{1,2,\cdots,n-2,\th\}$. Let $\omega_{ij}$ (resp. $\omega'_{ij}$) and $\Omega_{ij}$ (resp. $\Omega'_{ij}$) denote the connection and curvature forms of $g_r$ (resp. $g'_r$). We start with the following lemma, which computes the connection forms of $g'_r$:

\begin{lemma}\label{lem-ds:connforms}
    In the above setting, assume $X_\th=Y_\th=0$ (which holds under Condition \ref{cond-main:h_and_u}). Then the connection forms of $g'_r$ in \eqref{eq-ds:dsmetricXY} satisfy
    \begin{equation}\label{eq-ds:new_conn_wtha}
        \omega'_{\th a}=\frac{X}{Y}\omega_{\th a}-\frac12\big(\frac YX-\frac XY\big)\sum_b(\gamma_{\th ab}+\gamma_{\th ba})\omega_b-\frac{X_a}Y\omega_\th,
    \end{equation}
    and
    \begin{equation}\label{eq-ds:new_conn_wab}
        \omega'_{ab}=\omega_{ab}+\frac12(1-\frac{X^2}{Y^2})\gamma_{ab\th}\omega_\th+\frac{Y_a}Y\omega_b-\frac{Y_b}Y\omega_a.
    \end{equation}
\end{lemma}
\begin{proof}
    Let $\overline{\omega}_{\theta a}'=-\overline\omega_{a\theta}'$ and $\overline{\omega}_{ab}'$ denote the right sides of \eqref{eq-ds:new_conn_wtha} and \eqref{eq-ds:new_conn_wab}, respectively. By the uniqueness of connection forms, in order to show $\omega'_{ij}=\overline\omega_{ij}'$, it suffices to verify that $\overline{\omega}'_{ab}=-\overline{\omega}'_{ba}$ and $d\omega'_i=\sum\overline\omega'_{ij}\wedge\omega'_j$. The first condition is clear from the formulas and the asymmetry of $\gamma_{ijk}$ in $ij$. It remains to verify
    \begin{equation}\label{eq-ds:to_verify}
        d\omega'_\th=\sum_a\bar\omega'_{\th a}\wedge\omega'_a\ \ \text{and}\ \ d\omega'_a=\sum_b\bar\omega'_{ab}\wedge\omega'_b+\bar\omega'_{a\th}\wedge\omega'_\th.
    \end{equation}
    
    The first equation of \eqref{eq-ds:to_verify} follows by calculating
    \begin{align*}
        \begin{split}
            \sum_a\overline\omega'_{\th a}\wedge\omega'_a &=
            \sum_a \frac{X}{Y}\omega_{\th a}\wedge Y\omega_a
            - \frac12\big(\frac YX-\frac XY\big)
                \sum_{a,b}
                    \overbrace{(\gamma_{\th ab}+\gamma_{\th ba})}^{\text{symmetric in $a,b$}}\cdot
                    \overbrace{\omega_b\wedge Y\omega_a}^{\text{antisymmetric}} \\
            &\qquad -\sum_a\frac{X_a}Y\omega_\th\wedge Y\omega_a 
        \end{split}\\
        &= X\sum_a\omega_{\th a}\wedge\omega_a+0+dX\wedge\omega_\th=d(X\omega_\th)=d\omega'_\th.
    \end{align*}
    Next, we verify the second part of \eqref{eq-ds:to_verify} for each $a$. To this end, we directly compute
    \begin{align}
        \sum_{b}\overline\omega'_{ab}\wedge\omega'_b &=
            \sum_{b} \omega_{ab}\wedge Y\omega_b
            + \frac12\big(1-\frac{X^2}{Y^2}\big)\sum_{b}\gamma_{ab\th}\omega_\th\wedge Y\omega_b
            - \sum_{b} \frac{Y_b}{Y}\omega_a\wedge Y\omega_b \nonumber\\
        &= Y\sum_{b}\omega_{ab}\wedge\omega_b
            + dY\wedge\omega_a
            + \frac12(Y-\frac{X^2}{Y})\sum_{b}\gamma_{ab\th}\omega_\th\wedge\omega_b,
            \label{eq-ds:omegathetaawedge1}
    \end{align}
    and then,
    \begin{equation}\label{eq-ds:omegathetaawedge2}
        \overline\omega'_{a\th}\wedge\omega'_\th = \frac{X^2}Y\omega_{a\th}\wedge\omega_\th+\frac12\big(Y-\frac{X^2}Y\big)\sum_{b}(\gamma_{\th ab}+\gamma_{\th ba})\omega_b\wedge\omega_\th.
    \end{equation}
    Using the general fact $\gamma_{ijk}=\omega_{jk}(e_i)-\omega_{ik}(e_j)$, we note that
    \begin{align}
        \sum_{b}\gamma_{ab\th}\omega_\th\wedge\omega_b+\sum_{b}(\gamma_{\th ab}+\gamma_{\th ba})\omega_b\wedge\omega_\th &= \sum_b(\gamma_{\th ab}+\gamma_{\th ba}-\gamma_{ab\th})\omega_b\wedge\omega_\th \nonumber\\
        &= 2\sum_b\omega_{a\th}(e_b)\omega_b\wedge\omega_\th=2\omega_{a\th}\wedge\omega_\th. \label{eq-ds:omegathetaawedge3}
    \end{align}
    Adding up \eqref{eq-ds:omegathetaawedge1}, \eqref{eq-ds:omegathetaawedge2}, and \eqref{eq-ds:omegathetaawedge3}, we obtain
    \begin{align*}
    \sum_{b}\overline\omega'_{ab}\wedge\omega'_b+\overline\omega'_{a\th}\wedge\omega'_\th &= Y\sum_{b}\omega_{ab}\wedge\omega_b+dY\wedge\omega_a+\frac{X^2}Y\omega_{a\th}\wedge\omega_\th+\big(Y-\frac{X^2}Y\big)\omega_{a\th}\wedge\omega_\th \\
    &= d(Y\wedge\omega_a)=d\omega'_a.
    \end{align*}
    Thus \eqref{eq-ds:to_verify} is verified, and the lemma follows.
\end{proof}

We summarize the bounds on structural constants and their derivatives in the following lemma.

\begin{lemma}\label{lemma-ds:conn_forms}
    The structure constants and connection forms of $g_r$ satisfy
    \begin{equation}\label{eq-ds:struc_const_bound}
        |\gamma_{a\th\th}|\leq Cr,\quad |\gamma_{\th ab}|\leq C,\quad |\gamma_{ab\th}|\leq Cr,\quad |\gamma_{abc}|\leq C,
    \end{equation}
    \begin{equation}\label{eq-ds:conn_form_bound}
        |\omega_{ij}|_{g_r}\leq C,\quad
        |d\omega_i|_{g_r}\leq C,
    \end{equation}
    and
    \begin{equation}\label{eq-ds:conn_deriv_bound}
        |e_\th\gamma_{\th ab}|\leq Cr^{-1},\quad
        |e_a\gamma_{bcd}|\leq C,\quad
        |e_a\gamma_{\th b\th}|\leq Cr,
    \end{equation}
    for some constant $C$ in $\Omega_{r_0}$.
\end{lemma}
\begin{proof}
    \eqref{eq-ds:struc_const_bound} and \eqref{eq-ds:conn_deriv_bound} follow from Theorem \ref{thm-exp:local_frame} (i) and (ii), and \eqref{eq-ds:conn_form_bound} follows from Koszul formula \eqref{eq-ds:koszul} and $d\omega_i=\sum\omega_{ij}\wedge\omega_j$.
\end{proof}

It is finally time to consider the scalar curvature of $\Sigma_r$.

\begin{theorem}\label{thm-ds:scal_gr}
    Assume the General Setup \ref{setup-main:general} and suppose $r_0<r_I$ is the radius given by Theorem \ref{thm-exp:local_frame}. Suppose $h$ and $u$ are functions on $\Omega_{r_0}$ satisfying Condition \ref{cond-main:h_and_u}. Then the scalar curvatures of $\Sigma_r$ with respect to $g_r$ and $g_r'$ satisfy
    \begin{equation}\label{eq-ds:scal_final}
        \begin{aligned}
            |R'_r-e^{2pu}R_r| &\leq Ce^{2pu}r^{-1}\big|1-h\big|
            + Ce^{-2nu}\big|1-e^{(p+1)u}\big| \\
            &\qquad\qquad\qquad\qquad\qquad
            + Ce^{-2u}|w|\big(1+||v||^2_{C^2(\Sigma,g)}\big)
        \end{aligned}
    \end{equation}
    in $\Omega_{r_0}$, for some constant $C$ independent of $h$, $u$, where $w$ and $v$ are the functions that appear in Condition \ref{cond-main:h_and_u}.
\end{theorem}

\begin{proof}
    Preliminary remarks are in order. Recall that Condition \ref{cond-main:h_and_u} states $u(x)=v(\x)w(r)$, $h=h(r)$, and in $N(\Sigma,r_I)$ we have
    \begin{equation}\label{eq-ds:further_assump}
        u\leq0,\qquad \frac12\leq h\leq2,\qquad re^{-2nu}\leq1,\qquad |w|\leq1.
    \end{equation}
    In the present proof, the generic constant $C$ may increase from line to line, but always remains independent of $h$ and $u$. As usual, we use $\omega_{ij}$ (resp. $\omega'_{ij}$) and $\Omega_{ij}$ (resp. $\Omega'_{ij}$) to denote the connection and curvature forms of $g_r$ (resp. $g'_r$). For visual clarity, we will interchangably use $e_iQ$ and $Q_i$ to denote the the derivative of a scalar quantity $Q$ (in the direction $e_i$) throughout our computations. We will repeatedly use $X_\th=Y_\th=0$, which follows immediately from the specific forms of $h$ and $u$. 
    
    We proceed by computing $R'_r$ in terms of $R_r$ and additional quantities depending on $h$ and $u$. In particular, we will compute $\Omega'_{\th a}(e'_a,e'_\th)$ and $\Omega'_{ab}(e'_b,e'_a)$. First up is the computation of $d\omega'_{\th a}(e'_a,e'_\th)$ and $d\omega'_{ab}(e'_b,e'_a)$. Differentiating \eqref{eq-ds:new_conn_wtha} yields

    \begin{align}
        d\omega'_{\th a}(e'_a,e'_\th) &= \frac1{XY}d\omega'_{\th a}(e_a,e_\th) \notag\\
        &= \frac{1}{XY}\Big[d\Big(\frac{X}{Y}\Big)\wedge\omega_{\theta a}\Big](e_a,e_\theta)
        + \frac{1}{XY}\frac{X}{Y}d\omega_{\theta a}(e_a,e_\theta) \notag\\
        &\qquad -\frac{1}{2}\frac{1}{XY}\Big[d\Big(\frac{Y}{X}-\frac{X}{Y}\Big)\wedge\sum_b(\gamma_{\th ab}+\gamma_{\th ba})\omega_b\Big](e_a,e_\theta) \notag\\
        &\qquad -\frac{1}{2}\frac{1}{XY}\Big(\frac{Y}{X}-\frac{X}{Y}\Big)\sum_b\Big[d(\gamma_{\th ab}+\gamma_{\th ba})\wedge \omega_b\Big](e_a,e_\theta) \notag\\
        &\qquad -\frac{1}{2}\frac{1}{XY}\Big(\frac{Y}{X}-\frac{X}{Y}\Big)\sum_b(\gamma_{\th ab}+\gamma_{\th ba})d\omega_b(e_a,e_\theta)\notag\\
        &\qquad-\frac{1}{XY}\Big[d\Big(\frac{X_a}{Y}\Big)\wedge\omega_\theta\Big](e_a,e_\theta)-\frac{1}{XY}\frac{X_a}{Y}d\omega_\theta(e_a,e_\theta)\notag\\
        \begin{split}\label{eq-ds:dconn_form}
        &= \frac1{Y^2}d\omega_{\th a}(e_a,e_\th)+\frac{(XY^{-1})_a}{XY}\omega_{\th a}(e_\th) \\
        &\qquad -\frac12\Big(\frac1{X^2}-\frac1{Y^2}\Big)\sum_b \Big[d(\gamma_{\th ab}+\gamma_{\th ba})\wedge\omega_b\Big](e_a,e_\th) \\
        &\qquad -\frac12\Big(\frac1{X^2}-\frac1{Y^2}\Big)\sum_b(\gamma_{\th ab}+\gamma_{\th ba})d\omega_b(e_a,e_\th) \\
        &\qquad -\frac1{2XY}\Big[d\big(\frac YX-\frac XY\big)\wedge\sum_b(\gamma_{\th ab}+\gamma_{\th ba})\omega_b\Big](e_a,e_\th) \\
        &\qquad -\frac1{XY}\Big(\frac{X_a}Y\Big)_a-\frac{X_a}{XY^2}d\omega_\th(e_a,e_\th).
        \end{split}
    \end{align}
    Then perform the following simplifications on this expression: for the first line of \eqref{eq-ds:dconn_form} we use $\omega_{\th a}(e_\th)=\gamma_{a\th\th}$, which follows from \eqref{eq-ds:koszul}. For the second line we use
    \[\Big[d(\gamma_{\th ab}+\gamma_{\th ba})\wedge\omega_b\Big](e_a,e_\th)
        = -\delta_{ab}e_\th(\gamma_{\th ab}+\gamma_{\th ba})
        = -2\delta_{ab}e_\th(\gamma_{\th aa}),\]
    where $\delta_{ab}$ denotes the Kronecker delta. 
    For the third and fifth lines we use $d\omega_i(e_j,e_k)=-\omega_i([e_j,e_k])=-\gamma_{jki}$ for all $i,j,k$. Finally, the fourth line vanishes since $X_\th=Y_\th=\omega_b(e_\th)=0$. Taking all of these observations into account, \eqref{eq-ds:dconn_form} becomes
    \begin{equation}\label{eq-ds:aux4}
        \begin{aligned}
            d\omega'_{\th a}(e_a',e_\th') &= \frac1{Y^2}d\omega_{\th a}(e_a,e_\th)
            +\frac{(XY^{-1})_a}{XY}\gamma_{a\th\th}
            +\Big(\frac1{X^2}-\frac1{Y^2}\Big)e_\th\gamma_{\th aa} \\
        &\qquad
            + \frac12\Big(\frac1{X^2}-\frac1{Y^2}\Big)\sum_b(\gamma_{\th ab}+\gamma_{\th ba})\gamma_{a\th b}
            -\frac1{XY}\Big(\frac{X_a}Y\Big)_a
            +\frac{X_a}{XY^2}\gamma_{a\th\th}.
        \end{aligned}
    \end{equation}
    
    To summarize, equation \eqref{eq-ds:aux4} contains $d\omega_{\th a}(e_a,e_\th)$ (which contributes to the original $g_r$-scalar curvature), along with several error terms that eventually contribute to right hand side of \eqref{eq-ds:scal_final}. What happens next is slightly subtle: one of the error terms, $e_\th\gamma_{\th aa}$, blows up on the order of $O(r^{-1})$ and needs to be balanced with a part of the first term of \eqref{eq-ds:aux4}.
    To do this, we use the Koszul formula $\omega_{ij}(e_k)=\frac12(\gamma_{kij}-\gamma_{ijk}+\gamma_{jki})$ to find
    \[\begin{aligned}
        d\omega_{\th a}(e_a,e_\th) &= e_a\big(\omega_{\th a}(e_\th)\big)-e_\th\big(\omega_{\th a}(e_a)\big)-\omega_{\th a}([e_a,e_\th]) \\
        &= e_a\gamma_{a\th\th}+e_\th\gamma_{\th aa}-\omega_{\th a}([e_a,e_\th]).
    \end{aligned}\]
    From this it follows
    \begin{align}\label{eq-ds:aux4.1}
        & \frac1{Y^2}d\omega_{\th a}(e_a,e_\th)+\Big(\frac1{X^2}-\frac1{Y^2}\Big)e_\th\gamma_{\th aa} \\
        &\qquad\qquad\qquad\qquad\quad =\frac1{X^2}d\omega_{\th a}(e_a,e_\th)-\Big(\frac1{X^2}-\frac1{Y^2}\Big)\big[e_a\gamma_{a\th\th}-\omega_{\th a}([e_a,e_\th])\big].\notag
    \end{align}
    Applying this identity to \eqref{eq-ds:aux4} will have the effect of changing the coefficient of $d\omega_{\th a}(e_a,e_\th)$ from $\frac1{Y^2}$ to $\frac1{X^2}$ and removing the singular term $e_\th\gamma_{\th aa}$. On the other hand, according to Lemma \ref{lemma-ds:conn_forms}, the remaining structural constants $\gamma_{a\th\th}$, $e_a\gamma_{a\th\th}$, $\omega_{\th a}([e_a,e_\th])$, and $\gamma_{\th ab}$ are all bounded by constants. Thus combining \eqref{eq-ds:aux4} with \eqref{eq-ds:aux4.1} yeilds
    \begin{align}
        \Big|d\omega'_{\th a}(e_a',e_\th')-\frac1{X^2}d\omega_{\th a}(e_a,e_\th)\Big| &\leq
        C\Big|\frac1{X^2}-\frac1{Y^2}\Big|
        + \frac{C}{Y^2}\Big[|e_a\log X|+|e_a\log Y| \label{eq-ds:aux2}\\
        &\qquad\qquad\quad +|e_a^2\log X|+|e_a\log X|\cdot|e_a\log Y|\Big], \nonumber
    \end{align}
    where $e^2_a$ denotes the second derivative in the direction $e_a$.
    
    Next, we simplify the derivative terms in \eqref{eq-ds:aux2}. We note that $\log X=h-pu=h(r)-pv(\x)w(r)$, hence $e_a\log X(x)=-pw(r)e_av(\x)$ and $e_a^2\log X(x)=-pw(r)e_a^2v(\x)$. Similarly, we have $e_a\log Y=we_av$ and $e_a^2\log Y=we_a^2v$. By Theorem \ref{thm-exp:local_frame}(iv), this gives
    \begin{equation}\label{eq-ds:deriv_logX}
        \begin{aligned}
            & |e_a\log X|\leq Cw(r)|\D_\Sigma v|,\quad & |e_a^2\log X|\leq Cw(r)\big(|\D_\Sigma v|+|\D^2_\Sigma v|\big), \\
            & |e_a\log Y|\leq Cw(r)|\D_\Sigma v|,\quad & |e_a^2\log Y|\leq Cw(r)\big(|\D_\Sigma v|+|\D^2_\Sigma v|\big).
        \end{aligned}
    \end{equation}
    Applying \eqref{eq-ds:deriv_logX} to \eqref{eq-ds:aux2}, and using $|w|\leq 1$ from \eqref{eq-ds:further_assump}, we get the desirable estimate
    \begin{equation}\label{eq-ds:scal_component1}
        \begin{aligned}
            \Big|d\omega'_{\th a}(e'_a,e'_\th)-\frac1{X^2}d\omega_{\th a}(e_a,e_\th)\Big| &\leq C\Big|\frac1{X^2}-\frac1{Y^2}\Big|+Ce^{-2u}|w|\big(1+||v||_{C^2(\Sigma,g)}^2\big).
	\end{aligned}
    \end{equation}
 
    Moving on, we compute the next main component by directly differentiating \eqref{eq-ds:new_conn_wab}. Using the general formula $d\omega_i(e_j,e_k)=-\gamma_{jki}$, we have:
    \begin{align}
        d\omega'_{ab}(e'_b,e'_a) &= \frac1{Y^2}d\omega'_{ab}(e_b,e_a) \nonumber\\
        \begin{split}&= \frac1{Y^2}d\omega_{ab}(e_b,e_a)
        + \frac1{2Y^2}\overbrace{\Big[d\Big[\big(1-\frac{X^2}{Y^2}\big)\gamma_{ab\th}\Big]\wedge\omega_\th\Big](e_b,e_a)}^{=0} \nonumber\\
        &\qquad + \frac12\frac{Y^2-X^2}{Y^4}\gamma_{ab\th}\overbrace{d\omega_\th(e_b,e_a)}^{=-\gamma_{ba\th}}
        +\frac{1}{Y^2}\Big[d\Big(\frac{Y_a}{Y}\Big)\wedge\omega_b\Big](e_b,e_a)+\frac{1}{Y^2}\frac{Y_a}{Y}\overbrace{d\omega_b(e_b,e_a)}^{=-\gamma_{bab}} \\
        &\qquad -\frac{1}{Y^2}\Big[d\Big(\frac{Y_b}{Y}\Big)\wedge \omega_a\Big](e_b,e_a)- \frac{1}{Y^2}\frac{Y_b}{Y} \overbrace{d\omega_a(e_b,e_a)}^{=-\gamma_{baa}}\end{split} \nonumber\\
        \begin{split}&= \frac1{Y^2}d\omega_{ab}(e_b,e_a)+\frac12\frac{Y^2-X^2}{Y^4}\gamma_{ab\th}^2-\frac1{Y^2}\Big(\frac{Y_a}Y\Big)_a-\frac1{Y^2}\Big(\frac{Y_b}Y\Big)_b \\
        &\qquad -\frac{Y_a}{Y^3}\gamma_{bab}+\frac{Y_b}{Y^3}\gamma_{baa}.\end{split}\label{eq-ds:daby2}
    \end{align}
    This time, changing the coefficient of $d\omega_{ab}(e_b,e_a)$ in \eqref{eq-ds:daby2} from $\frac1{Y^2}$ to $\frac1{X^2}$ is more straight-forward. Indeed, we use the Koszul formula and Lemma \ref{lemma-ds:conn_forms} to find
    \[\begin{aligned}
        d\omega_{ab}(e_b,e_a) &= e_b\gamma_{baa}+e_a\gamma_{abb}-\omega_{ab}([e_b,e_a])=O(1),
    \end{aligned}\]
    and so
    \begin{equation}\label{eq-ds:auxYtoX1}
        \frac{1}{Y^2}d\omega_{ab}(e_a,e_b)\leq\frac{1}{X^2}d\omega_{ab}(e_a,e_b)+C\Big|\frac{1}{X^2}-\frac{1}{Y^2}\Big|.
    \end{equation}
    For the second term of \eqref{eq-ds:daby2}, using $|\gamma_{ab\th}|\leq Cr$ from Lemma \ref{lemma-ds:conn_forms} along with the assumptions \eqref{eq-ds:further_assump}, we have
    \begin{align}
        \frac{|X^2-Y^2|}{Y^4}\gamma_{ab\th}^2 &\leq Cr^2\frac{X^2}{Y^2}\Big|\frac1{X^2}-\frac1{Y^2}\Big|\notag
        = Cr^2e^{-2pu-2u}h^2\Big|\frac1{X^2}-\frac1{Y^2}\Big|\\
        &\leq C\Big|\frac1{X^2}-\frac1{Y^2}\Big|.\label{eq-ds:auxYtoX2}
    \end{align}
    Applying \eqref{eq-ds:auxYtoX1} and \eqref{eq-ds:auxYtoX2} to \eqref{eq-ds:daby2}, and using \eqref{eq-ds:deriv_logX} for the remaining terms, we obtain
    \begin{equation}\label{eq-ds:scal_component2}
	\begin{aligned}
            \Big|d\omega'_{ab}(e'_b,e'_a)-\frac1{X^2}d\omega_{ab}(e_b,e_a)\Big| &\leq C\Big|\frac1{X^2}-\frac1{Y^2}\Big| +Ce^{-2u}|w|\big(1+||v||_{C^2(\Sigma,g)}^2\big).
        \end{aligned}
    \end{equation}
    
    It remains to compute the components $\omega'_{\th i}\wedge\omega'_{ia}(e'_a,e'_\th)$ and $\omega'_{ai}\wedge\omega'_{ib}(e'_b,e'_a)$ in the curvature form. We define quantities $\rho_{\theta a}$ and $\rho_{ab}$ as
    \begin{equation}
        \rho_{\th a}:=-\frac12\big(\frac YX-\frac XY\big)\sum_b(\gamma_{\th ab}+\gamma_{\th ba})\omega_b-\frac{X_a}Y\omega_\th,
    \end{equation}
    and
    \begin{equation}
        \rho_{ab}:=\frac12(1-\frac{X^2}{Y^2})\gamma_{ab\th}\omega_\th+\frac{Y_a}Y\omega_b-\frac{Y_b}Y\omega_a,
    \end{equation}
    so that 
    \[\omega'_{\th a}=\frac XY\omega_{\th a}+\rho_{\th a},\qquad
        \omega'_{ab}=\omega_{ab}+\rho_{ab}.\]
    Using Lemma \ref{lemma-ds:conn_forms} we can bound
    \begin{equation}\label{eq-ds:rho_bound1}
        |\rho_{\th a}(e_\th)|\leq\frac{|X_a|}{Y},\qquad
	|\rho_{\th a}(e_b)|\leq C\frac{|X^2-Y^2|}{XY},
    \end{equation}
    and
    \begin{equation}\label{eq-ds:rho_bound2}
        |\rho_{ab}(e_\th)|\leq C\frac{|X^2-Y^2|}{Y^2}r,\qquad
	|\rho_{ab}(e_c)|\leq\frac{|Y_a|+|Y_b|}{Y}.
    \end{equation}
    Now expand
    \[\begin{aligned}
        \sum_i[\omega'_{\th i}\wedge\omega'_{ia}](e'_a,e'_\th) &= \frac1{XY}\sum_b[\omega'_{\th b}\wedge\omega'_{ba}](e_a,e_\th) \\
        &= \frac1{Y^2}\sum_b[\omega_{\th b}\wedge\omega_{ba}](e_a,e_\th)
        +\frac1{XY}\sum_b[\rho_{\th b}\wedge\omega_{ba}](e_a,e_\th) \\
        &\qquad +\frac1{Y^2}\sum_b[\omega_{\th b}\wedge\rho_{ba}](e_a,e_\th)
        +\frac1{XY}\sum_b[\rho_{\th b}\wedge\rho_{ba}](e_a,e_\th),
    \end{aligned}\]
    and use \eqref{eq-ds:rho_bound1}, \eqref{eq-ds:rho_bound2}, and $|\omega_{ij}|\leq C$ from \eqref{eq-ds:conn_form_bound} to bound
    \begin{align}
        & \Big|\sum_i[\omega'_{\th i}\wedge\omega'_{ia}](e'_a,e'_\th)-\frac1{Y^2}\sum_i[\omega_{\th i}\wedge\omega_{ia}](e_a,e_\th)\Big| \notag\\
        \leq&\, \frac{1}{XY}\sum_b\Big[C|\rho_{\theta b}(e_a)|+C|\rho_{\theta b}(e_\theta)|\Big]
        +\frac{1}{Y^2}\sum_b\Big[C|\rho_{ba}(e_a)|+C|\rho_{ba}(e_\theta)|\Big]\notag \\
        &\qquad +\frac{1}{XY}\sum_b\Big[|\rho_{\theta b}(e_a)|\,|\rho_{ba}(e_\theta)|+|\rho_{\theta b}(e_\theta)|\,|\rho_{ba}(e_a)|\Big] \notag\\
        \leq&\, \frac1{XY}\Big[C\frac{|X^2-Y^2|}{XY}+C\frac{|X_a|}{Y}\Big]
        +\frac1{Y^2}\Big[C\frac{|X^2-Y^2|}{Y^2}r+C\sum_b\frac{|Y_b|}{Y}\Big]\notag \\
        &\qquad +\frac{1}{XY}\Big[C\frac{|X^2-Y^2|^2}{XY^3}r+C\sum_b\frac{|X_a|\cdot|Y_b|}{Y^2}\Big] \notag\\
        \begin{split}\leq&\, C\Big|\frac1{X^2}-\frac1{Y^2}\Big|\cdot\underbrace{\Big[1+r\frac{X^2}{Y^2}+r\frac{|X^2-Y^2|}{Y^2}\Big]}_{=:\mathcal{A}_1} \\
        &\qquad + \frac{C}{Y^2}\underbrace{\Big(|e_a\log X|+\sum_b|e_b\log Y|+\sum_b|e_a\log X|\cdot|e_b\log Y|\Big)}_{=:\mathcal{A}_2}.\end{split}  \label{eq-ds:aux3}
    \end{align}
    
    To continue our estimate \eqref{eq-ds:aux3}, we make 3 observations. For the first line of \eqref{eq-ds:aux3}, note that $r\frac{X^2}{Y^2}=re^{-2pu-2u}h^2\leq 1$ by our assumptions \eqref{eq-ds:further_assump}, and it follows that $r\frac{|X^2-Y^2|}{Y^2}\leq r\big(1+\frac{X^2}{Y^2}\big)\leq 2$. Consequently, $\mathcal{A}_1$ is uniformly bounded. On the other hand, \eqref{eq-ds:deriv_logX} implies that $\mathcal{A}_2\leq Cw(1+||v||^2_{C^2(\Sigma,g)})$. Finally, the connection forms $[\omega_{\th i}\wedge\omega_{ia}](e_a,e_\th)$ are bounded by constants, and so we may change its coefficient from $\frac1{Y^2}$ to $\frac1{X^2}$ at the cost of $C\big|\frac1{X^2}-\frac1{Y^2}\big|$. Combining these observations with \eqref{eq-ds:aux3} yields
    \begin{align}
        & \Big|\sum_i[\omega'_{\th i}\wedge\omega'_{ia}](e'_a,e'_\th)-\frac1{X^2}\sum_i[\omega_{\th i}\wedge\omega_{ia}](e_a,e_\th)\Big| \notag\\
        &\hspace{90pt}\leq \Big|\sum_i[\omega'_{\th i}\wedge\omega'_{ia}](e'_a,e'_\th)-\frac1{Y^2}\sum_i[\omega_{\th i}\wedge\omega_{ia}](e_a,e_\th)\Big|\notag\\
        &\hspace{90pt}\qquad\qquad+\Big|\frac{1}{X^2}-\frac{1}{Y^2}\Big|\cdot\Big|\sum_i[\omega_{\th i}\wedge\omega_{ia}](e_a,e_\th)\Big|\notag\\
        &\hspace{90pt}\leq C\Big|\frac1{X^2}-\frac1{Y^2}\Big|+Ce^{-2u}|w|\big(1+||v||_{C^2(\Sigma,g)}^2\big).\label{eq-ds:scal_component3}
    \end{align}
    
    Moving along, we split the last component $\sum_i[\omega'_{ai}\wedge\omega'_{ib}](e'_b,e'_a)$ into two pieces $\sum_c[\omega'_{ac}\wedge\omega'_{cb}](e'_b,e'_a)+[\omega'_{a\th}\wedge\omega'_{\th b}](e_b,e_a)$ and separately compute them. Similarly as above, we expand the first piece:
    \[\begin{aligned}
        \sum_c[\omega'_{ac}\wedge\omega'_{cb}](e'_b,e'_a) 
        &= \frac1{Y^2}\sum_c[\omega_{ac}\wedge\omega_{cb}](e_b,e_a)+\frac1{Y^2}\sum_c[\rho_{ac}\wedge\omega_{cb}](e_b,e_a) \\
        &\qquad +\frac1{Y^2}\sum_c[\omega_{ac}\wedge\rho_{cb}](e_b,e_a)+\frac1{Y^2}\sum_c[\rho_{ac}\wedge\rho_{cb}](e_b,e_a).
    \end{aligned}\]
    Repeating the steps used in \eqref{eq-ds:aux3} and \eqref{eq-ds:scal_component3} yeilds
    \begin{align}
        & \Big|\sum_c[\omega'_{ac}\wedge\omega'_{cb}](e'_b,e'_a)-\frac1{X^2}\sum_c[\omega_{ac}\wedge\omega_{cb}](e_b,e_a)\Big| \notag\\
        &\hspace{90pt}\leq\, C\Big|\frac1{X^2}-\frac1{Y^2}\Big|+\frac1{Y^2}\Big[C\sum_c\frac{|Y_c|}{Y}\Big]
        + \frac1{Y^2}\Big[C\sum_c\frac{|Y_c|}{Y}\Big]^2 \notag\\
        &\hspace{90pt}\leq\, C\Big|\frac1{X^2}-\frac1{Y^2}\Big|+Ce^{-2u}|w|\big(1+||v||_{C^2(\Sigma,g)}^2\big).\label{eq-ds:scal_component4}
    \end{align}
    The second piece of the last component expands as
    \begin{align}
        [\omega'_{a\th}\wedge\omega'_{\th b}](e'_b,e'_a) &= \frac1{Y^2}[\omega'_{a\th}\wedge\omega'_{\th b}](e_b,e_a) \\
        \begin{split}
        &= \frac{X^2}{Y^4}[\omega_{a\th}\wedge\omega_{\th b}](e_b,e_a)
        +\frac{X}{Y^3}[\omega_{a\th}\wedge\rho_{\th b}](e_b,e_a) \\
        &\qquad +\frac{X}{Y^3}[\rho_{a\th}\wedge\omega_{\th b}](e_b,e_a)
        +\frac1{Y^2}[\rho_{a\th}\wedge\rho_{\th b}](e_b,e_a).\end{split}\label{eq-ds:finaltermaux1}
    \end{align}
    We proceed with the following estimate, using \eqref{eq-ds:conn_form_bound} in the first inequality, \eqref{eq-ds:finaltermaux1} and \eqref{eq-ds:rho_bound1} in the second, and \eqref{eq-ds:further_assump} in the final:
    \begin{align}
        &\hspace{-60pt} \Big|[\omega'_{a\th}\wedge\omega'_{\th b}](e'_b,e'_a)-\frac1{X^2}[\omega_{a\th}\wedge\omega_{\th b}](e_b,e_a)\Big| \nonumber\\
        \leq & \, C\Big| \frac{1}{X^2}-\frac{X^2}{Y^4}\Big|+\Big|[\omega'_{a\th}\wedge\omega'_{\th b}](e'_b,e'_a)-\frac{X^2}{Y^4}[\omega_{a\th}\wedge\omega_{\th b}](e_b,e_a)\Big| \nonumber\\
        \leq & \, C\Big|\frac1{X^2}-\frac{X^2}{Y^4}\Big|+\frac{X}{Y^3}\Big[C\frac{|X^2-Y^2|}{XY}\Big]\notag\\
        &+ \frac X{Y^3}\Big[C\frac{|X^2-Y^2|}{XY}\Big]
        + \frac1{Y^2}\Big[C\frac{|X^2-Y^2|}{XY}\Big]^2 \nonumber\\
        \leq&\, C\Big|\frac1{X^2}-\frac1{Y^2}\Big|\cdot\Big[\big|1+\frac{X^2}{Y^2}\big|+\frac{X^2}{Y^2}+\frac{|X^2-Y^2|}{Y^2}\Big] \nonumber\\
        \leq&\, C\Big|\frac1{X^2}-\frac1{Y^2}\Big|\cdot\Big[1+\frac{X^2}{Y^2}\Big] \nonumber\\
        =& C\Big|\frac1{X^2}-\frac1{Y^2}\Big|\cdot\Big[1+e^{-2pu-2u}h^2\Big] \nonumber\\
        \leq&\, C\Big|\frac1{X^2}-\frac1{Y^2}\Big|\cdot e^{-2pu-2u}. \label{eq-ds:scal_component5}
    \end{align}
    This completes the estimation of all relevant terms in $\Omega'_{ij}$. 
    
    Combining \eqref{eq-ds:scal_component1}, \eqref{eq-ds:scal_component2}, \eqref{eq-ds:scal_component3}, \eqref{eq-ds:scal_component4}, and \eqref{eq-ds:scal_component5} together, we eventually obtain
    \begin{equation}\label{eq-ds:scal_prefinal}
        \big|R'_r-\frac1{X^2}R_r\big| \leq Ce^{-2pu-2u}\Big|\frac1{X^2}-\frac1{Y^2}\Big|+Ce^{-2u}|w|\big(1+||v||_{C^2(\Sigma,g)}^2\big).
    \end{equation}
    Further analyzing the first term on the right hand side, we find
    \[\begin{aligned}
        e^{-2pu-2u}\Big|\frac1{X^2}-\frac1{Y^2}\Big| &= e^{-2pu-2u}\big|e^{2pu}h^{-2}-e^{2pu}+e^{2pu}-e^{-2u}\big| \\
        &\leq e^{-2u}\big|h^{-2}-1\big|+e^{-2pu-4u}|e^{2(p+1)u}-1| \\
        &\leq Ce^{2pu}r^{-1}\big|1-h\big|+Ce^{-2nu}\big|1-e^{(p+1)u}\big|,
    \end{aligned}\]
    where the last line follows from assumptions \eqref{eq-ds:further_assump}. The error between the scaling $\frac1{X^2}$ and the target scaling $e^{2pu}$ is bounded by
    \[\Big|\frac1{X^2}-e^{2pu}\Big|\cdot|R_r|\leq Ce^{2pu}|1-h|\cdot Cr^{-1},\]
    where the fact $|R_r|\leq Cr^{-1}$ follows from traced Gauss equation $R_r=R_g-2\Ric_g(\p_r,\p_r)+H^2-|A|^2$ and Lemma \ref{lemma-exp:H2-A2}. Combining the above observations, we finally obtain
    \begin{equation}
        \begin{aligned}
            \big|R'_r-e^{2pu}R_r\big| &\leq \big|R'_r-\frac1{X^2}R_r\big|+\big|\frac1{X^2}-e^{2pu}\big|\cdot|R_r| \\
            &\leq C\frac{e^{2pu}}r\big|1-h\big|+Ce^{-2nu}\big|1-e^{(p+1)u}\big|+Ce^{-2u}|w|\big(1+||v||_{C^2(\Sigma,g)}^2\big),\notag
        \end{aligned}
    \end{equation}
    finishing the proof.
\end{proof}

\subsection{The proof of Theorem \ref{thm-main:scal_estimate}}

To conclude, we combine the main results of of this section to finish the proof of our main estimate.

\begin{proof}[Proof of Theorem \ref{thm-main:scal_estimate}]
    Let $r_0$ be the radius provided by Theorem \ref{thm-exp:local_frame}. We combine the traced Gauss equations for $g$ and $g'$ \eqref{eq-setup:gauss_eq1} \eqref{eq-setup:gauss_eq2} with the extrinsic curvature estimate \eqref{eq-ds:ext_terms_final} in Theorem \ref{thm-ds:extrinsic_curv} and $\Sigma_r$'s scalar curvature estimate \eqref{eq-ds:scal_final} in Theorem \ref{thm-ds:scal_gr}. What results is not quite the desired estimate \eqref{eq-main:scal_final}, and we require a small observation about the error term $|I_6|$ in \eqref{eq-ds:remainder_bounds}. The $|I_6|$ error term is further bounded using Condition \ref{cond-main:h_and_u}(iii)
    \[\begin{aligned}
        Cr^2e^{2pu}\big|e^{-2pu-2u}h^2-1\big| &= Cr^2e^{-2u}\big|h^2-1-e^{2(p+1)u}+1\big| \\
        &\leq 3C\big|1-h|+2C\big|1-e^{(p+1)u}\big|.
    \end{aligned}\]
    In particular, $|I_6|$ can be absorbed in the terms $e^{2pu}r^{-1}\big|1-h\big|$ and $e^{-2nu}\big|1-e^{(p+1)u}\big|$.
    
    Thus, we have obtained the desired estimate on $\Omega_{r_0}$. In general, we cover $\Sigma$ with a finite collection of coordinate charts $\{U_i\}$, each trivializing $N\Sigma$. Within each region $\pi^{-1}(U_i)\cap N(\Sigma,r_0)$ we perform the above estimate, and it follows that \eqref{eq-main:scal_final} holds by choosing $C_1\,\sim\,C_5$ as the maximal constant among those in all the $\pi^{-1}(U_i)\cap N(\Sigma,r_0)$.
\end{proof}


\section{Construction of \texorpdfstring{$h$ and $u$}{h and u}}\label{sec:cutoff}

In Theorem \ref{thm-main:scal_estimate}, we obtained a lower bound on the scalar curvature of the metric $g'$, in terms of $R_g$ and several terms involving $h$ and $u$. This section aims to construct $h$ and $u$ such that $R_{g'}\geq R_g-\epsilon$, and the remaining conditions in Theorem \ref{thm-main:h_and_u} are satisfied. The construction here is a generalization of the one used in \cite{KXdrawstring}.

\vspace{6pt}

\noindent\textbf{Setup:} The following data are fixed through this section: the function $v_0\in C^\infty(\Sigma)$ satisfying $v_0\leq0$ and constant $\varepsilon>0$ appearing in Theorem \ref{thm-main:h_and_u}, and the constants $r_0, C_1,\cdots,C_5$ appearing in Theorem \ref{thm-main:scal_estimate}.

\vspace{6pt}

We start with determining the radius $r_1$ appearing in Theorem \ref{thm-main:h_and_u}. We set
\begin{align}
    C_6 &:= C_4(p+1)\sup_\Sigma|v_0|+C_5\big(1+||v_0||_{C^2(\Sigma,g)}^2\big),\label{eq-cutoff:def_C6} \\
    R_0 &:= \sup_{N(\Sigma,r_0)}|R_g|.\label{eq-cutoff:def_R0}
\end{align}
The choice of $r_1$ is fixed once and for all according to the following lemma:
\begin{lemma}\label{lemma-cutoff:r1}
    There exists $r_1>0$ depending on $v_0,\epsilon,r_0,C_1,\cdots,C_6,R_0$, such that
    \begin{enumerate}[label=(\roman*)]
        \item $r_1<\min\big\{\frac1{100},r_0,\epsilon,\frac1{2C_1},\frac1{C_3}\big\}$, and in particular, $\log(\frac1{r_1})>4$ and $\log\log(\frac1{r_1})>0$,

        \item for all $0<r\leq r_1$, $\frac12\leq a\leq2$ and $1\leq b\leq 5$ we have
        \begin{equation}\label{eq-cutoff:rawb}
            \frac1{r^a\log(1/r)^b}>\max\Big\{4C_1,\,2C_2\sup_\Sigma|v_0|,\,2C_3,\,4C_6,\,100(R_0+1)\Big\}.
        \end{equation}
    \end{enumerate}
\end{lemma}
\begin{proof}
    This follows from the fact that for any $a,b>0$, we have $\lim_{r\to0^+}\frac1{r^a\log(1/r)^b}=\infty$.
\end{proof}

\vspace{6pt}

Next, we construct the functions $h, u$. We fix two smooth functions $\zeta$ and $\eta$ satisfying
\begin{align}\label{eq-cutoff:etazeta}
    \begin{split}
        & \zeta|_{[0,\frac12]}=0,\quad \zeta|_{[1,\infty)}=1,\quad    0\leq\zeta'\leq4,\quad |\zeta''|\leq16,\\
        & \eta|_{[0,\frac12]}=1,\quad \eta|_{[1,\infty)}=0,\quad 0\geq\eta'\geq-4, \quad|\eta''|\leq16.
    \end{split}
\end{align} 
For positive constants $r_2,c_1,c_2$ with $r_2\ll r_1$, we set
\begin{align}
    \psi(r)&=\int_0^r\Big[\zeta\big(\frac\rho{r_2}\big)\frac{1}{\rho\log^2(1/\rho)}+\big(1-\zeta\big(\frac\rho{r_2}\big)\big)\frac \rho{r_2}\Big]\,d\rho, \label{eq-cutoff:def_psi}\\
    h(r)&=1-c_1\eta\big(\frac r{r_1}\big)\psi(r), \label{eq-cutoff:def_h}
\end{align}
and consider the function
\begin{align}
    w(r) &= c_2\int_r^\infty\zeta\big(\frac \rho{4r_2}\big)\eta\big(\frac{4\rho}{r_1}\big)\frac{d\rho}{\rho\log(1/\rho)}\label{eq-cutoff:def_u}.
\end{align}
Finally, set
\begin{equation}\label{eq-cutoff:hudefinition}
    h(x)=h(r(x)),\qquad u(x)=v_0(\pi(x))w(r(x)).
\end{equation}
Figures \ref{fig:cutoffs} and \ref{fig:handu} depict the graphs of the cutoff functions and $h$, $w$. In Figure \ref{fig:handu}, the fact that $w(0)=1$ is proved in Lemma \ref{lemma-cutoff:misc}, after all the parameters are chosen appropriately.

\vspace{-18pt}
\begin{figure}[h!]
    \setlength{\abovecaptionskip}{-14pt}
    \hspace{-133pt}\includegraphics[scale=1.15]{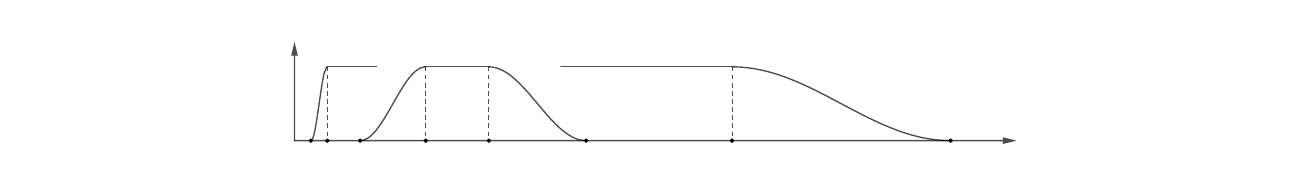}
    \begin{picture}(0,0)
        \put(48,90){$\zeta\big(\frac r{r_2}\big)$}
        \put(90,90){$\zeta\big(\frac{r}{4r_2}\big)\eta\big(\frac{4r}{r_1}\big)$}
        \put(210,90){$\eta\big(\frac r{r_1}\big)$}
        \put(22,28){$0$}
        \put(32,28){$\frac{r_2}2$}
        \put(44,28){$r_2$}
        \put(58,28){$2r_2$}
        \put(95,28){$4r_2$}
        \put(132,28){$\frac{r_1}8$}
        \put(186,28){$\frac{r_1}4$}
        \put(269,28){$\frac{r_1}2$}
        \put(394,28){$r_1$}
        \put(430,28){$r$}
    \end{picture}
    \caption{The cutoff functions that appear when defining $h$ and $w$.}
    \label{fig:cutoffs}
    \TODO{before publishing: the graph should be on the same page as the cutoff functions.}
\end{figure}
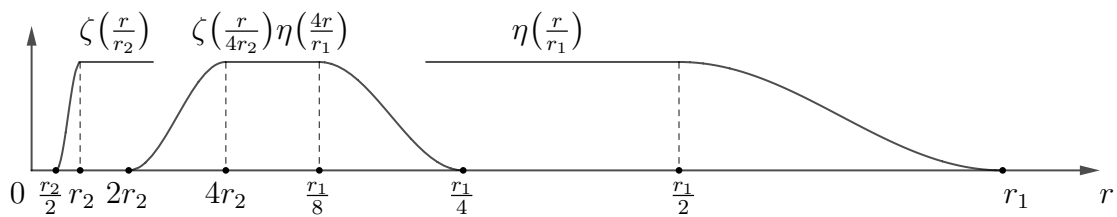

\TODO{before posting: adjust figure position}
\begin{figure}[h]
    \centering
    \hspace*{-24pt}\includegraphics[scale=0.8]{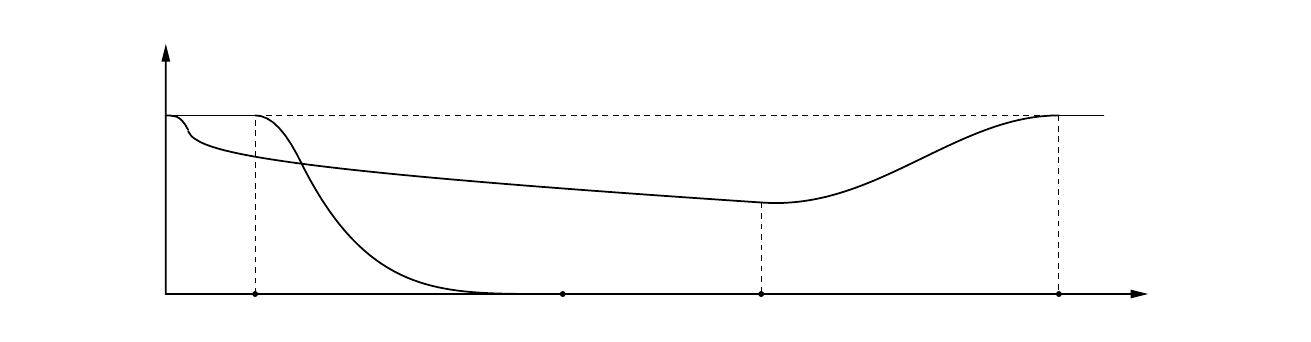}
    \begin{picture}(0,0)
        \put(-448,6){0}
        \put(-415,6){$2r_2$}
        \put(-293,5){$\frac{r_1}4$}
        \put(-217,5){$\frac{r_1}2$}
        \put(-101,7){$r_1$}
        \put(-451,83){$1$}
        \put(-370,43){$w(r)$}
        \put(-182,70){$h(r)$}
    \end{picture}
    \caption{The graphs of $h$ and $w$.}\label{fig:handu}
\end{figure}
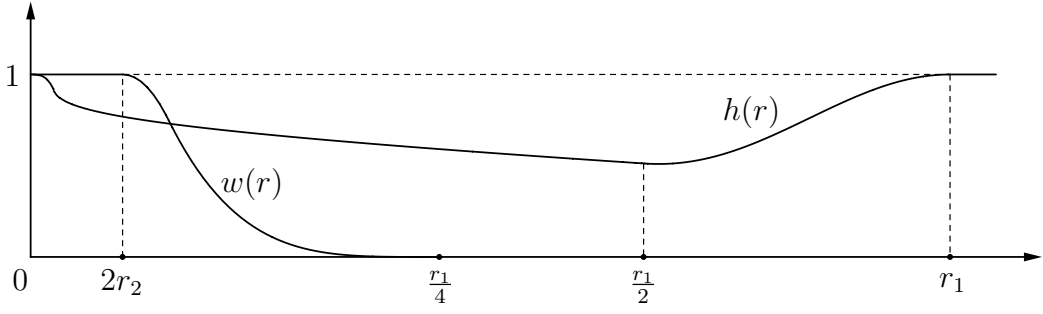

\begin{rmk}
    The prototype functions \eqref{eq-intro:KX23_main_func} appearing in the introduction can be extracted from \eqref{eq-cutoff:def_psi}\,$\sim$\,\eqref{eq-cutoff:hudefinition} by taking $\zeta=\eta\equiv1$.
\end{rmk}

Note that $h(r)\equiv1$ and $w(r)\equiv0$ when $r\geq r_1$, hence the metric $g'$ defined by \eqref{eq-main:new_g} smoothly concatenates with the original metric $g$. We will verify that, for a certain choice of the parameters, the functions defined here satisfy all the statements of Theorem \ref{thm-main:h_and_u}.

\vspace{6pt}

The following technical lemma is useful.

\begin{lemma}\label{lemma-cutoff:psi}
    Suppose $r_2<r_1$ and consider the function $\psi$ defined by \eqref{eq-cutoff:def_psi}. Then $\psi(r)\leq\frac12$ for all $r\leq r_1$.
\end{lemma}
\begin{proof}
    Estimating the cutoff functions by 1 and explicitly integrating, we have
    \[\psi(r)\leq\int_0^r\frac{d\rho}{\rho\log^2(1/\rho)}+\int_0^{\min(r,r_2)}\frac \rho{r_2}\,d\rho\leq\frac1{\log(1/r_1)}+\frac{r_2}2\leq\frac12. \qedhere
    \]
\end{proof}

Based on the choice of $r_1$ in Lemma \ref{lemma-cutoff:r1}, we determine the remaining parameters $c_1,c_2,r_2$ according to the following lemma:
\begin{lemma}\label{lemma-cutoff:c1c2r2}
    There exists 
    $c_1,c_2>0$ and $r_2\in(0,r_1)$, depending on $r_1,v_0,\epsilon,C_1,\cdots,C_6,R_0$, such that:
    \begin{enumerate}[label=(\roman*)]
        \item the function $h$ defined by \eqref{eq-cutoff:def_h} satisfies
        \begin{equation}\label{eq-cutoff:aux3}
            \Big(-\frac{h_{rr}}{h}-\frac{2h_r}{rh}\Big)\geq-\epsilon/3,\ \ \ 
            C_1|h_r|\leq\epsilon/3,\ \ \ 
            C_3r^{-1}\big|1-h\big|\leq\epsilon/3
        \end{equation}
        on the interval $[r_1/4,r_1]$,
        
        \item the constants $c_1,c_2$ satisfy $c_1<\min\big\{r_1,1/100\big\}$ and
        \begin{equation}
            c_2\cdot\big(1+\sup_\Sigma|v_0|\big)\leq\min\Big\{\sqrt{\frac{c_1}{\alpha_n}},\ \frac{c_1}{2p},\ \frac1{2(n+p)},\ \frac{r_1}p\Big\}.
        \end{equation}
        
        \item the radius $r_2$ satisfies 
        \begin{equation}
            r_2<\min\Big\{\frac{r_1}{64},\,
            \frac{c_1}{C_6}e^{-2(n+p)\sup_\Sigma|v_0|},\,
            \frac{c_1e^{-2p\sup_\Sigma|v_0|}}{200(R_0+1)}\Big\}
        \end{equation}
        and
        \begin{equation}\label{eq-cutoff:w0=1}
            c_2\int_0^\infty\zeta\big(\frac \rho{4r_2}\big)\eta\big(\frac{4\rho}{r_1}\big)\frac{d\rho}{\rho\log(1/\rho)}=1.
        \end{equation}
    \end{enumerate}
\end{lemma}

\begin{proof}
    \TODO{be careful about arxiv detection: copy pasted from previous paper:}
    We first establish some preliminary observations which are used below. Lemma \ref{lemma-cutoff:psi} states that $\psi\leq\frac12$ on $[0,r_1]$. By taking $c_1<1$, this implies $h\geq\frac12$. Next, we note that by taking $r_2<\frac1{64}r_1$, we have $\zeta(r/r_2)\equiv1$ on $[r_1/4,r_1]$. This implies that  $\psi'(r)=\frac1{r\log^2(1/r)}$ for $r\in[r_1/4,r_1]$. In particular, $\psi'$ and $\psi''$ are independent of $r_2$.
    
    To establish item (1) of the lemma we directly compute the derivatives of $h(r)=1-c_1\eta(r/r_1)\psi(r)$ in $[r_1/4,r_1]$:
    \begin{align}
        &h_r=-c_1r_1^{-1}\eta'\big(\frac r{r_1}\big)\psi(r)-c_1\eta\big(\frac r{r_1}\big)\psi'(r),\\
        &h_{rr}=-c_1r_1^{-2}\eta''\big(\frac r{r_1}\big)\psi(r)
            -2c_1r_1^{-1}\eta'\big(\frac r{r_1}\big)\psi'(r)
            -c_1\eta\big(\frac r{r_1}\big)\psi''(r).
    \end{align}
    From this and \eqref{eq-cutoff:etazeta}, we obtain the bounds
    \begin{align}
        &|h_r|\leq 2c_1r_1^{-1}+c_1\psi'(r)\\
        &|h_{rr}|\leq 8c_1r_1^{-2}+8c_1r_1^{-1}\psi'(r)+c_1|\psi''(r)|.
    \end{align}
    This leads to the inequality
    \begin{align}
        -\frac{h_{rr}}{h}-\frac{2h_r}{rh} &\geq -2|h_{rr}|-4r^{-1}|h_r| \\
        &\geq -2\Big[8c_1r_1^{-2}+8c_1r_1^{-1}\psi'(r)+c_1|\psi''(r)|\Big]
            -\frac4r\Big[2c_1r_1^{-1}+c_1\psi'(r)\Big].\label{eq-cutoff:aux2}
    \end{align}
    The right hand side of \eqref{eq-cutoff:aux2} depends only on $r,c_1,r_1$ but not on $r_2$, as long as $r_2<\frac1{64}r_1$ and $r\in[\frac{r_1}4,r_1]$. As $r_1$ are fixed, only $c_1,r$ enter as variables. Thus, we use $F(c_1,r)$ to denote the right hand side of \eqref{eq-cutoff:aux2}. By continuity, the function
    \begin{equation}
        G(c_1)=\min_{r\in[r_1/4,r_1]}F(c_1,r)
    \end{equation}
    is continuous for $c_1\in[0,1]$. Since $G(0)=0$, for all sufficiently small $c_1$ (depending only on $r_1$ and $\epsilon$), we have $c_1<\min\{r_1,1/100\}$ and $G(c_1)\geq-\epsilon/3$. Turning our attention to the other two quantities in item (i), notice that
    \[
    C_1|h_r|\leq 2C_1c_1\big[2r_1^{-1}+\psi'(r)\big],\qquad
    C_3r^{-1}|1-h|\leq C_3r^{-1}c_1\psi(r).
    \]
    Note that the right side of these inequalities are uniformly continuous (for $r\in[\frac{r_1}4,r_1]$) and vanish when $c_1=0$. We may therefore use the same argument as above, to find that there is a (perhaps smaller) choice $c_1<\min\{r_1,1/100\}$ such that \eqref{eq-cutoff:aux3} holds. We fix such a choice. This establishes the bounds in item (i) of the lemma so long as $r_2<\frac1{64}r_1$ holds.

    The remaining task is to choose $c_2$ and $r_2$ which satisfy items (ii) and (iii) of the lemma. Once this is achieved, we obtain $r_2<\frac1{64}r_1$, which ensures item (i). Set the constant
    \[r^*=\min\Big\{\frac{r_1}{64},\,\frac{c_1}{C_6}e^{-2(n+p)\sup|v_0|},\,\frac{c_1e^{-2p\sup_\Sigma|v_0|}}{200\max\{1,R_0\}}\Big\},\]
    which will be an upper bound for $r_2$, ensuring the first condition of item (iii). We note that there exists a choice of $c_2>0$ (depending on $c_1, r_1, n$) such that both item (ii) is satisfied and that
    \begin{equation}\label{eq-cutoff:aux4}
        c_2\Big[\log\log\frac1{r^*}-\log\log\frac1{r_1}\Big]<1.
    \end{equation}
    We fix this choice of $c_2$.
    
    Now consider the continuous function
    \begin{equation}
        I(r_2)=c_2\int_0^\infty\zeta\big(\frac \rho{4r_2}\big)\eta\big(\frac{4\rho}{r_1}\big)\frac{d\rho}{\rho\log(1/\rho)},\qquad r_2\in(0,r^*].
    \end{equation}
    By \eqref{eq-cutoff:aux4}, we have
    \begin{equation}
        I(r^*)\leq c_2\int_{r^*}^{r_1}\frac{d\rho}{\rho\log(1/\rho)}=c_2\Big[\log\log\frac1\rho\Big]_{r^*}^{r_1}<1.
    \end{equation}
    On the other hand, we have
    \begin{equation}
    \liminf_{r_2\to0^+}I(r_2)\geq c_2\liminf_{r_2\to0^+}\int_{4r_2}^{r_1/8}\frac{d\rho}{\rho\log(1/\rho)}=+\infty.
    \end{equation}
    Combining these last two inequalities, we may choose $r_2\in(0,r^*)$ so that $I(r_2)=1$. This establishes item (iii).
\end{proof}

Let us summarize:

\begin{lemma}[properties of $h, u$]\label{lemma-cutoff:misc}
    Consider functions $h,w,u$ constructed above, and assume the choice of $r_1,r_2,c_1,c_2$ given by Lemmas \ref{lemma-cutoff:r1} and \ref{lemma-cutoff:c1c2r2}. Then we have:
\begin{enumerate}[label=(\roman*), topsep=3pt, itemsep=0ex]
    \item $1-\frac{r_1}2\leq h\leq 1$ and $u\leq0$.
    \item $w(r)$ is non-increasing with $w(0)=1$, and is constant on $[0,2r_2]$. On $[0,r_1]$ we have
    \[0\leq w(r)\leq 1,\qquad w(r)\leq c_2\Big[\log\log\frac1r-\log\log\frac1{r_1}\Big].\]
    In particular, it holds $w(r)\leq c_2\log\log\frac1r$.
    \item On $[0,r_1]$ we have $e^{pu}\geq\log(\frac1r)^{-r_1}$. In particular, both $e^{2pu}\geq\log(\frac1r)^{-1}$ and $re^{-2nu}\leq1$ hold on $[0,r_1]$.
    \item The metric $g'$ in General Setup \ref{setup-main:general} constructed from $u$ and $h$ is smooth accross $\Sigma$.
\end{enumerate}
\end{lemma}
\begin{proof}
    Item (i) follows directly from Lemma \ref{lemma-cutoff:psi} and $c_1\leq r_1$ from Lemma \ref{lemma-cutoff:c1c2r2}.
    
    For item (ii): Notice that the integrand in the definition of $w$ is supported in $[2r_2,r_1/4]$. In particular, $w(r)$ is constant on $[0,2r_2]$. Due to the smallness of $r_1$ in Lemma \ref{lemma-cutoff:r1}, we know $\log(1/r)>0$ on the support of the integrand, and so $w$ is nonnegative. Similarly, computing $w_r$ using the Fundamental Theorem of Calculus, we find $w$ is non-increasing. The fact $w(0)=1$ follows from \eqref{eq-cutoff:w0=1}.
    
    Next, a simple computation shows
        \begin{equation}\label{eq-cutoff:u_bound}
            w(r)\leq c_2\int_r^{r_1}\frac{d\rho}{\rho\log(1/\rho)}=c_2\Big[\log\log\frac1r-\log\log\frac1{r_1}\Big].
        \end{equation}
    The second term has the favorable sign (since $r_1<1/100$), and so $w(r)\leq c_2\log\log\frac1r$.
    
    For item (iii): Using the condition $c_2\sup_\Sigma|v_0|\leq r_1/p$ from Lemma \ref{lemma-cutoff:c1c2r2}(ii), we obtain
    \[e^{pu} = e^{-p|v_0|w}\geq \log\big(\frac1r\big)^{-pc_2\sup_\Sigma|v_0|}\geq\log\big(\frac1r\big)^{-r_1}.\]
    Next, The remaining inequalities follow from $r_1<1/2$ and $\log(\frac1r)^{2nr_1/p}<\log(\frac1r)<\frac1r$.

    For item (iv): since $u_r=0$ for sufficiently small $r$, it follows that $u$ is smooth across $\Sigma$. Recall that $g'=e^{-2pu}dr^2+e^{-2pu}h^2\omega_\th^2+e^{2u}g_{\H}$, which can be rewritten as
    \[g'=e^{-2pu}g_{\V}+e^{2u}g_{\H}+e^{-2pu}(1-h^2)\omega_\th^2,\]
    where $g_\V, g_\H$ are the restriction of $g$ to the distributions $\V,\H$. Integrating \eqref{eq-cutoff:def_h} we find $h(r)=1-c_1\frac{r^2}{2r_2}$ for all $0\leq r\leq\frac12r_2$. Since $r\omega_\th$ is smooth across $\Sigma$ by Lemma \ref{lemma-exp:E}, it follows that $(1-h^2)\omega_\th^2$ is smooth across $\Sigma$ as well. This shows that $g'$ is smooth.
\end{proof}

\subsection{Proof of Theorem \ref{thm-main:h_and_u}}

Now that we have constructed $h$, $u$, we are ready to begin the main proof of this section.

\begin{proof}[Proof of Theorem \ref{thm-main:h_and_u}] {\ }

    Let $h, u$ be the functions defined in \eqref{eq-cutoff:def_h} and \eqref{eq-cutoff:def_u}, with the parameters $r_1,r_2,c_1,c_2$ chosen as in Lemmas \ref{lemma-cutoff:r1} and \ref{lemma-cutoff:c1c2r2}. As always, let the metric $g'$ be defined by the formula \eqref{eq-main:new_g}. By the form of $h,u$,  Lemma \ref{lemma-cutoff:misc}, and the facts $h=1$ and $u=0$ outside $[0,r_1]$, we see that $h, u$ satisfy Condition \ref{cond-main:h_and_u}. Since $w(0)=1$ and $0\leq w\leq1$, one sees that $u|_\Sigma=v_0$ and $\sup|u|\leq\sup_\Sigma|v_0|$. This verifies item (I) in the main statement. For item (II), the smoothness of $g'$ follows from Lemma \ref{lemma-cutoff:misc}(iii), and $g'=g$ outside $N(\Sigma,r_1)$ follows from $h\equiv1, w\equiv0$ for $r\geq r_1$. Item (VII) follows from Lemma \ref{lemma-cutoff:misc}(i)(ii).
    
    The remaining items of the main statement are shown below.

    \vspace{6pt}
    \noindent\textbf{The scalar curvature lower bound.}
    
    Inserting $u=v_0w$ into the scalar curvature estimate \eqref{eq-main:scal_final} in Theorem \ref{thm-main:scal_estimate}, we obtain
    \begin{equation}\label{eq-cutoff:reduced_scal}
        \begin{aligned}
            R_{g'} &\geq e^{2pu}R_g+2e^{2pu}\Big[-\frac{h_{rr}}{h}-\frac{2h_r}{rh}-\alpha_nv_0^2(w_r)^2\Big] \\
            &\qquad - C_1e^{2pu}|h_r|-C_2e^{2pu}|v_0|\cdot|w_r|-C_3\frac{e^{2pu}}r\big|1-h\big|-C_4e^{-2nu}\big|1-e^{(p+1)u}\big| \\
            &\qquad -C_5e^{-2u}|w|\big(1+||v_0||_{C^2(\Sigma)}^2\big).
        \end{aligned}
    \end{equation}
    We aim to show that this is no less than $R_g-\epsilon$, thus verifying item (III).

    First consider the simpler case when $r\in[\frac14r_1,r_1]$. For such $r$, we have $w(r)=0$ and so $u(r)=0$. This greatly simplifies \eqref{eq-cutoff:reduced_scal}, and we may directly apply Lemma \ref{lemma-cutoff:c1c2r2}(i) to find
    \begin{equation}\label{eq-cutoff:scalboundbigr}
        R_{g'}\geq R_g+2\Big[-\frac{h_{rr}}{h}-\frac{2h_r}{rh}\Big]-C_1|h_r|-C_3r^{-1}|1-h|\geq R_g-\epsilon.
    \end{equation}
    
    In what follows, we suppose $r\in[0,\frac14r_1]$. On this interval we have $\eta(r/r_1)\equiv1$, which simplifies the computation for $h$. For convenience, we set
    \begin{equation}
        s:=\log(1/r)
    \end{equation}
    and record the following direct computations
    \begin{align}\label{eq-cutoff:dh}
        h_r &= -c_1\zeta\big(\frac r{r_2}\big)\frac1{rs^2}
        -c_1\big(1-\zeta\big(\frac r{r_2}\big)\big)\frac r{r_2}\\
        \begin{split}\label{eq-cutoff:ddh}
        h_{rr} &= -\frac{c_1}{r_2}\zeta'\big(\frac r{r_2}\big)\frac1{rs^2}
                +c_1\zeta\big(\frac r{r_2}\big)\frac1{r^2s^2}
                -2c_1\zeta\big(\frac r{r_2}\big)\frac1{r^2s^3} \\
                &\qquad +\frac{c_1}{r_2}\zeta'\big(\frac r{r_2}\big)\frac r{r_2}
                -c_1\big(1-\zeta\big(\frac r{r_2}\big)\big)\frac 1{r_2}.
        \end{split}
    \end{align}
    Also note that
    \begin{equation}\label{eq-cutoff:du}
        |w_r|=\frac{c_2}{rs}\zeta\big(\frac r{4r_2}\big)\eta\big(\frac{4r}{r_1}\big)\leq\frac{c_2}{rs}\zeta\big(\frac r{r_2}\big).
    \end{equation}
    
    Throughout the remainder of the proof, we adopt the shorthand $\zeta=\zeta(r/r_2)$ and $\zeta'=\zeta'(r/r_2)$. Combining \eqref{eq-cutoff:dh}, \eqref{eq-cutoff:ddh}, and \eqref{eq-cutoff:du} we estimate the first line of \eqref{eq-cutoff:reduced_scal}:
    \begin{align}
        -\frac{h_{rr}}h-\frac{2h_r}{rh}&-\alpha_nv_0^2w_r^2 \nonumber\\
        \begin{split}\geq&\, \frac1{1-c_1\psi}\Big[
            \frac{c_1}{r_2}\zeta'\frac1{rs^2}
            -c_1\zeta\frac1{r^2s^2}
            +2c_1\zeta\frac1{r^2s^3}
            -\frac{c_1}{r_2}\zeta'\frac r{r_2}
            +c_1(1-\zeta)\frac1{r_2}\Big] \\
        &\quad +\frac{2}{r(1-c_1\psi)}\Big[c_1\zeta\frac1{rs^2}+c_1(1-\zeta)\frac r{r_2}\Big]-\frac{\alpha_nv_0^2c_2^2}{r^2s^2}\zeta^2 
        \end{split}  \\ 
        =&\, \frac{c_1}{1-c_1\psi}\Big[\zeta'\cdot(\frac1{r_2rs^2}-\frac r{r_2^2})+\frac{\zeta}{r^2s^2}+\frac{2\zeta}{r^2s^3}+3\frac{1-\zeta}{r_2}\Big]-\frac{\alpha_nv_0^2c_2^2}{r^2s^2}\zeta^2.\label{eq-cutoff:0to18p1}
    \end{align}
    To proceed, notice that Lemma \ref{lemma-cutoff:r1}(ii) implies $\frac1{rs^2}>1\geq\frac r{r_2}$ holds on $[0,r_2]$. Since $\zeta'\geq0$ and $\zeta$ is constant on $[r_2,r_1]$, it follows that $\zeta'\cdot(\frac1{r_2rs^2}-\frac r{r_2^2})\geq0$ holds everywhere on $[0,\frac14r_1]$. Using this observation, we continue the estimate \eqref{eq-cutoff:0to18p1}:
    \begin{align}
        -\frac{h_{rr}}h-\frac{2h_r}{rh}-\alpha_nv_0^2w_r^2
        &\geq \frac{c_1}{1-c_1\psi}\Big[\frac{\zeta}{r^2s^2}+\frac{2\zeta}{r^2s^3}+3\frac{1-\zeta}{r_2}\Big]-\frac{\alpha_nc_2^2v_0^2}{r^2s^2}\zeta^2 \\
        &\geq (c_1-\alpha_nc_2^2v_0^2)\frac{\zeta}{r^2s^2}+c_1\frac{2\zeta}{r^2s^3}+3c_1\frac{1-\zeta}{r_2} \label{eq-cutoff:aux1}\\
        {}&\geq c_1\frac{2\zeta}{r^2s^3}+3c_1\frac{1-\zeta}{r_2}, \label{eq-cutoff:interval1_term1}
    \end{align}
    where we used $\psi\geq0$ and $\zeta^2\leq\zeta$ in deriving \eqref{eq-cutoff:aux1}, and used $c_2\sup_\Sigma|v_0|\leq \sqrt{c_1/\alpha_n}$ from Lemma \ref{lemma-cutoff:c1c2r2}(ii) in deriving \eqref{eq-cutoff:interval1_term1}. As a result,
    \begin{equation}\label{eq-cutoff:main_term}
        2e^{2pu}\Big[-\frac{h_{rr}}h-\frac{2h_r}{rh}-\alpha_nu_r^2\Big]\geq c_1e^{2pu}\frac{4\zeta}{r^2s^3}+c_1e^{2pu}\frac{6(1-\zeta)}{r_2}.
    \end{equation}

    Our next task is to show that each of the fives terms in the last two lines in \eqref{eq-cutoff:reduced_scal} are dominated by the right hand side of \eqref{eq-cutoff:main_term}. For this step, the final result is \eqref{eq-cutoff:cleaned_up} below. For the $C_1$ term, we use \eqref{eq-cutoff:dh} to obtain
    \begin{align}
            C_1e^{2pu}|h_r| &\leq C_1e^{2pu}\cdot c_1\Big[\frac{\zeta}{rs^2}+(1-\zeta)\frac r{r_2}\Big] \\
            \label{eq-cutoff:bound_C1}&< c_1e^{2pu}\Big[\frac{\zeta}{2r^2s^3}+\frac{1-\zeta}{r_2}\Big],
    \end{align}
    where in the second line we used $4C_1<\frac1{rs}$ and $r<\frac1{2C_1}$ from Lemma \ref{lemma-cutoff:r1}(i)(ii). For the $C_2$ term, we apply \eqref{eq-cutoff:du} to estimate
    \begin{align}
        C_2e^{2pu}|v_0|\cdot|w_r| &\leq C_2e^{2pu}\cdot\sup_{\Sigma}|v_0|\cdot\frac{c_2\zeta}{rs}\\
        \label{eq-cutoff:bound_C2}&< c_1e^{2pu}\frac{\zeta}{2r^2s^3},
    \end{align}
    where we used $c_2\leq c_1$ from Lemma \ref{lemma-cutoff:c1c2r2}(ii) and $C_2\sup_\Sigma|v_0|\leq\frac1{2rs^2}$ from Lemma \ref{lemma-cutoff:r1}(ii).
    
    Next, we consider the $C_3$ term of \eqref{eq-cutoff:reduced_scal}
    \begin{equation}\label{eq-ds:C3_1}
        C_3\frac{e^{2pu}}r\big|1-h\big|= C_3\frac{e^{2pu}}rc_1\psi.
    \end{equation}
    To proceed, we further estimate $\psi$ using its definition, recalling our abbreviation $\zeta=\zeta(r/r_2)$:
    \begin{align}
        \psi(r) &\leq \int_0^r\frac{\zeta(\rho/r_2)}{\rho\log^2(1/\rho)}d\rho+\int_0^{\min\{r,r_2\}}\frac{\rho}{r_2}\,d\rho&&{} \nonumber\\
        &\leq \int_0^r\frac{\zeta(r/r_2)}{\rho\log^2(1/\rho)}d\rho+\int_0^{\min\{r,r_2\}}\frac{\rho}{r_2}\,d\rho
        \qquad&&\text{(since $\zeta$ is non-decreasing)} \nonumber\\
        &= \frac{\zeta}s+\frac{\min\{r,r_2\}^2}{2r_2}\qquad&&\text{(by direct integration).} \label{eq-ds:C3_2}
    \end{align}
    Combining \eqref{eq-ds:C3_1} and \eqref{eq-ds:C3_2} shows
    \begin{equation}\label{eq-ds:C3_1aux}
        C_3\frac{e^{2pu}}r\big|1-h\big|\leq C_3e^{2pu}c_1\frac{\zeta}{rs}+C_3e^{2pu}c_1\frac{\min\{r,r_2\}^2}{2r_2r}.
    \end{equation}
    The first term on the right hand side is bounded by $C_3e^{2pu}c_1\frac{\zeta}{rs}<c_1e^{2pu}\frac{\zeta}{2r^2s^3}$, according to Lemma \ref{lemma-cutoff:r1}(ii). For the second term on the right side of \eqref{eq-ds:C3_1aux}, we separately consider the case $\zeta<\frac12$ and $\zeta>\frac12$. When $\zeta<\frac12$, we have $r<r_2$ and so
    \[\begin{aligned}
        C_3\frac{\min\{r,r_2\}^2}{2r_2r}= C_3\frac{r}{2r_2}\leq C_3\frac r{r_2}(1-\zeta)<\frac{1-\zeta}{r_2},
    \end{aligned}\]
    since $r<1/C_3$. When $\zeta>\frac12$ we have
    \[\begin{aligned}
        C_3\frac{\min\{r,r_2\}^2}{2r_2r}\leq C_3\frac{r_2^2}{2r_2r}\leq C_3\frac{r_2}r\zeta<\frac{\zeta}{2r^2s^3},
    \end{aligned}\]
    since $r_2<1$ and $2C_3<\frac1{rs^3}$ by Lemma \ref{lemma-cutoff:r1}(ii). To summarize, we have obtained
    \begin{equation}\label{eq-cutoff:bound_C3}
        C_3\frac{e^{2pu}}r\big|1-h\big|\leq c_1e^{2pu}\frac{\zeta}{r^2s^3}+c_1e^{2pu}\frac{1-\zeta}{r_2}.
    \end{equation}
    
    Finally, we consider the terms of \eqref{eq-cutoff:reduced_scal} containing $C_4$ and $C_5$. Note that $|1-e^x|\leq|x|$ holds for $x\leq0$, and thus
    \[C_4e^{-2nu}\big|1-e^{(p+1)u}\big|\leq C_4e^{-2nu}(p+1)|u|\leq C_4e^{-2nu}(p+1)\sup_\Sigma|v_0|\cdot w.\]
    This can be combined with the term containing $C_5$ to give
    \[C_4e^{-2nu}|1-e^{(p+1)u}|+C_5e^{-2u}|w|\big(1+||v_0||_{C^2(\Sigma,g)}^2\big)\leq C_6e^{-2nu}w,\]
    where we have used the fact that $e^{-2u}\leq e^{-2nu}$ and we recall \eqref{eq-cutoff:def_C6} for the definition of $C_6$.
    To bound this quantity, we do a case discussion. When $\zeta<\frac12$, we use $w(r)\leq1$ to derive
    \[\begin{aligned}
        C_6e^{-2nu}w &\leq C_6e^{-2nu}
        \leq C_6e^{2pu}\cdot e^{2(n+p)\sup_\Sigma|v_0|}\cdot 2(1-\zeta) \\
        &< e^{2pu}\cdot c_1\frac{1-\zeta}{r_2}
        \qquad\text{(by Lemma \ref{lemma-cutoff:c1c2r2}(iii))}.
    \end{aligned}\]
    When $\zeta\geq\frac12$, we use $w(r)\leq c_2\log s$ from Lemma \ref{lemma-cutoff:misc}(ii) and $2(n+p)c_2\sup_\Sigma|v_0|\leq1$ from Lemma \ref{lemma-cutoff:c1c2r2}(ii) to obtain
    \[e^{-2(p+n)u}=e^{2(p+n)|v_0|w}\leq s^{2(n+p)c_2|v_0|}\leq s,\]
    and then use $w\leq c_2\log s\leq c_1s$ to find
    \[\begin{aligned}
        C_6e^{-2nu}w &\leq C_6\cdot e^{2pu}s\cdot c_1s\cdot 2\zeta\\
        &= c_1e^{2pu}\zeta\cdot 2C_6s^2 \\
        &\leq c_1e^{2pu}\frac{\zeta}{2r^2s^3}
        \qquad\text{(by Lemma \ref{lemma-cutoff:r1}(ii)).}
    \end{aligned}\]
    Combining both cases, we have found
    \begin{equation}\label{eq-cutoff:bound_C4C5}
        \begin{aligned}
            C_4e^{-2nu}|1-e^{(p+1)u}|+C_5e^{-2u}|w|\big(&1+||v_0||_{C^2(\Sigma,g)}^2\big) \\
            &\leq c_1e^{2pu}\frac{\zeta}{2r^2s^3}+c_1e^{2pu}\frac{1-\zeta}{r_2}.
        \end{aligned}
    \end{equation}
    
    Finally, by combining inequalities \eqref{eq-cutoff:main_term}, \eqref{eq-cutoff:bound_C1}, \eqref{eq-cutoff:bound_C2}, \eqref{eq-cutoff:bound_C3}, and \eqref{eq-cutoff:bound_C4C5} to the main scalar curvature lower bound \eqref{eq-cutoff:reduced_scal}, we obtain the cleaned-up inequality
    \begin{equation}\label{eq-cutoff:cleaned_up}
        R_{g'}\geq e^{2pu}R_g+c_1e^{2pu}\frac{\zeta}{r^2s^3}+c_1e^{2pu}\frac{1-\zeta}{r_2},\qquad r\in[0,\frac14r_1].
    \end{equation}
    Let us show that $R_{g'}\geq R_g$ in the region $\big\{r<r_1/4\big\}$. We divide into three cases:

    \vspace{6pt}

    {\noindent\bf{Case i.}} Assume $R_g(x)\leq0$. Then we drop the last two terms in \eqref{eq-cutoff:cleaned_up} and use $u\leq0$ to obtain
    \begin{equation}\label{eq-cutoff:interval1_result1}
        R_{g'}\geq e^{2pu}R_g\geq R_g.
    \end{equation}
    
    {\noindent\bf{Case ii.}} Assume $R_g(x)>0$ and $r=r(x)$ is such that $\zeta<\frac12$. Note that this implies $r<r_2$. Using $w(r)\leq1$ from Lemma \ref{lemma-cutoff:misc}(ii), we have $e^{2pu}\geq e^{-2p\sup_\Sigma|v_0|}$, therefore by dropping the first two terms in \eqref{eq-cutoff:cleaned_up} and using Lemma \ref{lemma-cutoff:c1c2r2}(iii), we obtain
    \begin{equation}
        R_{g'}\geq c_1e^{2pu}\frac{1-\zeta}{r_2}
        \geq c_1e^{-2p\sup_\Sigma|v_0|}\frac1{2r_2}
        \geq 100R_0>R_g. \label{eq-cutoff:interval1_result2}
    \end{equation}

    {\noindent\bf{Case iii.}} Assume $R_g(x)>0$ and $\zeta\geq\frac12$. Note the following two lower bounds for $e^{2pu}$: on one hand, we use the definition of $u$ and the bound on $w$ from Lemma \ref{lemma-cutoff:misc}(ii) to find
    \[e^{2pu}\geq e^{-2p\sup_\Sigma|v_0|\cdot w}
    \geq \Big(\frac{\log(1/r_1)}{s}\Big)^{2pc_2\sup_\Sigma|v_0|}
    \geq \Big(\frac{\log(1/r_1)}{s}\Big)^{c_1}\]
    where we used Lemma \ref{lemma-cutoff:c1c2r2}(ii) in the final inequality. On the other hand, we have $e^{2pu}\geq\frac1s$ directly from Lemma \ref{lemma-cutoff:misc}(iii). Applying these two inequalities in turn to the first two terms in \eqref{eq-cutoff:cleaned_up} and dropping the final term, we obtain
    \begin{align}
        R_{g'}&\geq R_g\Big(\frac{\log(1/r_1)}{s}\Big)^{c_1}+\frac{c_1}{2r^2s^4} \nonumber\\
        &\geq R_g\Big(\frac{\log(1/r_1)}{s}\Big)^{c_1}+\frac{R_gc_1}{r}\qquad\text{(by Lemma \ref{lemma-cutoff:r1}(ii))}.\label{eq-cutoff:interval1_result21}
    \end{align}
    
    To proceed, we need the observation that the function $q(r)=\big(\frac{\log(1/r_1)}{\log(1/r)}\big)^{c_1}+\frac{c_1}{r}$ is decreasing on $[0,r_1]$. Indeed, we use the fact that $r\mapsto\log(1/r)$ is decreasing to compute
    \begin{equation*}
        \begin{aligned}
            q'(r) &= c_1\Big(\frac{\log(1/r_1)}{\log(1/r)}\Big)^{c_1}\cdot\frac1{r\log(1/r)}-\frac{c_1}{r^2} \\
            &\leq c_1\cdot 1\cdot\frac1{r\log(1/r)}-\frac{c_1}{r^2} \\
            &= c_1\frac{r-\log(1/r)}{r^2\log(1/r)}<0,
        \end{aligned}
    \end{equation*}
    since $r<r_1<\frac{1}{100}$.
    Finally, we use this observation with \eqref{eq-cutoff:interval1_result21} to see
    \begin{equation}\label{eq-cutoff:interval1_result3}
        R_{g'}\geq R_gq(r)\geq R_gq(r_1)=R_g(1+\frac{c_1}{r_1})>R_g.
    \end{equation}

    Inspecting the inequalities \eqref{eq-cutoff:interval1_result1}, \eqref{eq-cutoff:interval1_result2}, and \eqref{eq-cutoff:interval1_result3}, we have established $R_{g'}\geq R_g$ for $r\in[0,\frac14r_1]$. In light of estimate \eqref{eq-cutoff:scalboundbigr} on the complimentary interval, this proves condition (III) of Theorem \ref{thm-main:h_and_u}.
    
    \vspace{6pt}
    
    It remains to verify conditions (IV), (V), and (VI) of Theorem \ref{thm-main:h_and_u}.
    
    \vspace{6pt}
    
    \noindent\textbf{The distance estimate.}
    Fix $x\in\p N(\Sigma,r_1)$, consider a radial path traveling from $x$ to its projection $\pi(x)\in \Sigma$, and proceed by estimating the length of such a path. We apply $e^{-pu}\leq \log(1/r)^{r_1}$ on $\{r\leq r_1\}$ from Lemma \ref{lemma-cutoff:misc}(iii) with an integration by parts to find
    \begin{equation}\label{eq-ds:dist1}
        \begin{aligned}
            \mathrm{dist}_{g'}(x,\pi(x))\leq\int_0^{r_1}e^{-pu}\,dr &\leq \int_0^{r_1}\log(\frac1r)^{r_1}\,dr \\
            &= r_1\log(\frac1{r_1})^{r_1}+r_1\int_0^{r_1}\log(\frac1r)^{r_1-1}\,dr.
        \end{aligned}
    \end{equation}
    Using the facts $(\log\frac1x)^x<2$ for all $x>0$ and $\log\frac1{r_1}>1$, we estimate the previous line to find
    \begin{equation}\label{eq-cutoff:distest}
        \mathrm{dist}_{g'}(x,\pi(x))\leq 2r_1+r_1^2\leq 3r_1.
    \end{equation}
    This establishes condition (IV).
    
    \vspace{6pt}
    
    \noindent\textbf{The mean convexity condition.} Recall from \eqref{eq-ds:formula_H} that $H'=e^{pu}(H+\frac{h_r}h)$. Meanwhile, recall that our choice of $r_I$ ensures $H\geq\frac1{2r}$. Using the expression \eqref{eq-cutoff:dh} and the facts $c_1\leq\frac1{100}$, $\psi\leq\frac12$, $s\geq1$, we obtain
    \[\begin{aligned}
        \frac{|h_r|}h &\leq 2\big(c_1\zeta\frac1{rs^2}+c_1(1-\zeta)\frac r{r_2}\big) \leq \frac1{8r}+1\leq\frac1{4r}.
    \end{aligned}\]
    By Lemma \ref{lemma-cutoff:misc}(iii) we have $e^{pu}\geq\frac{1}{\log(1/r)^{1/2}}\geq \frac{1}{\log(1/r)}$, and the mean convexity condition (V) follows.
    
    \vspace{6pt}
    
    \noindent\textbf{The volume estimate.} By the metric expression \eqref{eq-main:new_g}, the volume is expressed as
    \[\Vol_{g'}(N(\Sigma,r_1))=\int_{N(\Sigma,r_1)}e^{-pu}h\,dV_g.\]
    Using the same bound for $e^{-pu}$ as we used in the distance estimate, we have
    \begin{align}
        \int_{N(\Sigma,r_1)}e^{-pu}h\,dV_g
        &\leq \int_{N(\Sigma,r_1)}\log(\frac1r)^{pc_2\sup_\Sigma|v_0|}\,dV_g&&\\
        &\leq \int_{N(\Sigma,r_1)}\log(\frac1r)^{r_1}\,dV_g
        &&\text{(by Lemma \ref{lemma-cutoff:c1c2r2}(ii))}\\
        &= \int_0^{r_1}\log(\frac1r)^{r_1}|\Sigma_r|dr &&\\
        &\leq 4\pi r_1|\Sigma|_g\int_0^{r_1}\log(\frac1r)^{r_1}dr&&\text{(by our choice of $r_I$)}.
    \end{align}
    Integrating by parts as we did in the distance estimate \eqref{eq-ds:dist1} yields the desired volume bound $\Vol_{g'}(N(\Sigma,r_1))\leq 12\pi|\Sigma|_gr_1^2$ in condition (VI).
\end{proof}


\section{Applications}\label{sec:app}

Now that the primary construction has been made and its properties established, we are ready to discuss its consequences.

\subsection{Collapsing of subsets}

We begin with Theorems \ref{thm-intro:collapse_submfd} and \ref{thm-intro:partial_collapse} on collapsing sequences with controlled scalar curvature. Recall that Theorem \ref{thm-intro:collapse_submfd} describes metrics collapsing a submanifold $K\subset M$ of arbitrary codimension.

\begin{proof}[Proof of Theorem \ref{thm-intro:collapse_submfd}] {\ }

    Fix $\delta\ll1$. We first claim that there exists a connected embedded codimension 2 submanifold $\Sigma$, such that:
    
    (i) $\Sigma\subset N_g(K,\delta/2)$,

    (ii) for any $x\in K$, there exists $z\in\Sigma$ with $d(x,z)\leq\delta$.

    \noindent To construct $\Sigma$, we first find a smoothly embedded contractible loop $\gamma:S^1\to N_g(K,\delta/4)$ so that for any $x\in K$ there exists $z\in\gamma$ with $d(x,z)\leq\delta/2$. If $n=3$ then the choice $\Sigma=\gamma$ already satisfies the required conditions. Now we assume $n\geq4$. Since the normal bundle of $\gamma$ is trivial, we may consider a normal framing $e_1,\cdots,e_{n-1}$ of $\gamma$. Then for sufficiently small $r<\delta/2$, the submanifold
    \[\Sigma=\Big\{\exp_{\gamma(t)}\big(a_1e_1+\cdots+a_{n-2}e_{n-2}\big): t\in S^1, a_1^2+\cdots+a_{n-2}^2=r^2\Big\}\]
    is a connected smooth submanifold which satisfies our requirements.

    Then we use Theorem \ref{thm-main:h_and_u} to produce a drawstring along $\Sigma$ to produce a metric $g_\delta$ on $M$ with $R_{g_\delta}\geq R_g-\delta$. In our application of Theorem \ref{thm-main:h_and_u}, we choose $v_0$ sufficiently negative so that the (intrinsic) diameter of $(\Sigma,e^{2v_0}g|_\Sigma)$ is less than $\delta/2$, and $\epsilon<\delta/4$ sufficiently small so that
    \begin{equation}\label{eq-app:aux1}
        \Vol_{g_\delta}(N_g(K,\delta))\leq(1+\delta)\Vol_g(N_g(K,\delta)).
    \end{equation}
    To see such a choice exists, we note by Theorem \ref{thm-main:h_and_u}(VI) that there exists $r_1<\epsilon$ so that $\Vol_{g_\delta}(N_g(\Sigma,r_1))\leq 12\pi r_1^2\Area_g(\Sigma)$. Thus
    \[\begin{aligned}
        \Vol_{g_\delta}(N_g(K,\delta)) &= \Vol_{g_\delta}(N_g(\Sigma,\epsilon))+\Vol_{g_\delta}(N_g(K,\delta)\setminus N_g(\Sigma,\epsilon)) \\
        &\leq 12\pi\epsilon^2|\Sigma|_g+\Vol_g(N_g(K,\delta)),
    \end{aligned}\]
    where we have used the fact that $g_\delta=g$ away from $N_g(\Sigma,r_1)$. Thus \eqref{eq-app:aux1} holds by choosing $\epsilon\ll1$.
    Moreover, using the size of $v_0$ and Theorem \ref{thm-main:h_and_u}(IV), it follows that the diameter of $N_g(K,\delta)$ with respect to $g_\delta$ is less than $6\delta$.
    
    Having verified $\diam_{g_\delta}(N_g(K,\delta))\to0$ and \eqref{eq-app:aux1}, we may apply the scrunching theorem \cite[Theorem 2.5]{BS} to conclude that $(M,g_\delta)\to (M,g)/\{K\sim\pt\}$ in the Intrinsic-Flat and Gromov-Hausdorff senses for a sequence $\delta\to0$. One may check that the proof of \cite{BS} (see also \cite[Lemmas 5.4, 5.5, 5.7]{BDS}) carries over to all dimensions without modification.
\end{proof}

Next up is Theorem \ref{thm-intro:partial_collapse} on sequences which $c$-partially collapse codimension-2 submanifolds $\Sigma\subset M$.

\begin{proof}[Proof of Theorem \ref{thm-intro:partial_collapse}] {\ }

    Let $\Sigma^{n-2}$ and $c$ be as in the theorem. We consider the case where $c<+\infty$, which occupies the majority of the proof. For a sequence of numbers $\delta\to0$, we apply Theorem \ref{thm-main:h_and_u} with the choice $v_0\equiv -c$ and $\epsilon=\delta$, to obtain a metric $g_\delta$. The scalar curvature condition is clear from Theorem \ref{thm-main:h_and_u}(V). It remains to show the uniform minA lower bound and the Gromov-Hausdorff convergence.

    \vspace{6pt}
    \noindent\textbf{MinA lower bound.}
    
    Recall the choice of $r_I$ in Section \ref{sec:main}, which ensures that the equidistant surfaces $\Sigma_r:=\big\{d_g(\cdot,\Sigma)=r\big\}$ have mean curvature at least $1/(2r)$ for $r\leq 2r_I$. In combination with Theorem \ref{thm-main:h_and_u}(VI), it follows that the region $N_g(\Sigma,2r_I)$ is foliated by strictly mean convex hypersurfaces. Now let $S\subset (M^n,g_\delta)$ be a closed embedded minimal hypersurface. By the maximum principle, $\Sigma$ cannot be contained in $N_g(\Sigma,2r_I)$. Thus, $\Sigma$ exits this region and we can find a point $x\in S\setminus N_g(\Sigma,2r_I)$. Observe that $g_\delta=g$ in $B(x,r_I)$. Since the geometry of this region is independant of $\delta$, we may apply the monotonicity formula \cite[(7.5)]{CM} and conclude that $|S\cap B(x,r_I)|$ has a lower bound depending only on $r_I$ and the geometry of $g$.

    \vspace{6pt}
    \noindent\textbf{Gromov-Hausdorff convergence.}

    We will in fact show that for some sequence $\delta\to0$, we have $d_{g_\delta}\to d_c$ in the $C^0$ sense.

    First we show that $d_c\geq d_{g_\delta}+o(1)$. On $\Sigma$ consider the tensor $g_\Sigma$ which equals $g$ in the tangential directions and equals zero in the normal directions. Define the $L^\infty$ Riemannian metric
    \[g_\infty=\left\{\begin{aligned}
         &g&&\qquad\text{outside $\Sigma$}, \\
         &e^{-2c}g_\Sigma+e^{2(n-2)c}(g-g_\Sigma)&&\qquad\text{on $\Sigma$.}
    \end{aligned}\right.\]
    
    From the definition of $d_c$ \eqref{eq-intro:def_dc} and Theorem \ref{thm-main:h_and_u}(VII) (from which we have $\frac12e^{-2c}g\leq g_\delta\leq e^{2(n-2)c}g$), $d_c$ and $\{d_{g_\delta}\}_{\delta>0}$ are uniformly bi-Lipschitz with respect to $d_g$. Hence $\{d_c,d_{g_\delta}\}_{\delta>0}$ is an equi-continuous family with respect to $d_g$. By the Arzela-Ascoli Theorem, $\{d_{g_\delta}\}$ has a subsequence which $C^0$-converges to a limiting function $F$.
    
    According to properties Theorem \ref{thm-main:h_and_u}(I) and (II), there is a pointwise convergence $g_\delta\to g_\infty$ as $\delta\to0$. As a consequence, for each $x,y\in M$ we may estimate
    \[\begin{aligned}
        d_c(x,y) &= \inf\Big\{|\gamma|_{g_\infty}: \gamma\text{ is a piecewise smooth path joining $x,y$}\Big\} \\
        &= \inf\Big\{\lim_{\delta\to0}|\gamma|_{g_\delta}: \gamma\text{ is a piecewise smooth path joining $x,y$}\Big\} \\
        &\geq\limsup_{\delta\to0} \inf\Big\{|\gamma|_{g_\delta}: \gamma\text{ is a piecewise smooth path joining $x,y$}\Big\} \\
        &= \limsup_{\delta\to0}d_{g_\delta}(x,y).
    \end{aligned}\]
    It follows that $d_c\geq F$. Hence we have $d_c\geq d_{g_\delta}-o(1)$ for our selected subsequence.

    Next we show $d_c\leq d_{g_\delta}+o(1)$. For $\delta\ll1$, let $\tilde d_\delta$ be the distance obtained by $c$-partially pulling $\bar{N_g(\Sigma,\delta)}$. By the form of $g_\delta$ (see \eqref{eq-main:new_g}) and Theorem \ref{thm-main:h_and_u}(VII), note that
    \[
    d_{g_\delta}\geq (1-o(1))\tilde d_\delta.
    \]
    It suffices to show that $\tilde d_\delta\geq(1-o(1))d_c$. Consider the map $f: M\to M$ defined so that
    
    (i) $f|_{N_g(\Sigma,\delta)}=\pi$, the previously defined orthogonal projection map,
    
    (ii) $f|_{M\setminus N_g(\Sigma,r_I)}=\text{id}$,
    
    (iii) on $N_g(\Sigma,r_I)\setminus N_g(\Sigma,\delta)$ we set
    \[
        f\big(\exp_y(t\nu)\big)=\exp_y\Big(\frac{t-\delta}{r_I-\delta}\cdot r_I\nu\Big),
    \]
    where $\delta\leq t\leq r_I$ and $\nu$ is a unit normal vector of $\Sigma$. It can be verified that $f$ is a $(1+C\delta)$-Lipschitz map, where $C$ is a constant depending only on the geometry of $\Sigma$. By Lemma \ref{lemma-app:Lipschitz_to_cpull} below, $f$ is a $(1+C\delta)$-Lipschitz map from $(M,\tilde d_\delta)$ to $(M,d_c)$. This implies $\tilde d_\delta\geq(1-o(1))d_c$.
\end{proof}

\begin{lemma}\label{lemma-app:Lipschitz_to_cpull}
    Let $X$ be a length space, and $B\subset A$ be two compact sets of $X$. Denote by $d_{c,A}$ and $d_{c,B}$ the metrics obtained by $c$-partially pulling $A$ and $B$, respectively. Suppose $f: X\to X$ is an $L$-Lipschitz map, such that $f(A)\subset B$ and $f|_B=\text{id}$. Then $f$ is $L$-Lipschitz from $(X,d_{c,A})$ to $(X,d_{c,B})$.
\end{lemma}
\begin{proof}
    We let $d_A(p,q)$ (resp. $d_B(p,q)$) denote the infimum of lengths of paths within $A$ (resp. $B$) from $p$ to $q$ and equal to $\infty$ if no such path exists. According to the definition of the $c$-partially pulled metrics, for any $x,y\in X$ and $\epsilon>0$ there are points $p_i,q_i\in A$ ($1\leq i\leq N$) such that
    \[d_X(x,p_1)+e^{-c}d_A(p_1,q_1)+d_M(q_1,p_2)+e^{-c}d_A(p_2,q_2)+\cdots+d_X(q_N,y)<(1+\epsilon)d_{c,A}(x,y).\]
    Since $f(A)\subset B$, we have
    \[d_B\big(f(p_i),f(q_i)\big)\leq Ld_A(p_i,q_i).\]
    This implies
    \[\begin{aligned}
        d_{c,B}\big(f(x),f(y)\big) &\leq d_X\big(f(x),f(p_1)\big)+e^{-c}d_B\big(f(p_1),f(q_1)\big) \\
        &\hspace{60pt} +d_X\big(f(q_1),f(p_2)\big)+\cdots+d_X\big(f(q_N),f(y)\big) \\
        &\leq Ld_X(x,p_1)+e^{-c}Ld_A(p_1,q_1)+\cdots+Ld_X(q_N,y) \\
        &\leq (1+\epsilon)d_{c,A}(x,y).
    \end{aligned}\]
    Taking $\epsilon\to0$ proves the lemma.
\end{proof}

\subsection{Realizing arbitrary conformal limits}

We move on to Theorem \ref{thm-intro:leetopping}, showing that metrics in a Yamabe-positive conformal class can be approximated (in the uniform convergence topology) by metrics of positive scalar curvature.

\begin{proof}[Proof of Theorem \ref{thm-intro:leetopping}] {\ }

    The proof is done by putting drawstrings along an increasingly dense collection of closed curves. On each curve, the function $v_0$ in Theorem \ref{thm-main:h_and_u} is set to be equal to $f$.

    Suppose $\epsilon_0,L$ are constants so that
    \[R_g\geq\lambda+\epsilon_0,\qquad |f|\leq L,\qquad |\D f|_g\leq L\]
    everywhere on $M$. Let $r_0$ be the injectivity radius of $M$. Set $\epsilon_1<\min\big\{\epsilon_0/2,r_0/4,1/100\big\}$. Define $X:=\{(x,y): d_g(x,y)=2\epsilon_1\}\subset M\times M$. Let $Y$ be a finite $\epsilon_1^2$-dense set in $X$. Then it follows that: for all $(x,y)\in X$ there exists $(p,q)\in Y$ such that $d_g(x,p)<\epsilon_1^2$ and $d_g(y,q)<\epsilon_1^2$. For each $(p,q)\in Y$ we associate a $g$-geodesic segment $\sigma_{pq}:(-\epsilon_1,\epsilon_1)\to M$, such that $\sigma_{pq}(-\epsilon_1)=p$ and $\sigma_{pq}(\epsilon_1)=q$. Through a small perturbation of $Y$, we may assume that the closures of $\sigma_{pq}$ are pairwise disjoint. Finally, we extend each $\sigma_{pq}$ to a smooth closed curve $\Sigma_{pq}\supset\sigma_{pq}$ in the way that $\{\Sigma_{pq}\}_{(p,q)\in Y}$ are disjoint.

    Choose a radius $r_2$ such that:

    (1) the neighborhoods $\big\{N(\Sigma_{pq},2r_2)\big\}_{(p,q)\in Y}$ are pairwise disjoint,

    (2) For each $(p,q)\in Y$, the radius $r_2$ is less than the radius $r_I$ specified in Section \ref{sec:main}.

    Denote $\epsilon_3=\min\{\epsilon_1^2,r_2\}$. For each $(p,q)\in Y$, we create a drawstring around $\Sigma_{pq}$ by applying Theorem \ref{thm-main:h_and_u} with $\epsilon=\epsilon_3$, $v_0=f|_{\Sigma_{pq}}$. Since the surgery regions are all disjoint, we can put all the drawstrings into a single metric, which we denote by $g'$. By Theorem \ref{thm-main:h_and_u}(III) we have $R_{g'}\geq R_g-\epsilon_3>\lambda$.

    Denote $\bar g=e^{2f}g$. In each $N_g(\Sigma_{pq},\epsilon_3)$ we have the metric expression
    \[g=dr^2+\omega_\th^2+g_{\H}\]
    and
    \[g'=e^{-2u}dr^2+e^{-2u}h^2\omega_\th^2+e^{2u}g_{\H}\]
    for some $u,h$ satisfying the conditions of Theorem \ref{thm-main:h_and_u}. In particular, we have $1-\epsilon_3\leq h\leq1$ and $f\circ\pi\leq u\leq0$ (where recall that $\pi$ is the orthogonal projection onto $\Sigma_{pq}$). Thus, in $N_g(\sigma_{pq},\epsilon_3)$ we have
    \begin{equation}\label{eq-leetopping:lower}
        \begin{aligned}
            g' &\geq dr^2+(1-\epsilon_3)^2\omega_\th^2+e^{2f(\pi(x))}g_{\H}
            \geq (1-\epsilon_3)^2e^{2f\circ\pi}g
            \geq (1-\epsilon_3)^2e^{-2L\epsilon_3}\cdot e^{2f}g
        \end{aligned}
    \end{equation}
    and
    \begin{equation}\label{eq-leetopping:upper1}
        g'\leq e^{2|f\circ\pi|}dr^2+e^{2|f\circ\pi|}\omega_\th^2+g_{\H}\leq e^{2\sup|f|}g\leq e^{2L}g.
    \end{equation}
    Outside the union of $N_g(\Sigma_{pq},\epsilon_3)$ we have $g'=g$, so in summary we have obtained
    \begin{equation}\label{eq-leetopping:joint_bounds}
        (1-\epsilon_1^2)^2e^{-2L\epsilon_1^2}\bar g\leq g'\leq e^{2L}g.
    \end{equation}
    Note that this implies $g'\geq(1-o(1))\bar g$, where $o(1)\to0$ when $\epsilon_1\to0$. In particular, $d_{g'}\geq(1-o(1))d_{\bar g}$.

    Next, we obtain the sharp distance upper bound $d_{g'}\leq(1+o(1))d_{\bar g}$. Fix $x,y\in M$. Let $\gamma$ be a shortest $\bar g$-geodesic from $x$ to $y$. On $\gamma$ we select a sequence of points $x_1, x_2, x_3,\cdots, x_N$ in the following way: we first choose $x_1=x$, and inductively choose $x_{k+1}$ to be the first point with $d_g(x_k,x_{k+1})=2\epsilon_1$. When $d_g(x_N,y)\leq2\epsilon_1$ the selection process stops. Let $\gamma_k$ be the portion of $\gamma$ starting with $x_k$ and ending with $x_{k+1}$. Notice that $\diam_g(\gamma_k)\leq2\epsilon_1$, since otherwise this violates our choice of $x_{k+1}$. On the other hand, we have $\length_g(\gamma_k)\geq d_g(x_k,x_{k+1})=2\epsilon_1$. Counting the total length, it follows that
    \begin{equation}\label{eq-leetopping:N_bound}
        N-1\leq(2\epsilon_1)^{-1}\length_g(\gamma)\leq \epsilon_1^{-1}e^{\sup|f|}\length_{\bar g}(\gamma)\leq\epsilon_1^{-1}e^{L}\diam_{\bar g}(M).
    \end{equation}
    By our construction of the dense net, for each $1\leq k\leq N-1$ there is $(p,q)\in Y$ such that $d_g(x_k,p)\leq\epsilon_1^2$ and $d_g(x_{k+1},q)\leq\epsilon_1^2$. By \eqref{eq-leetopping:joint_bounds} we have
    \begin{equation}\label{eq-leetopping:error1}
        d_{g'}(x_k,p)\leq e^{L}d_g(x_k,p)\leq e^L\epsilon_1^2,
    \end{equation}
    and $d_{g'}(q,x_{k+1})\leq e^L\epsilon_1^2$ similarly. To estimate $d_{g'}(p,q)$, we note that
    \[\sup\big\{d_g(y,z): y\in\gamma_k,z\in\sigma_{pq}\big\}\leq\diam_g(\gamma_k)+\diam_g(\sigma_{pq})+d_g(x_k,p)\leq 5\epsilon_1.\]
    Also, recall from the drawstring construction that $g'|_{\sigma_{pq}}=e^{2f}g|_{\sigma_{pq}}$. Hence
    \[\begin{aligned}
        d_{g'}(p,q)\leq\length_{g'}(\sigma_{pq})\leq\sup_{\sigma_{pq}}(e^{f})\cdot2\epsilon_1\leq e^{5L\epsilon_1}\inf_{\gamma_k}(e^f)\cdot\length_g(\gamma_k)\leq e^{5L\epsilon_1}\length_{\bar g}(\gamma_k).
    \end{aligned}\]
    Summing over $k$ of what we obtained:
    \[\begin{aligned}
        d_{g'}(x,y) &\leq \sum_{k=1}^{N-1}\Big[d_{g'}(x_k,p)+d_{g'}(p,q)+d_{g'}(q,x_{k+1})\Big]+d_{g'}(x_N,y) \\
        &\leq \sum_{k=1}^{N-1}\Big[2e^{L}\epsilon_1^2+e^{5L\epsilon_1}\length_{\bar g}(\gamma_k)\Big]+e^{L}d_g(x_N,y).
    \end{aligned}\]
    Using \eqref{eq-leetopping:N_bound} and \eqref{eq-leetopping:error1} it follows that
    \begin{align}
        d_{g'}(x,y) &\leq 2e^{L}\epsilon_1^2\cdot\epsilon_1^{-1}e^{L}\diam_{\bar g}(M)+e^{5L\epsilon_1}d_{\bar g}(x,y)+2e^{L}\epsilon_1 \nonumber\\
        &\leq d_{\bar g}(x,y)+\big(e^{5L\epsilon_1}-1\big)\diam_{\bar g}(M)+2e^{2L}\big(1+\diam_{\bar g}(M)\big)\epsilon_1. \label{eq-leetopping:upper}
    \end{align}
    Finally, take the above construction for a sequence $\epsilon_1\to0$, and rename the resulting metrics $g'$ as $g_i$. The desired conclusions immediately follow from \eqref{eq-leetopping:joint_bounds} and \eqref{eq-leetopping:upper}.
\end{proof}

\subsection{Asymptotically flat examples}\label{sec:AFexamples}

This subsection concerns asymptotically flat manifolds and the Positive Mass Theorem. We first construct the drawstring examples containing no closed minimal surfaces.

\begin{proof}[Proof of Theorem \ref{thm-intro:AFexample}]
    Fix $m<1/4$. Choose a cutoff function $\eta:[0,\infty)\to[0,1]$ so that $\eta|_{[0,1/2]}\equiv1$ and $\eta|_{[1,\infty)}\equiv0$ and $\eta'<0$ in $(1/2,1)$. Consider the function
    \begin{equation}
        V_m(r)=1-\eta(r)2mr^2-\big(1-\eta(r)\big)\frac{2m}r,
    \end{equation}
    which is positive since $m<1/4$. Using polar coordinates, we define the following metric on $\mathbb{R}^3$
    \begin{equation}\label{eq-af:spheretoschw}
        g_m=\frac{dr^2}{V_m(r)}+r^2g_{S^2},
    \end{equation}
    In $\{r\leq1/2\}$, the metric $g_m$ is isometric to a geodesic ball in a $3$-sphere of curvature $2m$, via the change of variable $r=\frac{1}{\sqrt{2m}}\sin(\sqrt{2m}\rho)$. In $\{r>1\}$, $g_m$ is an exterior portion of the mass $m$ Schwarzschild metric (in static coordinates). Moreover, $g_m$ smoothly converges to the flat Euclidean metric as $m\to0$.
    
    The scalar curvature of $g_m$ is computed as (see \cite[Section 3.1]{geometric})
    \[\begin{aligned}
        \frac{r^2}2R_{g_m} &= 1-V_m-rV_m' \\
        &= 2\eta mr^2+(1-\eta)\frac{2m}r+4\eta mr^2-(1-\eta)\frac{2m}r+r\eta'\big(2mr^2-\frac{2m}r\big) \\
        &= 6m\eta r^2+2m\eta'(r^3-1).
    \end{aligned}\]
    Since $\eta$ is non-increasing, $\eta'(r)(r^3-1)\geq0$ everywhere, hence $R_{g_m}\geq0$ throughout $M$ and strict positivity $R_{g_m}>0$ holds on $\{r<1\}$. Also, note that the constant $r$-sphere has mean curvature
    \begin{equation}
        H=\frac12\tr_{r^2g_{S^2}}\big(\sqrt{V_m}\partial_r(r^2 g_{S^2})\big)=\frac{2\sqrt{V_m}}{r}>0,
    \end{equation}
    and so these spheres form a mean convex folliation of $\mathbb{R}^3$.
    
    We now apply the drawstring construction. For $\epsilon\in(0,1/100)$, apply Theorem \ref{thm-main:h_and_u} to $(\mathbb{R}^3,g_m)$ taking $\Sigma$ to be the circle $\sigma=\big\{z=0,x^2+y^2=1/4\big\}$, $v_0\equiv \log\varepsilon$, and let $g_{m,\varepsilon}$ denote the resulting metric on $\mathbb{R}^3$. Evidently, $g_{m,\varepsilon}$ satisfies conditions (i), (iii), (iv), and (v) of Theorem \ref{thm-intro:AFexample}. It remains to show that $(\mathbb{R}^3,g_{m,\varepsilon})$ has no closed embedded minimal surfaces. Once this is established, Theorem \ref{thm-intro:AFexample} follows by scaling $g$ with the factor of 4.
    
    We first notice that $g_{m,\varepsilon}$ is axisymmetric about the $z$-axis. To see this, observe that the functions $u$ and $h$ in Theorem \ref{thm-main:h_and_u} depend only on the $g_m$-distance to $\sigma$. Since this distance function has the desired axisymmetry, so does $g_{m,\varepsilon}$. 
    
    For sake of contradiction, suppose that there exists an embedded minimal $\mathcal{S}\subset\mathbb{R}^3$. Since $\mathbb{R}^3$ is contractible, $\mathcal{S}$ bounds some region $\Omega$. By passing to the outermost minimal surface enclosing $\Omega$, we may assume that $\mathcal{S}$ is the only closed minimal surface in $\mathbb{R}^3\setminus\Omega$. Notice that $\mathcal{S}$ is a stable minimal surface in a non-flat $3$-manifold of nonnegative scalar curvature. As such, well-known arguments \cite{geometric} imply that $\mathcal{S}$ is topologically a sphere. Moreover, notice that $\mathcal{S}$ is axially symmetric. Indeed, if it were not, then the outward minimizing enclosure of all the rotations of $\mathcal{S}$ would be a non-trivial minimal surface in $\RR^3\setminus\Omega$.
    
    Since $\mathcal{S}$ is symmetric about the $z$-axis and topologically a sphere, it must intersect the $z$-axis.
    In particular, we may find a $p\in\mathcal{S}$ so that $B_{g_m}(p,1/8)$ lies away from the drawstring region $N_{g_0}(\sigma,\varepsilon)$. Finally, the smooth convergence $g_m\to g_0$ implies that $g_m$ satisfies a sectional curvature bound for all small $m$, allowing us to apply the monotonicity formula on $B_{g_m}(p,1/8)$. This yields a lower bound $|\mathcal{S}|\geq A_0$ independent of $m$ and $\varepsilon$. For sufficiently small $m$, this contradicts the Riemannian Penrose inequality \cite{Bray,HI} which states that the outerminimizing surface has area no greater than $16\pi m^2$. It follows that the metrics $g_{m=\varepsilon,\varepsilon}$ satisfy the desired properties.
\end{proof}

\begin{proof}[Proof of Corollary \ref{cor-intro:pmt}]
    For a sequence $\epsilon_i\to0$, consider the metrics $g_{\epsilon_i}$ constructed in the proof of Theorem \ref{thm-intro:AFexample}. Let $g_\infty$ be the Euclidean metric in $\RR^3$.
    
    For each integer $k\geq10$, apply Proposition \ref{prop-prlim:pulledstringconv} with the choice $M^3=B_{\RR^3}(\vec{0},k)$, $p=\vec{0}$, $\sigma=\{z=0,x^2+y^2=1\}$, $g_i=g_{\epsilon_i}$, and $U_i=N_{g_\infty}(\sigma,\epsilon_i)$. Theorem \ref{thm-intro:AFexample}(iv) and Theorem \ref{thm-main:h_and_u}(IV) and (VI) ensures that the hypotheses of Proposition \ref{prop-prlim:pulledstringconv} are met. In particular, there is a sequence of compact domains exhausting $\mathbb{R}^3$ on which we have convergence in the pointed sense. It follows that $(\RR^3,g_i,\vec{0})$ converges in the pointed Gromov-Hausdorff and intrinsic flat senses to the space obtained from $\RR^3$ by pulling $\sigma$.
\end{proof}

We conclude this section with the proof of Theorem \ref{thm-intro:DSplane} on Dong-Song's stability result for the positive mass theorem. 

\begin{proof}[Proof of Theorem \ref{thm-intro:DSplane}]
    The result of Dong-Song \cite[Theorem 1.3]{Dong-Song_2023} provides an open set $Z_i\subset M_i$ and a map $\mathcal{U}_i:M_i\setminus Z_i\to\mathbb{R}^3$ satisfying the desired items 1 and 2. It is immediately clear from the construction of $Z_i$ that $\mathcal{U}_i$ is an injective immersion and that the $C^0$ norm of $g_i-\mathcal{U}_i^*\delta$ on $M_i\setminus Z_i$ tends to 0, where $\delta$ denotes the flat metric. It follows that the Euclidean area of $\partial \mathcal{U}_i(Z_i)$ tends to 0 as $i\to\infty$. Now let $\pi:\mathbb{R}^3\to P$ denote orthogonal projection onto a plane $P\in\mathbb{R}^3$. Since $\pi$ is distance non-increasing and $\pi(\mathcal{U}_i(\partial Z_i))$ contains $\mathcal{U}_i(Z_i)\cap P$, we find that $\big|\mathcal{U}_i(Z_i)\cap P\big|_\delta\leq \big|\partial \mathcal{U}_i(Z_i)\big|_\delta$, completing the proof.
\end{proof}

\subsection{Instability of the 2-form version of Llarull's Theorem}

The next task is to introduce drawstring geometry to the unit sphere and show Theorem \ref{thm-intro:larrull}, which demonstrates a type of instability of Llarull's Theorem.

\begin{proof}[Proof of Theorem \ref{thm-intro:larrull}] {\ }

    Let $(S^3,g_0)$ be the round sphere, and $\gamma$ be an arbitrary great circle. For each $i$ we take $g_i$ to be the metric constructed in Theorem \ref{thm-main:h_and_u}, with $\Sigma=\gamma$, $v_0\equiv-i$, and $\epsilon=\frac{1}{100i}$. Let $(r,\th,t)$ denote Fermi coordinates about $\gamma$ (here $t$ parameterizes $\gamma$ and $r$ represents radial distance to $\gamma$), in which the metric takes the form
    \[
    g_0=dr^2+\sin^2rd\th^2+\cos^2rdt^2.
    \]
    Recall that the drawstring metrics take the form
    \[
    g_i=e^{-2u_i}dr^2+e^{-2u_i}h_i^2\sin^2 rd\th^2+e^{2u_i}\cos^2rdt^2.
    \]
    Therefore, the three eigenvalues of $g_i$ relative to $g_0$ are $\{e^{-u_i},e^{-u_i}h_i,e^{u_i}\}$. As the forms $dr\wedge d\th, dr\wedge dt, d\th\wedge dt$ constitute an orthogonal basis of $\Lambda^2(T^*M)$, at the level of 2-forms we find that $g_i$ has eigenvalues $\{e^{-2u_i}h_i,1,h_i\}$ relative to $g_0$. By Theorem \ref{thm-main:h_and_u}(VII), these eigenvalues are greater than $1-\tfrac{1}{100i}$.

    Finally, we show that the Cheeger constant of $g_i$ is uniformly bounded below. In what follows, for ease of notation we fix some $i$ and use $|\Omega|$, $|\p\Omega|$ to denote the $g_i$-volume and area of a set $\Omega$. By the classical theory, there exists a minimizer $E$ satisfying
    \begin{equation}
        \frac{|\partial E|}{|E|}=\mathrm{Ch}(S^3,g_i),\qquad |E|\leq \frac12|S^3|.
    \end{equation}
    Classical regularity theory ensures that $\p E$ is a smooth surface, and the first variation formula shows that $\p E$ has constant mean curvature. Let $H$ be its mean curvature. 
    
    Next, we make the basic observation that $|H|\leq\mathrm{Ch}(S^3,g_i)$. 
    To see this, consider a variation $\{E_t\}_{t\in(-\epsilon,\epsilon)}$, such that $E_0=E$, and let $\varphi$ denote the speed of variation at $t=0$ in the direction of the unit normal pointing out of $E$. By the first variation, we have
    \[\frac d{dt}|E_t|\Big|_{t=0}=\int_{\p E}\varphi,\qquad\frac d{dt}|\p E_t|\Big|_{t=0}=\int_{\p E}H\varphi.\]
    First consider $\varphi=-1$. Then $E_t$ is moving inward and hence $\frac{|\p E_t|}{|E_t|}\geq\mathrm{Ch}(S^3,g_i)=\frac{|\p E_0|}{|E_0|}$ for all $t\geq0$. Taking the one-sided derivative at $t=0$, we obtain
    \begin{align}
    \label{eq-Llarull:variation}0\leq\frac d{dt}\frac{|\p E_t|}{|E_t|}\Big|_{t=0^+}&=\frac{\int_{\p E}H\varphi\cdot|E|-|\p E|\cdot\int_{\p E}\varphi}{|E|^2}\\
    {}&=\frac{|\p E|}{|E|}\big(\!-H+\Ch(S^3,g_i)\big).
    \end{align}
    This implies that $H\leq\Ch(S^3,g_i)$. To estimate $-H$, we consider two cases: If $|E|<\frac12|S^3|$ strictly, then we can take $\varphi=1$ in \eqref{eq-Llarull:variation} and conclude that $H=\Ch(S^3,g_i)$. In the case where $|E|=\frac12|S^3|$, the setting $\varphi=1$ moves $E$ outward, thus we need to compare $\frac{|\p E_t|}{|S^3\setminus E_t|}\geq\Ch(S^3,g_i)=\frac{|\p E|}{|S^3\setminus E|}$. It follows that
    \begin{align}
        0 \leq \frac d{dt}\frac{|\p E_t|}{|S^3\setminus E_t|}\Big|_{t=0^+} &= \frac{\int_{\p E}H\varphi\cdot|E|+|\p E|\cdot\int_{\p E}\varphi}{|E|^2} \\
        &= \frac{|\p E|}{|E|}\big(H+\Ch(S^3,g_i)\big).
    \end{align}
    As a result, we have $H\geq-\Ch(S^3,g_i)$, thus proving our claim.
    
    The claim will yield a uniform bound on $H$. By considering a fixed geodesic ball away from $\gamma$, there is an upper bound $|H|\leq \mathrm{Ch}(S^3,g_i)\leq100$. By combining this bound with the mean convexity of the tubular neighborhoods $\Sigma_r$, the maximum principle for hypersurfaces will show that $\partial E$ cannot be contained in a small neighborhood of $\gamma$. Indeed, by Theorem \ref{thm-main:h_and_u} there is a radius $r_i<1/100i$ so that: the mean curvature of the tubular neighborhoods $\Sigma_r$ is at least $\frac1{4r\log(1/r)}$ for $r\leq r_i$, and $g_i=g_0$ outside $N(\gamma,r_i)$. It follows that the $g_i$-mean curvatures of $\Sigma_r$ satisfy
    \begin{equation}\label{eq-Llarull:meancurv}
    H_{\Sigma_r}\geq\min\Big\{\cot(r)-\tan(r),\frac1{4r\log(1/r)}\Big\}
    \end{equation}
    for all $r<\frac{1}{10}$ and all $i$. In particular, there is a uniform $\rho$ so that $H_{\Sigma_r}\geq200$ for all $i$ and $r<\rho$. Then by the maximum principle, there exists a point $x\in\p E$ such that $d_{g_0}(x,\gamma)\geq\rho$, meaning that the ball $B_{g_0}(x,\rho/2)$ is disjoint from the drawstring for large $i$. This allows us to use the monotonicity formula for hypersurfaces with bounded mean curvature, yielding a uniform lower bound on the quantity $|\p E\cap B(x,\rho/2)|$. Hence $\mathrm{Ch}(S^3,g_i)\geq|\p E\cap B(x,\rho/2)|\cdot|S^3|^{-1}$ is uniformly bounded below.

    To conclude, we will apply the Basilio-Sormani Scrunching Theorem \cite[Theorem 2.5]{BS} (discussed in Appendix \ref{sec:scrunch}) to show that $(S^3,g_i)$ converges to the space pulling $\gamma$ to a point.  In particular, in \cite[Theorem 2.5]{BS} we take the scrunching regions $U_i$ as $N_{g_0}(\gamma,1/i)$. To verify the hypotheses there: item (i) is evident, item (ii) follows from Theorem \ref{thm-main:h_and_u}(VI), and item (iii) follows from Theorem \ref{thm-main:h_and_u}(I)(IV).
\end{proof}


\section{Conformal inversion and the prototype functions}\label{sec:confinv}

The goal of this section is to present a heuristic derivation of the prototype drawstring formulas \eqref{eq-intro:KX23_main_func}. In particular, we consider a warped product ansatz $g_0+\varphi^2dt^2$ and attempt to produce a metric of positive scalar curvature with the property that the $dt^2$ factor rapidly changes in size. In doing so, we are naturally led to the extreme examples of Sormani-Tian-Wang \cite{STW23} where $\varphi$ becomes unbounded. In order to create examples where the $dt^2$ factor becomes arbitrarily small, we consider $\varphi^4g_0+\varphi^{-2}dt^2$ and note a special relationship between the scalar curvature of this metric and the original warped product's. After changing coordinates appropriately, this new metric may be identified with a special case of the doubly warped products \eqref{eq-intro:warped_prod} used as the fundamental building block for drawstring metrics.

\vspace{6pt}

Let $(S^2,g_0)$ be the unit round sphere, and suppose $\varphi$ is a positive smooth function on $S^2$. It is known that the warped product metric $g_0+\varphi^2dt^2$ has nonnegative scalar curvature if and only if $\Delta_0\varphi\leq\varphi$. As $\varphi$ is a supersolution to an elliptic equation, a uniform lower bound on $\varphi$ is obtainable from Moser's Harnack inequality. On the other hand, the condition $\Delta_0\varphi\leq\varphi$ does not imply any uniform upper bound on $\varphi$. In \cite{STW23} the following example is considered:
\begin{equation}\label{eq-app:STW}
    \varphi(r,\th)=\frac12\log\Big(\frac{1+\delta}{\sin^2r+\delta}\Big)+1,
\end{equation}
where $(r,\th)$ are geodesic polar coordinates on the unit sphere and $\delta>0$. Such functions satisfy $\Delta_0\varphi\leq\varphi$ and become unbounded as $\delta\to0$. To obtain a drawstring-like metric from \eqref{eq-app:STW}, we seek a transformation that has the effect of inverting the warping factor. The key observation is the following:

\begin{lemma}\label{lemma-app:scalarwarpedconformal}
    Suppose $(\Sigma,g)$ is a Riemannian surface with a smooth function $\varphi>0$. Set $\tilde g=\varphi^4g$. Then $\Delta_g\varphi\leq K\varphi$ if and only if $\Delta_{\tilde g}\varphi^{-1}\leq\tilde K\varphi^{-1}$ (where $K,\tilde K$ are the Gauss curvatures of $g,\tilde g$).
\end{lemma}
\begin{proof}
    Define a function $v$ so that $\varphi=e^v$. Then we have $\tilde g=e^{4v}g$ and the conformal change formula gives
    \[\tilde K=e^{-4v}\big(K-2\Delta_g v\big).\]
    On the other hand, one can directly compute
    \[\begin{aligned}
        \Delta_{\tilde g}\varphi^{-1} &= e^{-4v}\Delta_g(e^{-v})=e^{-5v}\big(-\Delta_g v+|\D_g v|^2\big).
    \end{aligned}\]
    Combining these two computations shows
    \[\Delta_{\tilde g}\varphi^{-1}-\tilde K\varphi^{-1}=e^{-5v}\big(\Delta_gv+|\D_g v|^2-K\big)=e^{-6v}\big(\Delta_ge^v-Ke^v\big). \qedhere\]
\end{proof}

Lemma \ref{lemma-app:scalarwarpedconformal} implies that $g_0+\varphi^2dt^2$ has positive scalar curvature if and only if the metric $\varphi^4g_0+\varphi^{-2}dt^2$ does. Taking $\varphi$ as in \eqref{eq-app:STW} with $\delta\to0$, we obtain another sequence of warped products in which the warping factor is arbitrarily close to zero at $r=0$.

We rewrite the new warped product as
\begin{equation}\label{eq-app:g'conformalform}
    g=\varphi^4g_0+\varphi^{-2}dt^2=e^{2v}\big(e^{2v}dr^2+e^{2v}\sin^2r d\th^2\big)+e^{-2v}dt^2,
\end{equation}
and consider the re-parametrization $\tilde r=\tilde r(r)$ so that
\[d\tilde r=e^{v(r)}dr=\varphi(r)\,dr.\]
Next, keeping in mind that $r$ is a function of $\tilde{r}$, consider the functions $u(\tilde r)=-v(r)$, $f(\tilde r)=e^{v(r)}\sin r$. We finally obtain
\begin{equation}
    g=e^{-2u(\tilde r)}\big(d\tilde r^2+f(\tilde r)^2d\th^2\big)+e^{2u(\tilde r)}dt^2,
\end{equation}
which has the same form as \eqref{eq-intro:warped_prod}. It is natural to investigate the asymptotics of $f, u$ as $\tilde r\to0$. We show that in the extreme case $\delta=0$ in \eqref{eq-app:STW}, we obtain exactly the main formula \eqref{eq-intro:KX23_main_func}, up to an immaterial difference in coefficients:

\begin{theorem}\label{thm-app:main}
    The functions $f,u$ defined above using $\delta=0$ in \eqref{eq-app:STW} have the following asymptotics as $\tilde r\to0$
    \begin{equation}
        u(\tilde r)=-\log\log\frac1{\tilde r}+o(1),\qquad
        f(\tilde r)=\tilde r\Big[1-\frac1{\log(1/\tilde r)}+o\big(\frac1{\log(1/\tilde r)}\big)\Big].
    \end{equation}
\end{theorem}

For convenience, we introduce the notation $s=\log(1/r)$ and $\tilde s=\log(1/\tilde r)$. We remind the reader that $s,\tilde s$ are large when $r\to0$. To prove Theorem \ref{thm-app:main}, we establish some preliminary asymptotics. Setting $\delta=0$ in \eqref{eq-app:STW}, we have the expansion
\[\varphi(r)=s+c+o(1)\]
near $r=0$, where $c$ is a constant (whose value is not important). By direct integration, we have
\begin{equation}\label{eq-app:tilde_r}
    \begin{aligned}
        \tilde r=\int_0^r\Big[\log\frac1r+c+o(1)\Big]\,dr &= rs+(c+1)r+o(r) \\
        &= r\big[s+(c+1)+o(1)\big].
    \end{aligned}
\end{equation}
The next step is to express $r$ as a function in $\tilde r$.

\begin{lemma}\label{lemma-app:asymp2}
    As $r\to0$ we have
    \begin{equation}\label{eq-app:asymp_r}
        r=\frac{\tilde r}{\tilde s}\Big[1-\frac{\log\tilde s}{\tilde s}-(c+1)\frac{1}{\tilde s}+o\big(\frac{1}{\tilde s}\big)\Big].
    \end{equation}
\end{lemma}
\begin{proof}
    It is sufficient to show that
    \begin{equation}\label{eq-app:to_verify}
        0=\lim_{r\to0}\tilde s\Big[\frac{r\tilde s}{\tilde r}-1+\frac{\log\tilde s}{\tilde s}+(c+1)\frac1{\tilde s}\Big]=\lim_{r\to0}\Big[\frac{r\tilde s^2}{\tilde r}-\tilde s+\log\tilde s+(c+1)\Big].
    \end{equation}
    Note that \eqref{eq-app:tilde_r} implies
    \begin{equation}\label{eq-app:tilde_s}
        \tilde s = \log\frac1{rs\big[1+(c+1)s^{-1}+o(s^{-1})\big]}
        = s+\log\frac1{s}-(c+1)s^{-1}+o(s^{-1}),
    \end{equation}
    where we have made use of Taylor expansion in the variable $s^{-1}$.
    Plugging \eqref{eq-app:tilde_r} and \eqref{eq-app:tilde_s} into the first term of \eqref{eq-app:to_verify} and expanding the square, we obtain
    \begin{align}
        \frac{r\tilde s^2}{\tilde r} &= \frac{s^2\big[1+2s^{-1}\log\frac1{s}+s^{-2}\log^2\frac1s-2(c+1)s^{-2}+o(s^{-2})\big]^2}{s\big[1+(c+1)s^{-1}+o(s^{-1})\big]} \nonumber\\
        &= s\Big[1+2s^{-1}\log\frac1s+o(s^{-1})\Big]\cdot\Big[1-(c+1)s^{-1}+o(s^{-1})\Big] \nonumber\\
        &= s+2\log\frac1s-(c+1)+o(1). \label{eq-app:term1}
    \end{align}
    On the other hand, we may use \eqref{eq-app:tilde_s} to find
    \begin{equation}\label{eq-app:term3}
        \log\tilde s = \log \Big[s\big(1+s^{-1}\log\frac1s+o(s^{-1})\big)\Big] 
        = \log s+o(1).
    \end{equation}
    Combining \eqref{eq-app:tilde_s}, \eqref{eq-app:term1}, and \eqref{eq-app:term3} with \eqref{eq-app:to_verify}, we obtain the result.
\end{proof}

\begin{proof}[Proof of Theorem \ref{thm-app:main}]
    From Lemma \ref{lemma-app:asymp2} we have
    \[s=\log\frac1r=\log\Big[\frac{\tilde s}{\tilde r}\big(1+o(1)\big)\Big]=\tilde s\cdot\big(1+o(1)\big),\]
    therefore
    \[\begin{aligned}
        u(\tilde r) &= -\log\varphi(r)
        = -\log\big[s+c+o(1)\big]
        = -\log s-cs^{-1}+o(s^{-1})\\
        &= -\log\tilde s-o(1).
    \end{aligned}\]
    Turning our attention to $f$, we use Lemma \ref{lemma-app:asymp2} to obtain a finer asymptotics:
    \begin{equation}\label{eq-app:asymp_s}
        \begin{aligned}
            s &= -\log r
            = -\log\frac{\tilde r}{\tilde s}+\log\Big[1-\frac{\log\tilde s}{\tilde s}-(c+1)\frac1{\tilde s}+o(\frac1{\tilde s})\Big] \\
            &= \tilde s+\log\tilde s-\frac{\log\tilde s}{\tilde s}-(c+1)\frac1{\tilde s}+o(\frac1{\tilde s}) \\
            &= \tilde s+\log\tilde s+o(1).
        \end{aligned}
    \end{equation}
    Using \eqref{eq-app:asymp_s} and \eqref{eq-app:asymp_r} and the fact $r=o(\tilde r)$ from \eqref{eq-app:asymp_r}, we have
    \[\begin{aligned}
        f(\tilde r) &= \varphi(r)\sin r
        = \big[s+c+o(1)\big]\cdot r\big[1+O(r^2)\big] \\
        &= \big[\tilde s+\log\tilde s+c+o(1)\big]\cdot\frac{\tilde r}{\tilde s}\Big[1-\frac{\log\tilde s}{\tilde s}-(c+1)\frac{1}{\tilde s}+o\big(\frac{1}{\tilde s}\big)\Big]\cdot\big[1+o(\tilde r^2)\big] \\
        &= \tilde r\Big[1-\frac1{\tilde s}+o\big(\frac1{\tilde s}\big)\Big],
    \end{aligned}\]
    where in the last equality we algebraically expanded, noted a cancelation of the second order term and in terms involving $c$, and used the fact that $\frac{(\log \tilde s)^2}{\tilde{s}^2}=o(\frac{1}{\tilde{s}})$.
\end{proof}


\appendix

\section{Convergence to pulled string space}\label{sec:scrunch}

Suppose we are given a metric space $(X,d^X)$ and a connected compact subset $K\subset X$. One may use this data to construct a {\emph{pulled string}} space $(Y,d^Y)$ by, informally speaking, declaring $K$ to have vanishing length. Precisely, the metric space $(Y,d^Y)$ is given by
\begin{equation}
    Y=(X\setminus K)\cup\{K\},\quad d^Y(x,y)=\min\left\{d^X(x,y),d^X(x,K)+d^X(K,y)\right\}.
\end{equation}
We refer to $(Y,d^Y)$ as {\emph{created by pulling $K$ in $(X,d^X)$}}. In the case where $(X,d^X)$ carries an integral current structure, such as the case of a Riemannian manifold, an associated structure is induced on the pulled-string space, see \cite{BDS} for details.

The following proposition is a modification of a result by Basilio-Sormani \cite[Theorem 2.5]{BS} which gives a sufficient condition for a sequence $(M,g_i)$ to converge to a pulled string space. In particular, we do not require that the metrics $g_i$ are constant away from the curve $\sigma$.

\begin{prop}\label{prop-prlim:pulledstringconv}
    Fix a compact Riemannian $3$-manifold $(M^3,g_\infty)$, a basepoint $p\in M\setminus\p M$, and a closed curve $\sigma\subset M^3\setminus(\{p\}\cup\partial M^3)$. Suppose $(M^3,g_i)$ are Riemannian manifolds with subsets $U_i\subset M^3$ containing $\sigma$ and positive numbers $\delta_i\to0$ such that:
    \begin{enumerate}[label=(\roman*), topsep=3pt, itemsep=0ex]
        \item $U_i=N_{g_\infty}(\sigma,\delta_i)$,
        \item there is $C^0$ convergence away from $U_i$, $||g_i-g_\infty||_{C^0(M\setminus\mathring{U}_i,g_\infty)}\to0$,
        \item $\Vol(U_i,g_i)\to0$ and $\Vol(M,g_i)\leq \Vol(M,g_\infty)(1+\delta_i)$,
        \item $\diam(U_i,g_i)\to0$.
    \end{enumerate}
    Then $(M_i,d_{g_i},p)$ converges to the space resulting from pulling $\sigma$ in $(M,d_{g_\infty},p)$ in the pointed intrinsic flat and Gromov-Hausdorff senses.
\end{prop}
\begin{proof}
    For each $i$, construct auxillary metrics $g'_i$ on $M$ so that $g'_i=g_\infty$ on $M\setminus U_i$, $\diam(U_i,g'_i)\to0$, and $\Vol(U_i,g'_i)\leq \Vol(B_{g_\infty}(\sigma,1/i))$. Such metrics may be constructed by conformally modifying $g_\infty$ on $U_i$. According to \cite[Theorem 5.2]{BDS}, $(M,g'_i)$ converges to the pulled string space created by pulling $\sigma$ in $(M,g_\infty)$. Moreover, the map in lines \cite[(131)--(135)]{BDS} used to prove \cite[Theorem 2.5]{BDS} is the identity outside $U_i$, and it follows that the convergence holds in the pointed sense. On the other hand, Lemma \ref{lem-prelim:auxconvergence} shows that $(M,d_{g_i},p)$ is close to $(M,d_{g'_i},p)$ in the pointed Gromov-Hausdorff and Intrinsic Flat senses. Combining these facts yields the desired conclusion.
\end{proof}

\begin{lemma}\label{lem-prelim:auxconvergence}
    Let $M^n$ be a compact $n$-manifold equipped with Riemannian metrics $g,g_0$ and fix a compact set $U\subset M$ and a basepoint $p\in M$. For each $D,V,\mu>0$, there are $\delta,\epsilon>0$ so that the following holds: If 
    \begin{enumerate}[label=(\roman*), topsep=3pt, itemsep=0ex]
        \item $\diam(M,g),\diam(M,g_0)\leq D$ and $\Vol(M,g),\Vol(M,g_0)\leq V$,
        \item $||g-g_0||_{C^0(M\setminus U,\,g_0)}<\delta$,
        \item $\diam(U,g),\diam(U,g_0)\leq\epsilon$,
        \item $\Vol(U,g),\Vol(U,g_0)<\epsilon$,
    \end{enumerate}
    then $d_{\GH}\big((M,d_g,p),(M,d_{g_0},p)\big)<\mu$ and $d_{\IF}\big((M,d_g,p),(M,d_{g_0},p)\big)<\mu$.
\end{lemma}

\begin{proof}
    We build a comparison space $(Z,\bar g)$. Fix
    \begin{equation}\label{eq-prelim:eta}
        \eta=\max\Big\{\sqrt{3\delta D^2}, \sqrt{3D(\epsilon+2\delta D)}\Big\},
    \end{equation}
    noting that $\eta\to0$ as both $\epsilon,\delta\to0$.
    The comparison space will take the form $(Z,\bar g)=(M\times[0,3\eta],g_t+dt^2)$, where $g_t$ is a metric on $M\times\{t\}$. First define a preliminary metric $\hat g_t$
    \begin{equation}
        \widehat g_t=
        \begin{cases}
            g&\text{ for }t\in[0,\eta],\\
            (2-t/\eta)g+(t/\eta-1)g_0&\text{ for }t\in[\eta,2\eta],\\
            g_0&\text{ for }t\in [2\eta,3\eta].
        \end{cases}
    \end{equation}
    Now we will modify $\widehat g_t$ to decrease its volume in the region $U\times[\eta,2\eta]$. In particular, set $g_t(x)=F(x,t)^2\hat g_t(x)$ where $F$ is a positive function on $Z$ chosen so that $F(x,t)=1$ for $(x,t)\in Z\setminus(U\times[\eta,2\eta])$ and $\int_{U\times[\eta,2\eta]}F^{n+1}dV_{dt^2+\widehat g_t}<\epsilon$.
    
    We claim that $d_Z(x,y)=d_g(x,y)$ for all $x,y\in M\times\{0\}\subset Z$. It is sufficient to show that $d_Z(x,y)\geq d_g(x,y)$. Let $\gamma$ be a $d_Z$-minimizing geodesic between $x$ and $y$. Set $W=U\times[\eta,3\eta]$. Let $\Pi:M\times[0,3\eta]\to M\times\{0\}$ denote the projection map. Note the following elementary inequality: for a point $x\in M\times\{0\}$ and a point $p\in M\times\{\eta\}$, we can estimate the product distance by
    \begin{align}\label{eq-prelim:productdist}
        \begin{split}
            d_{M\times[0,\eta]}(x,p)&=\sqrt{d_g(x,\Pi(p))^2+\eta^2}\\
            &= d_g(x,\Pi(p))^2+\frac{\eta^2}{\sqrt{d_g(x,\Pi(p))^2+\eta^2}+d_g(x,\Pi(p))}\\
            &\geq d_g(x,\Pi(p))^2+\frac{\eta^2}{3D}.
        \end{split}
    \end{align}
    Also, note that outside $W$ we have
    \begin{equation}\label{eq-prelim:outside_W}
        \bar g\geq dt^2+(1-\delta)g\geq dt^2+(1-\delta)^2g.
    \end{equation}
    
    There are two cases for the $d_X$-geodesic $\gamma$:
    
    {\bf{Case 1:}} $\gamma\cap W=\emptyset$. If $\gamma$ is contained in $M\times[0,\eta]$, then since $\bar g=g+dt^2$ inside it, we clearly have $|\gamma|\geq d_g(x,y)$. Suppose $\gamma$ intersects $M\times[\eta,3\eta]$. Let $p,q\in M\times\{\eta\}$ be the points of first entry and last exit. By \eqref{eq-prelim:outside_W}, and the fact that $\gamma$ is minimizing and avoids $W$, we have $d_Z(p,q)\geq(1-\delta)d_g(\Pi(p),\Pi(q))$. Therefore, using \eqref{eq-prelim:productdist},
    \[\begin{aligned}
        d_Z(x,y) &= d_Z(x,p)+d_Z(p,q)+d_Z(q,y) \\
        &\geq \sqrt{d_g(x,\Pi(p))^2+\eta^2}+(1-\delta)d_g(\Pi(p),\Pi(q))+\sqrt{d_g(y,\Pi(q))^2+\eta^2} \\
        &\geq d_g(x,\Pi(p))+\frac{\eta^2}{3D}+d_g(\Pi(p),\Pi(q))-\delta D+d_g(y,\Pi(q))+\frac{\eta^2}{3 D} \\
        &\geq d_g(x,y)+\frac{2\eta^2}{3D}-\delta D.
    \end{aligned}\]
    Then by \eqref{eq-prelim:eta} we have $d_Z(x,y)\geq d_g(x,y)$ in this case.

    {\bf{Case 2:}} $\gamma\cap W\neq \emptyset$. Let $p$ and $q$ be the first and final points where $\gamma$ intersects $W$. Again, since $\bar g\geq dt^2+(1-\delta)^2g$ outside $W$, and the $t$ coordinates of $p,q$ are at least $\eta$, we may estimate as in \eqref{eq-prelim:productdist}
    \[d_Z(x,p)\geq\sqrt{(1-\delta)^2d(x,\Pi(p))^2+\eta^2}\geq(1-\delta)d_g(x,\Pi(p))+\frac{\eta^2}{3D}.\]
    Then we have
    \[\begin{aligned}
        d_Z(x,y) &\geq d_Z(x,p)+d_Z(q,y) \\
        &= (1-\delta)d_g(x,\Pi(p))+(1-\delta)d_g(y,\Pi(q))+\frac{2\eta^2}{3D} \\
        &\geq d_g(x,\Pi(p))+d_g(y,\Pi(q))-2\delta D+\frac{2\eta^2}{3D} \\
        &\geq d_g(x,y)-\diam(U)-2\delta D+\frac{2\eta^2}{3D}.
    \end{aligned}\]
    By \eqref{eq-prelim:eta}, we also have $d_Z(x,y)\geq d_g(x,y)$ in this case.
    
    Since the roles of $g$ and $g_0$ are symmetric in the above arguments, we also conclude that $d_{g_0}(x,y)=d_Z(x,y)$ for all $x,y\in M\times\{3\eta\}$. It follows that the inclusions $(M,g)\hookrightarrow M\times\{0\}$ and $(M,g_0)\hookrightarrow M\times\{3\eta\}$ are isometries. Hence these imply that $d_{\GH}\big((M,d_g,p),(M,d_{g_0},p)\big)<3\eta$.
    
    Next, note that $d_{\IF}\big((M,d_g,p),(M,d_{g_0},p)\big)\leq\Vol(Z,\bar g)$. We directly bound the volume of $Z$ using assumptions $(i)$, $(iv)$, and the condition on $F$:
    \[\begin{aligned}
        \Vol(Z,\bar g) &\leq
            \Vol\big((M\setminus U)\times[0,3\eta],dt^2+(1+\delta)^2g\big)
            + \Vol\big(U\times[0,\eta],dt^2+g\big)\\
        {}&\qquad
            + \Vol\big(U\times[2\eta,3\eta],dt^2+g_0\big)
            + \Vol\big(U\times[\eta,2\eta],dt^2+F\widehat{g}_t\big)\\
        {}&\leq 3\eta(1+\delta)^n \Vol(M\setminus U,g)+2\eta \epsilon +\int_{U\times[\eta,2\eta]}F^{n+1}dV_{dt^2+\widehat g_t}\\
        {}&\leq 3\cdot 2^n\eta V+2\eta\epsilon+\epsilon.
    \end{aligned}\]
    The result follows by choosing $\epsilon,\delta\ll1$ so that $3\eta<\mu$ and $3\cdot 2^n\eta V+2\eta\epsilon+\epsilon<\mu$.
\end{proof}


\section{Analytic lemmas}\label{sec:ana}

The following lemmas concern the smoothness of a two-variable function expressed in polar coordinates.

\begin{lemma}\label{lemma-ana:smoothness}
    Let $U\in\RR^{n-2}$ be an open set, and $(\x,y,z)$ be the Cartesan coordinates on $U\times D^{\RR^2}(r_0)\subset\RR^n$, where $\x=(x_1,\cdots,x_{n-2})$. Consider the polar coordinates $(\x,r,\th)$ with the transformation $x=r\cos\th$, $y=r\sin\th$. Suppose a function $f=f(\x,r,\th)$ is smooth on $U\times[0,r_0)\times S^1$ and satisfies for all $k\geq0$
    \begin{equation}\label{eq-ana:deriv_condition}
        \frac{\p^k f}{\p r^k}(\x,0,\th)=\sum_{i=0}^k c_{k,i}(\x)\cos^i\th\sin^{k-i}\th\quad
    \end{equation}
    for some smooth functions $c_{k,i}$ on $U$. Then $f$ is a smooth function on $U\times D^{\RR^2}(r_0)$, when expressed in the Cartesan coordinates $(\x,y,z)$.
\end{lemma}
\begin{proof}
    For each $k>0$, there exists a function $R_k(\x,r,\th)$ which is smooth on $U\times[0,r_0)\times S^1$, such that
    \begin{equation}\label{eq-ana:taylor}
        f(\x,r,\th)=\sum_{l=0}^k\frac{r^l}{l!}\frac{\p^l f}{\p r^l}(\x,0,\th)+r^{k+1}R_k(\x,r,\th).
    \end{equation}
    Indeed, Taylor's formula with integral remainder gives the expression
    \[R_k(\x,r,\th)=\frac1{k!}\int_0^1 \p^{k+1}_rf(\x,tr,\th)(1-t)^k\,dt.\]
    Applying \eqref{eq-ana:deriv_condition} to \eqref{eq-ana:taylor} and transforming to Cartesan coordinates, we obtain
    \begin{equation}\label{eq-ana:taylor2}
        f(\x,r,\th)=\sum_{l=0}^k\sum_{i=0}^l\frac{c_{l,i}(\x)}{l!}y^iz^{l-i}+r^{k+1}R_k(\x,r,\th).
    \end{equation}
    Recall that $r=\sqrt{y^2+z^2}$ and $\th=\arctan\frac zy$. It follows from homogeneity that
    \[|\p^i_y\p^{k-i}_zr|\leq C(k)r^{1-k},\quad
        |\p^i_y\p^{k-i}_z\th|\leq C(k)r^{-k}.\]
    Consecutively taking derivatives, we obtain
    \[\big|\p^i_y\p^{k-i}_z\p_{\x}^j(r^{k+1}R_k)\big|\leq Cr,\]
    where $C$ depends on the derivatives of $R_k$. Along with \eqref{eq-ana:taylor2} this shows that $f\in C^n(U\times D^{\RR^2}(r_0))$ in the Cartesan coordinates.
\end{proof}

\begin{lemma}\label{lemma-ana:smoothness_jac}
    Assume the same notations as in Lemma \ref{lemma-ana:smoothness}. Let $K(\x,y,z)$ be a function on $U\times D^{\mathbb{R}^2}(r_0)$ which is smooth in Cartesan coordinates. Define a function $h(\x,r,\th)$ by solving the initial value problem
    \begin{equation}\label{eq-ana:ivp}
	\left\{\begin{aligned}
	   & h_{rr}=-K(\x,r\cos\th,r\sin\th)h, \\
	   & h(\x,0,\th)=0,\ h_r(\x,0,\th)=1.
        \end{aligned}\right.
    \end{equation}
    Then $r^{-1}h$ is smooth when expressed in terms of the Cartesan coordinates $(\x,y,z)$.
\end{lemma}
\begin{proof}
    Since $h(\x,0,\th)=0$, we have
    \[h(\x,r,\th)=\int_0^r h_r(\x,\rho,\th)\,d\rho=r\int_0^1 h_r(\x,tr,\th)\,dt.\]
    It suffices to show that $h_r$ is a smooth function in Cartesan coordinates (this implies the smoothness of $h_r(\x,tr,\th)=h_r(\x,ty,tz)$ in Cartesan coordinates, and the lemma follows). By Lemma \ref{lemma-ana:smoothness}, it is sufficient to verify condition \eqref{eq-ana:deriv_condition} for $h_r$. The case $k=0$ follows directly from the initial condition. Suppose that \eqref{eq-ana:deriv_condition} holds for $k\leq N$, let us verify it with $k=N+1$. Differentiating the first line of \eqref{eq-ana:ivp}, we have
    \[\begin{aligned}
        \p_r^{N+1}(h_r) &= \p_r^N(h_{rr})=-\sum_{\substack{a,b,c\geq0, \\ a+b+c=N}}\binom{N}{c}\binom{a+b}{a}\p^a_y\p^b_z K\cdot\cos^a\th\sin^b\th\cdot\p_r^c h
    \end{aligned}\]
    Note that $h(\x,0,\th)=0$, hence the terms with $c=0$ are zero when evaluating at $(\x,0,\th)$. Therefore
    \[\begin{aligned}
        \p_r^{N+1}(h_r)(\x,0,\th) &= \sum_{\substack{a,b,d\geq0, \\ a+b+d=N-1}}C(N,a,b)\p^a_y\p^b_z K(\x,0,0)\cdot\cos^a\th\sin^b\th\cdot\p_r^d(h_r) \\
        &= \sum_{\substack{a,b,d\geq0, \\ a+b+d=N-1}}C(N,a,b)\sum_{i=0}^d\Big[\p^a_y\p^b_z K(\x,0,0)c_{d,i}(\x)\Big]\cdot\cos^{a+i}\th\sin^{b+d-i}\th,
    \end{aligned}\]
    where we used the induction hypothesis in the second line. In this way we can express
    \[\p_r^{N+1}(h_r)(\x,0,\th)=\sum_{i=0}^{N-1}\tilde c_i(\x)\cos^i\th\sin^{N-1-i}\th\]
    for some functions $\tilde c_i(\x)$. To further express this as a linear combination of $\cos^i\th$ and $\sin^{N+1-i}\th$, we multiply the right hand side by $\cos^2\th+\sin^2\th$ and re-group the coefficients. This shows condition \eqref{eq-ana:deriv_condition}.
\end{proof}

\bibliographystyle{alpha}
\bibliography{refs}

\begin{thebibliography}{STW24}

\bibitem[ABK23]{AllenBrydenK}
Brian Allen, Edward Bryden, and Demetre Kazaras.
\newblock {On the Stability of Llarull's Theorem in Dimension Three}, 2023.
\newblock \href{https://arxiv.org/abs/2305.18567}{{arXiv:2305.18567}}.

\bibitem[Bam16]{Bamler}
Richard~H. Bamler.
\newblock A {R}icci flow proof of a result by {G}romov on lower bounds for scalar curvature.
\newblock {\em Math. Res. Lett.}, 23(2):325--337, 2016.

\bibitem[BDS18]{BDS}
J.~Basilio, J.~Dodziuk, and C.~Sormani.
\newblock Sewing {R}iemannian manifolds with positive scalar curvature.
\newblock {\em J. Geom. Anal.}, 28(4):3553--3602, 2018.

\bibitem[Bra01]{Bray}
Hubert~L. Bray.
\newblock Proof of the {R}iemannian {P}enrose inequality using the positive mass theorem.
\newblock {\em J. Differential Geom.}, 59(2):177--267, 2001.

\bibitem[BS21]{BS}
J.~Basilio and C.~Sormani.
\newblock Sequences of three dimensional manifolds with positive scalar curvature.
\newblock {\em Differential Geom. Appl.}, 77:Paper No. 101776, 27, 2021.

\bibitem[CM11]{CM}
Tobias~Holck Colding and William~P. Minicozzi, II.
\newblock {\em A course in minimal surfaces}, volume 121 of {\em Graduate Studies in Mathematics}.
\newblock American Mathematical Society, Providence, RI, 2011.

\bibitem[Don24]{DongStability}
Conghan Dong.
\newblock Some stability results of positive mass theorem for uniformly asymptotically flat 3-manifolds.
\newblock {\em Ann. Math. Qu\'e.}, 48(2):427--451, 2024.

\bibitem[DS25]{Dong-Song_2023}
Conghan Dong and Antoine Song.
\newblock Stability of euclidean 3-space for the positive mass theorem.
\newblock {\em Invent. Math.}, 239:287--319, 2025.

\bibitem[GL80]{GL}
Mikhael Gromov and H.~Blaine Lawson, Jr.
\newblock Spin and scalar curvature in the presence of a fundamental group. {I}.
\newblock {\em Ann. of Math. (2)}, 111(2):209--230, 1980.

\bibitem[Gro14]{Billiards}
Misha Gromov.
\newblock Dirac and {P}lateau billiards in domains with corners.
\newblock {\em Cent. Eur. J. Math.}, 12(8):1109--1156, 2014.

\bibitem[Gro19]{Gromov4lectures}
Misha Gromov.
\newblock Four lectures on scalar curvature, 2019.
\newblock \href{https://arxiv.org/abs/1908.10612}{arXiv:1908.10612}.

\bibitem[HI97]{HI}
G.~Huisken and T.~Ilmanen.
\newblock The {R}iemannian {P}enrose inequality.
\newblock {\em Internat. Math. Res. Notices}, 59(20):1045--1058, 1997.

\bibitem[HZ24]{HirschZhang}
Sven Hirsch and Yiyue Zhang.
\newblock Stability of {L}larull's theorem in all dimensions.
\newblock {\em Adv. Math.}, 458:Paper No. 109980, 17, 2024.

\bibitem[KKL24]{KKL}
Demetre~P. Kazaras, Marcus~A. Khuri, and Dan~A. Lee.
\newblock Stability of the positive mass theorem under {R}icci curvature lower bounds.
\newblock {\em Math. Res. Lett.}, 31(3):747--794, 2024.

\bibitem[KSX24]{KSXentropy}
Demetre Kazaras, Antoine Song, and Kai Xu.
\newblock Scalar curvature and volume entropy of hyperbolic 3-manifolds, 2024.
\newblock \href{https://arxiv.org/abs/2312.00138}{arXiv:2312.00138}.

\bibitem[KX23]{KXdrawstring}
Demetre Kazaras and Kai Xu.
\newblock Drawstrings and flexibility in the geroch conjecture, 2023.
\newblock \href{https://arxiv.org/abs/2309.03756}{arxiv:2010.15663}.

\bibitem[Lee19]{geometric}
Dan~A. Lee.
\newblock {\em Geometric relativity}, volume 201 of {\em Graduate Studies in Mathematics}.
\newblock American Mathematical Society, Providence, RI, 2019.

\bibitem[Lla98]{Llarull}
Marcelo Llarull.
\newblock Sharp estimates and the {D}irac operator.
\newblock {\em Math. Ann.}, 310(1):55--71, 1998.

\bibitem[LNN23]{LNN}
Man-Chun Lee, Aaron Naber, and Robin Neumayer.
\newblock $d_p$ convergence and $\epsilon$-regularity theorems for entropy and scalar curvature lower bounds.
\newblock {\em Geom. Topol.}, 27(1):227--350, 2023.

\bibitem[LS14]{LeeSormanispherical}
Dan~A. Lee and Christina Sormani.
\newblock Stability of the positive mass theorem for rotationally symmetric {R}iemannian manifolds.
\newblock {\em J. Reine Angew. Math.}, 686:187--220, 2014.

\bibitem[LT22]{Lee-Topping}
Man-Chun Lee and Peter~M. Topping.
\newblock Metric limits of manifolds with positive scalar curvature, 2022.
\newblock \href{https://arxiv.org/abs/2203.01223}{arXiv:2203.01223}.

\bibitem[Sor23]{SormaniConj}
Christina Sormani.
\newblock Conjectures on convergence and scalar curvature.
\newblock In {\em Perspectives in scalar curvature. {V}ol. 2}, pages 645--722. World Sci. Publ., Hackensack, NJ, 2023.

\bibitem[STW24]{STW23}
Christina Sormani, Wenchuan Tian, and Changliang Wang.
\newblock An extreme limit with nonnegative scalar.
\newblock {\em Nonlinear Anal.}, 239:Paper No. 113427, 24, 2024.

\bibitem[SW11]{SorWen}
Christina Sormani and Stefan Wenger.
\newblock The intrinsic flat distance between {R}iemannian manifolds and other integral current spaces.
\newblock {\em J. Differential Geom.}, 87(1):117--199, 2011.

\bibitem[SY79a]{SY}
R.~Schoen and S.~T. Yau.
\newblock On the structure of manifolds with positive scalar curvature.
\newblock {\em Manuscripta Math.}, 28(1-3):159--183, 1979.

\bibitem[SY79b]{SYPMT}
Richard Schoen and Shing~Tung Yau.
\newblock On the proof of the positive mass conjecture in general relativity.
\newblock {\em Comm. Math. Phys.}, 65(1):45--76, 1979.

\bibitem[Wit81]{Witten}
Edward Witten.
\newblock A new proof of the positive energy theorem.
\newblock {\em Comm. Math. Phys.}, 80(3):381--402, 1981.

\bibitem[Zho23]{Zhou_2023}
Shengxuan Zhou.
\newblock {On the Gromov-Hausdorff limits of Tori with Ricci conditions}, 2023.
\newblock \href{https://arxiv.org/abs/2309.10997}{arxiv:2309.10997}.

\end{thebibliography}

\vspace{24pt}

\noindent\textit{Department of Mathematics, Duke University, Durham NC, USA}


\noindent\textit{Email: \href{mailto:kx35@math.duke.edu}{kx35@math.duke.edu}}

\vspace{6pt}

\noindent\textit{Department of Mathematics, Michigan State University, East Lansing MI, USA}


\noindent\textit{Email: \href{mailto:kazarasd@msu.edu}{kazarasd@msu.edu}}

\end{document}